\documentclass[12pt]{article}

\usepackage{version}

\usepackage[utf8]{inputenc}
\usepackage[T1]{fontenc}
\usepackage{babel}
\usepackage[babel]{csquotes}

\usepackage{amsmath}
\usepackage{amsfonts}
\usepackage{amsthm}
\usepackage{bbm}
\usepackage{amssymb}
\usepackage[top=3cm, bottom=3cm, left=3.3 cm, right=3.3 cm]{geometry}
\usepackage{mathrsfs}
\usepackage{stmaryrd}
\usepackage{hyperref}

\usepackage{xcolor}  

\hypersetup{
    colorlinks=true,        
    linkcolor=blue,         
    citecolor=magenta,        
    filecolor=green,      
    urlcolor=cyan           
}

\usepackage{cleveref}
\usepackage{lipsum}

\usepackage{thmtools}%
\usepackage{thm-restate}

\usepackage{float}
\usepackage{tikz}
\usetikzlibrary{shapes.misc}
\tikzset{cross/.style={cross out, draw, 
minimum size=2*(#1-\pgflinewidth), 
inner sep=0pt, outer sep=0pt}}
\usepackage{IEEEtrantools}
\usepackage{url}

\usepackage{enumitem}
\usepackage{accents}

\newtheorem{theorem}{Theorem}[section]

\newtheorem{proposition}[theorem]{Proposition}
\newtheorem{fact}[theorem]{Fact}
\newtheorem{lemme}[theorem]{Lemma}

\newtheorem{corollary}[theorem]{Corollary}
\newtheorem*{th*}{Theorem}
\newtheorem*{fact*}{Fact}
\newtheorem*{lemma*}{Lemma}
\theoremstyle{definition}

\newtheorem{example}{Example}

\mathcode`l="8000
\begingroup
\makeatletter
\lccode`\~=`\l
\DeclareMathSymbol{\lsb@l}{\mathalpha}{letters}{`l}
\lowercase{\gdef~{\ifnum\the\mathgroup=\m@ne \ell \else \lsb@l \fi}}
\endgroup

\DeclareMathOperator{\SL}{\mathrm{SL}}

\DeclareMathOperator{\Ad}{Ad}
\DeclareMathOperator{\ad}{ad}

\DeclareMathOperator{\Cov}{\mathbb{C}ov}

\newcommand{\R}{\mathbb{R}}
\newcommand{\C}{\mathbb{C}}
\newcommand{\Z}{\mathbb{Z}}
\newcommand{\N}{\mathbb{N}}

\newcommand{\E}{\mathbb{E}}
\newcommand{\1}{\mathbbm{1}}

\newcommand{\kF}{\mathcal{F}}

\newcommand{\DN}{\mathcal{D}_{N}}

\newcommand{\kl}{\mathfrak{l}}

\newcommand{\kh}{\mathfrak{h}}

\newcommand{\km}{\mathfrak{m}}

\newcommand{\kr}{\mathfrak{r}}
\newcommand{\kg}{\mathfrak{g}}
\newcommand{\g}{\mathfrak{g}}

\newcommand{\ww}{\underline{w}}

\newcommand{\oalpha}{\overline{\alpha}}
\newcommand{\W}{\Wmega}

\newcommand{\xx}{\underline{x}}
\newcommand{\yy}{\underline{y}}

\newcommand{\uu}{\underline{u}}
\newcommand{\rr}{\underline{r}}

\newcommand{\dkg}{\widehat{\kg}}
\newcommand{\tkg}{\widetilde{\kg}}
\newcommand{\tmu}{\widetilde{\mu}}
\newcommand{\teta}{\widetilde{\eta}}

\newcommand{\htmuNchi}{\widehat{\tmu^{*N}_{-N\chi}}}

\newcommand{\tkm}{\widetilde{\mathfrak{m}}}

\newcommand{\aalpha}{\underline{\alpha}}
\newcommand{\ug}{\underline{g}}

\newcommand{\hmu}{\widehat{\mu}}

\newcommand{\hf}{\widehat{f}}

\newcommand{\eps}{\varepsilon}
\newcommand{\supp}{\text{supp}}
\newcommand{\LL}{\mathscr{L}}
\renewcommand{\kl}{\mathfrak{l}}

\newcommand{\nn}{\underline{n}}

\renewcommand{\SS}{\mathscr{S}}
\newcommand{\ot}{\overline{t}}
\newcommand{\YY}{\mathcal{Y}}
\newcommand{\oSS}{\mathring{\SS}}
\newcommand{\Pc}{\mathscr{P}}
\newcommand{\tX}{\widetilde{X}}

\newcommand{\ki}{\mathfrak{I}}
\newcommand{\XXab}{\overline{X}_{\mu}}
\newcommand{\XX}{X_{\mu}}
\newcommand{\XXkl}{\overline{X}_{\kl}}

\newcommand{\dd}{d_{\Xab}}
\newcommand{\DilN}{D_{\sqrt{N}}}
\newcommand{\DilsN}{D_{\frac{1}{\sqrt{N}}}}
\newcommand{\Xm}{X}
\newcommand{\Xab}{\overline{X}}

\renewcommand{\W}{W}
\renewcommand{\Z}{Z}

\newcommand{\normtp}[1]{{|#1|'_{\tkg}} } 
\newcommand{\normp}[1]{{|#1|'_{\kg}} }  
\newcommand{\norm}[1]{|#1|_{\kg} }  

\newcounter{namedthm}



\newcommand{\keywordsMSC}[1]{\renewcommand{\thefootnote}{}\footnotetext{#1}\renewcommand{\thefootnote}{\arabic{footnote}}}

\newcommand{\keywordsi}[1]{\renewcommand{\thefootnote}{}\footnotetext{#1}\renewcommand{\thefootnote}{\arabic{footnote}}}

\begin{document}

\title{The central limit theorem on nilpotent Lie groups}
\author{Timothée Bénard and Emmanuel Breuillard}

\maketitle

\begin{abstract}
We formulate and establish the central limit theorem for products of i.i.d. random variables on arbitrary simply connected nilpotent Lie groups, allowing a possible bias. We find that some interesting new phenomena arise in the presence of a bias: the walk spreads out at a higher rate in the ambient group, while the limiting hypoelliptic diffusion process may not always have full support. We use elementary Fourier analysis to establish our results which include Berry-Esseen bounds under optimal moment assumptions, as well as an analogue of Donsker's invariance principle. Various examples of nilpotent Lie groups are treated in detail showing the variety of different behaviours. We also obtain a characterization of when the limiting distribution is an ordinary gaussian and answer a question of Tutubalin regarding asymptotically close distributions on nilpotent Lie groups.
\end{abstract}

\keywordsMSC{60B15, 60F05, 35H10, 35R03, 22E25, 58J65}
\keywordsi{Nilpotent Lie groups, biased random walks, central limit theorem, hypoelliptic diffusion process}


\tableofcontents


\section{Introduction}



Given a Lie group $G$ and a probability measure $\mu$ on $G$, it is a natural problem to investigate the large  scale geometry of the convolution powers $\mu^{*N}$ for large $N$, or, in other words, the distribution of the product $X_{1}\cdots X_{N}$  of $N$ independent random variables in $G$ with law $\mu$. We refer the reader to the survey \cite{saloff-coste01} for a pleasant introduction to the topic, as well as to the books \cite{grenander63, varopoulos-saloff-coulhon92, liao-book, benoist-quint-book, varopoulos21}.



In the abelian case $G=\R^d$, the answer is provided by  the classical central limit theorem: provided $\E(X_{1}^2)<\infty$, we have the convergence in law:
$$\frac{1}{\sqrt{N}}\left( X_{1}+\dots +X_{N} -N\chi \right) \underset{N\to +\infty}{\longrightarrow} \mathscr{N}(0, \sigma^2),$$ where $\chi$ and $\sigma^2$ denote the expectation and covariance matrix of $X_{1}$ respectively, and $\mathcal{N}(0,\sigma^2)$ is the centered Gaussian distribution on $\R^d$ with covariance matrix $\sigma^2$. 


What is meant by a central limit theorem in a general Lie group $G$ is less clear, because there rarely is a canonical way to renormalize the random walk. See \cite{grenander63}, \cite[p. 3]{tutubalin-sazonov66}, \cite{virtser74}, \cite[Chapter 1] {breuillard-thesis04} for a discussion.  However, if the Lie group $G$ is nilpotent and simply connected, then $G$ is diffeomorphic to its Lie algebra $\kg$ via the exponential map  \cite{corwin-greenleaf90} and, although there may not exist any dilating automorphisms \cite{dyer69, dixmier-lister57}, there are natural ways to renormalize the space, leading to meaningful limit theorems.

In this paper we shall be interested in  the central limit theorem for i.i.d. random walks on an arbitrary simply connected nilpotent Lie group $G$. These are the simply connected real Lie groups $G$ whose central descending series $G^{[i+1]}:=[G,G^{[i]}]$ terminates at the trivial group in a finite number of steps. They can alternatively be described as the connected subgroups of unipotent upper triangular matrices \cite[Theorem 1.1.1]{corwin-greenleaf90}. 

\bigskip

 In his early work on the subject,  Tutubalin  \cite{tutubalin64} handled successfully the case where $G$ is the Heisenberg group. It is the  simplest instance of a non-commutative nilpotent Lie group.

\begin{example}[Heisenberg group] \label{heis-ex} Let $G$ be the classical $3$-dimensional Heisenberg group, defined as $\R^3=\oplus_{i=1}^3\R e_{i}$ endowed with the group structure 
$$x*y= (x_{1}+y_{1}, \,x_{2}+y_{2}, \,x_{3}+y_{3}+\frac{1}{2}(x_{1}y_{2}-x_{2}y_{1})).$$ Let  $\mu$ be a probability measure on $G$ satisfying the moment conditions $\E_\mu(|x_i|^2)<\infty$ for $i=1,2$ and $\E_\mu(|x_3|)<\infty$. The mean of $\mu$ in the abelianization is identified with $\XX:=\E_\mu(x_1e_1+x_2e_2)$. Let $r>0$. If $\XX=0$, set $D_r:x \mapsto (x_1 r, x_2 r, x_3 r^2)$, while if $\XX \neq 0$ set $D_r: x\mapsto (x_1 r, x_2 r, x_3 r^{3})$.    

\noindent In the centered case $\XX=0$, Tutubalin \cite{tutubalin64} proves that the renormalized distribution $\DilsN(\mu^{*N})$ converges to a probability measure on $G$, which is absolutely continuous with respect to Lebesgue. In fact, it is the law of the value at time $t=1$ of a hypoelliptic diffusion $(B^{(1)}_t,B^{(2)}_t,L_t)$ on $G$, where $(B^{(1)}_t,B^{(2)}_t)$ is a planar Brownian motion and $L_t$ is its L\'evy area (see e.g. \cite{friz-victoir10}). 

\noindent On the other hand, if $\XX \neq 0$, Tutubalin \cite{tutubalin64} shows that the recentered and renormalized distribution $\DilsN(\mu^{*N} *\delta_{-N\XX})$, i.e. the distribution of $\DilsN(x *(-N\XX))$ when $x\sim \mu^{*N}$, converges towards an ordinary gaussian (normal) distribution on $\R^3$. Notice that $\det D_r=r^{3}$ in the centered case, while  $\det D_r=r^{5}$ in the non-centered case. This means that \emph{the random walk ``spreads out'' more in space in the presence of a non-trivial bias.}
\end{example}

For an arbitrary simply connected nilpotent Lie group $G$,  the central limit theorem (CLT) for \emph{centered} random walks, i.e. for measures $\mu$ whose projection to the abelianization $G/[G,G]$ has zero mean, is by now a classical result and several proofs are known. The case where $G$ is stratified follows easily from Wehn's infinitesimal CLT  \cite{wehn62} or from Stroock-Varadhan \cite{stroock-varadhan73}, while the general case is handled by Cr\'epel and Raugi in \cite{crepel-raugi78}. See also \cite{caramellinoetal} for a more recent treatment of the centered CLT for stratified nilpotent Lie groups. Berry-Esseen estimates have also been derived by Pap \cite{pap91} and Bentkus-Pap \cite{bentkus-pap96} in the stratified centered case.  Last but not least, in the case where the driving measure $\mu$ has a continuous density, Alexopoulos \cite{alexopoulos02-2, alexopoulos02-3} obtained very complete results for the centered CLT on an arbitrary nilpotent Lie group including Berry-Esseen bounds and even a full Edgeworth expansion.

Our main goal in this paper will be to formulate and establish the general case of the CLT and the Berry-Esseen estimates, see Theorems \ref{CLT} \ref{BE}, allowing the group to be non-stratified and most importantly the walk to be \emph{non-centered} (we also say ``biased''). We shall make no assumption of regularity  on the measure $\mu$, only necessary moment assumptions. We will also prove an analogue of Donsker's invariance principle in our context (\Cref{Donsker-nilpotent}), establishing that the stochastic process obtained by interpolating the random walk does converge in law towards a limiting diffusion process in the topology of H\"older continuous paths (and even in the ``rough paths'' topology), see \Cref{donsker-sec}.

Unlike the case of commutative groups, \emph{the analysis of non-centered walks does not reduce to that of centered walks}. 
 It is true that certain questions for non-centered walks can be solved by reducing to the centered case, for example the behavior of $\mu^{*n}$ in a neighborhood of the origin:  as Alexopoulos observes in \cite[1.11]{alexopoulos02-2},  $d\mu^{*n}(x)= e^{-\beta_{\mu}n} \chi_{\mu}(x) d\mu_{0}^{*n}(x)$ where $\mu_{0}$ is a centered probability measure, $\beta_{\mu}>0$, and $ \chi_{\mu}$ is a well chosen character of the group. This however  yields no information on the kind of questions we are considering in this paper, namely the behavior of $\mu^{*n}$ around its average, where most of the mass is located. A delicate issue in establishing a central limit theorem for biased walks is to understand precisely how to renormalize the random walk after recentering, in order to obtain an absolutely continuous probability measure in the limit. The large scale geometry of the distribution $\mu^{*N}$ depends in a subtle way on the increment average and attempts to find a suitable renormalization have been made by  Virtser \cite{virtser74} for the full upper-triangular unipotent matrix group, then by Raugi \cite{raugi78} for general nilpotent Lie groups. However neither formulation is fully satisfactory: Virtser's renormalization is too strong, leading in some cases to non-absolutely continuous distributions in the limit, whereas Raugi's renormalization may lead to full escape of mass in the limit (see the end of \Cref{Sec-CLT+BE} for a detailed discussion). We shall show the right way to renormalize non-centered random walks on nilpotent Lie groups, so that the limiting distribution is always a smooth probability measure.

 We also present  a detailed analysis of the limiting measure.  In particular, we give decay estimates at infinity (see \Cref{density-reg})  proving that the limit density is a rapidly decreasing Schwartz function similarly to what is known in the centered case \cite{hebisch89, varopoulos-JFA-88, siebert84}. It came as a surprise to us to discover that a new phenomenon takes place   in the biased case, and only in step $3$ or more: due to certain inaccessibility constraints \cite{sussmann87, agrachev-sachkov04, sachkov07} appearing in the presence of a bias, \emph{the limiting measure may not have full support} in general and the identity can lie on the boundary of the support, even though it is absolutely continuous. This phenomenon occurs in all free nilpotent Lie groups of step at least $3$ whenever there is a non-zero bias (see \Cref{never-full-thm}). On the other hand there are Lie groups, such as the upper triangular unipotent groups, where it never happens and the limit measure always has full support (see \Cref{full-lower-bound-thm}). For such groups we will also establish a lower bound on the density, by exploiting a version of Harnack's inequality for hypoelliptic operators recently established in \cite{kogoj-polidoro16}. Yet there are also examples, such as filiform Lie algebras of dimension at least $4$, where both phenomena  occur depending on the increment average. Typically the limiting measure is far from being an ordinary Gaussian distribution, but we shall also characterize exactly when this is the case (\Cref{charact-gaussian}).

In the context of hypoelliptic heat kernel estimates and diffusion processes on manifolds, the consequences of a non-trivial drift for the support and the small time estimates have been studied in several influential papers, starting with those of Ben Arous and L\'eandre \cite{benarous-leandre} (see also \cite{aidaetal} and more recently \cite{colin-hillairet-trelat}). In these papers the drift refers to a first order term in the infinitesimal generator of the diffusion. In our context, the non-centeredness assumption on the driving measure is not reflected by a first order term in the infinitesimal generator of the limiting diffusion. Rather, it forces the generator to be time dependent. This time dependency plays a crucial role here, especially as the generator obtained for a frozen time t is not hypoelliptic. Below, we will present examples of limiting measures with non-full support and also no drift term at all in the (hypoelliptic) infinitesimal generator.

We note that our central limit theorem also applies to random walks on \emph{finitely generated} (discrete) nilpotent groups. Indeed a torsion-free finitely generated nilpotent group can always be realized as a co-compact discrete subgroup of its Malcev closure, which is a connected and simply connected nilpotent Lie group \cite{corwin-greenleaf90} (the torsion subgroup is a finite normal subgroup, which does not play a role in the asymptotic behavior of the random walk).  We refer the reader to \cite{alexopoulos02, breuillard10, ishiwata03, ishiwata-kawabi-namba-I, ishiwata-kawabi-namba-II, diaconis-hough21} for related articles  where limit theorems for random walks on finitely generated nilpotent groups as well as on nilpotent covers of compact manifolds are obtained.

Contrary to previous work on the nilpotent CLT, which relied on martingale techniques \cite{stroock-varadhan73, virtser74, caramellinoetal} or convergence theorems for semigroups of contracting operators \cite{wehn62, crepel-raugi78, raugi78}, our proof  follows a rather down-to-earth Fourier analytic approach, which allows for greater flexibility and provides directly information on the rate of convergence.  Similar techniques are employed by the second named author \cite{breuillard05} and by Diaconis and Hough \cite{diaconis-hough21, hough19} in their work on the centered local limit theorem. Similarly as in Lindeberg's proof of the classical CLT on $\R$ (see \cite{feller}), we shall use a  \emph{replacement scheme}, where at each step one inserts a new random variable distributed according to the limit measure. However,  we will first need to replace the Lie group product $*$ by an associated graded product $*'$, for which the renormalization acts by automorphisms. Coordinatewise, we are led to consider random variables of the form:
\begin{equation}
\sum_{n_1<\ldots<n_t \leq N} M(X_{n_1},\ldots,X_{n_t})\end{equation}
where $M$ is a fixed polynomial function on $\kg$, and the $X_{i}$'s represent i.i.d. variables with law $\mu$. Lindeberg's replacement scheme for polynomials has also been exploited recently in a different context in works on noise stability and influences \cite{mossel-odonnell10, chatterjee06}, see also \cite{rotar75}.


\bigskip

We now pass to the more precise description of our results.

\subsection{Main results}

Let $G$ be a simply connected nilpotent Lie group with Lie algebra $\kg$. We identify $G$ with $\kg$ via the exponential map.
Let $s$ be the step (or nilpotency class) of $\kg$, that is the smallest integer $s$ such that $\kg^{[s+1]}=\{0\}$, where $\kg^{[i]}:=[\kg,\kg^{[i-1]}]$ is the central descending series. We denote by $*$ the Lie product on $G$, which we view as being defined on $\kg$. By abuse of notation we also denote by $*$ the associated convolution product of measures on $\kg$.  Note that $*$ is given by a polynomial map on $\kg$ as follows from the Campbell-Hausdorff formula \cite{dynkin47}:
$$x*y = x+y + \frac{1}{2}[x,y] + \frac{1}{12}[x,[x,y]] +\frac{1}{12}[y,[y,x]]+\dots $$
 We  consider a random walk $S_n=X_1* \ldots * X_n$ associated to a probability measure $\mu$ on $(\kg,*)$. We let $\mu_{ab}$ be the projection of $\mu$ onto the abelianization $\kg_{ab}:=\kg/[\kg,\kg]$ and $\XXab=\E(\mu_{ab}) \in \kg_{ab}$ the mean. 

It turns out that the large scale behaviour of $S_n$ on the Lie group, is captured by a certain filtration of $\kg$, namely a certain decreasing sequence of ideals, whose definition depends only on $\XXab \in \kg_{ab}$,
$$\kg^{(1)} \supseteq \kg^{(2)} \supseteq \cdots  \supseteq \kg^{(t+1)} = \{0\}$$
such that $[\kg^{(i)},\kg^{(j)}]\subseteq \kg^{(i+j)}$ for all $i,j$.

In {\bf the centered case}, i.e. when $\E(\mu_{ab})=0$, the filtration is simply the central descending series, i.e. $\kg^{(i)}=\kg^{[i]}$ and $t=s$. The projection of the walk to $\kg_{ab}$ exhibits diffusive behavior and will therefore be of order $\sqrt{n}$ as $n$ grows. On the other hand the coordinates of $S_n$ lying deeper in the filtration, i.e. in say $\kg^{[i]}\smallsetminus \kg^{[i+1]}$, grow faster at a rate $n^{i/2}$. Therefore if we renormalize the walk accordingly we may hope that it converges towards a non-degenerate random variable on $\kg$, analogous say to the Gaussian law in the abelian case. 
For a precise formulation, let us fix supplementary vector subspaces $\km^{(i)}$ such that $\kg^{(i)}= \km^{(i)} \oplus \kg^{(i+1)}$. This induces  a direct sum decomposition $\kg=\km^{(1)}\oplus \dots \oplus \km^{(s)}$,  whose  coordinate projections we denote by
 \begin{align*}
 \pi^{(i)}: \kg &\to \km^{(i)}\\ x &\mapsto \pi^{(i)}(x)=x^{(i)}.
 \end{align*} 
Let us define a one-parameter family of \emph{dilations} $(D_{r})_{r>0}$ by the formula
  \begin{equation} \label{dilation} 
  D_r:\kg \to \kg, \,\,\,D_{r}(x)= \sum_{i \ge 1} r^{i}x^{(i)}.
  \end{equation} 
Given $r>0$ and a Borel measure $\eta$ on $\kg$, we write $D_{r}(\eta)$ the push-forward measure of $\eta$ by $D_{r}$, it is characterized by the relation $D_{r}(\eta)(f)=\eta(f\circ D_{r})$ for every measurable function $f :\kg\rightarrow \R^+$. 
  The centered CLT for the $\mu$-walk on $(\kg, *)$ claims the  convergence in law
\begin{equation}\label{TCL}\DilsN(\mu^{*N}) \to \nu \end{equation}
where $\nu$ is a certain smooth probability measure on $\kg$. It also involves a description of $\nu$ as the value at time  $t=1$ of a certain explicit one-parameter semigroup of probability measures $(\sigma_{t})_{t\in \R^+}$ on $\kg$. However, one must be careful that the semigroup structure for $\sigma_{t}$ refers to a new Lie group product $*'$ on $\kg$,   inherited from the choice of subspaces $\km^{(i)}$ and defined by 
\begin{equation}\label{*'-product}a*'b=\lim_{t \to +\infty} D_{1/t}(D_t(a)*D_t(b)).\end{equation} 
Equivalently, $*'$ is the Lie group structure induced by the new Lie bracket $[x,y]'$ on $\kg$,  defined on $\km^{(i)}\times \km^{(j)}$ as the projection of $[x,y]$ onto $\km^{(i+j)}$ parallel to the other $\km^{(k)}$. With this new Lie structure, $(\kg,*')$ is again an $s$-step nilpotent Lie group, and the dilations $D_r$ become automorphisms. The semigroup  $(\sigma_{t})_{t>0}$ can now be defined (as in   \cite{hunt56}) by way of its infinitesimal generator
\begin{equation}
L_\mu  = \frac{1}{2}\sum_{i \leq \dim \kg_{ab}} E^2_{i}  + B 
\end{equation}
where $E_{i}, B$ denote left invariant vectors fields on $(\kg, *')$, with the $E_i$'s  forming a basis of $\km^{(1)}$ in which the covariance matrix of $\mu_{ab}$ is the identity, while $B \in \km^{(2)}$ is $\E_\mu(x^{(2)})$. The  measure $\nu=\sigma_{1}$ admits a smooth and exponentially decaying density, which may  be characterized alternatively as the value at time $t=1$ of  the fundamental solution (heat kernel) $u : \R_{>0}\times \kg\rightarrow \R^+$ to the PDE 
\begin{equation}
    \partial_t - \frac{1}{2}\sum_{i \leq \dim \kg_{ab}} E^2_{i}  + B=0.
\end{equation}
Here, we identify  a vector field on $\kg$ with the corresponding Lie derivative.

A proof of the centered central limit theorem \eqref{TCL} can be found in the work of Tutubalin \cite{tutubalin64} in the case of the Heisenberg group and  Crepel and Raugi \cite{crepel-raugi78} in general, with roots in the early work of Wehn \cite{wehn62} who proved an infinitesimal CLT on general Lie groups, see also \cite{grenander63, stroock-varadhan73, breuillard-thesis04}. Note that the limit law $\nu$ is entirely determined by the covariance matrix of $\mu_{ab}=\pi^{(1)}(\mu)$ and the mean of $\pi^{(2)}(\mu)$, where $\pi^{(i)}(\mu)$ denotes the push-forward of $\mu$ by $\pi^{(i)}$. The convergence \eqref{TCL} holds assuming the existence of a finite second moment, i.e. $d(0,x) \in L^2(\mu)$ for some (any) left invariant geodesic distance $d$ on $(\kg, *)$ or $(\kg, *')$.  Equivalently,  $\|x^{(i)}\|^{1/i} \in L^2(\mu)$ for each $i$.

\begin{example}Let $G$ be the $3$-dimensional Heisenberg group, so $$\kg=\{ xe_1+ye_2+ze_3 \,|\, [e_1,e_2]=e_3 , [e_1,e_3]=[e_2,e_3]=0\},$$ and assume that $\mu_{ab}$ is normalized, that is $\E(x^2)=\E(y^2)=1$, $\E(xy)=\E(x)=\E(y)=\E(z)=0$. Then $\nu$ coincides with the distribution of the random variable $Y=(x(1),y(1),z(1))$, where $(x(t),y(t))$ is a standard Brownian motion on $\R^2$ and $z(t)$ is the L\'evy area (signed area between the cord and the brownian path). The density of $\nu$ has been computed by Paul L\'evy \cite{levy51} (see also \cite[Th\'eor\`eme 1]{gaveau77}): 
$$d\nu(x,y,z)=\frac{dx dy dz }{2\pi^2} \int_\R \cos(2\xi z) \frac{\xi}{\sinh{\xi}} e^{-\frac{1}{2}(x^2+y^2)\xi/\tanh(\xi)} d\xi.$$
\end{example}

For centered walks on general nilpotent Lie groups there is in general no known closed formula for the limit density $u(1,x)$ of $\nu=u(1,x)dx$, but so-called \emph{Gaussian estimates} that control its decay at infinity are available \cite{hebisch89, varopoulos-JFA-88, varopoulos-saloff-coulhon92}. Namely there is $C>0$ such that for all $x \in \kg$:
\begin{equation}
    \frac{1}{C} \exp(- C d(0,x)^2) \leq u(1,x) \leq C  \exp(-d(0,x)^2/C)
\end{equation} with a  similar upper bound  for all partial derivatives of $u(1,.)$. See also \cite{asaad-gordina16} where certain explicit formulas for $u(1,x)$ are derived in the case of filiform Lie algebras. 
\bigskip

Now let us turn to {\bf the non-centered case}, i.e. when $\XXab:=\E(\mu_{ab})$ is possibly non-zero. Then, as in the abelian case, the random walk $S_n$ needs to be translated by $-nx$ for some representative $x$ of $\XXab$ in $\kg$, so that its projection to $\kg_{ab}$ has zero mean and we can then attempt to renormalize the resulting random variable. It turns out that the right filtration to be considered for the renormalization is finer and that it depends on the way $[x,.]$ acts on $\kg$.  Indeed, we are looking at the random product:
\begin{equation} X_1* \ldots * X_n * (-nx) = Y_1 * Y_2^{x} * \ldots * Y_n^{(n-1)x}
\end{equation}
where $Y_i :=X_i* x^{-1}$ and $Y^{g} :=g*Y*g^{-1}$. This is now a product of independent centered variables, though they are no longer identically distributed, due to the iterated adjoint action of $x$.

 The drift of the random walk is linear in the $x$ direction, while it only moves diffusively, at rate $n^{1/2}$, in the other directions of $\kg/[\kg,\kg]$, so in effect $x$ behaves as a bracket of order $2$ even though it is not in the commutator subgroup. In order to take this into account, we introduce a new filtration as follows. For $i\leq 2s-1$, we let $\kg^{(i)}$ be the subspace generated by brackets of length $a$ in which $x$ appears at least $b$ times, where $a+b\geq i$. For $i=2$ however we set $\kg^{(2)}=[\kg,\kg]$ rather than $[\kg,\kg]+\R x$. Equivalently, $\kg^{(i)}$ is defined recursively by setting $\kg^{(i+1)}=[\kg,\kg^{(i)}]+[x,\kg^{(i-1)}]$ for $i\ge 1$, while $\kg^{(0)}=\kg^{(1)}=\kg$.
 
 We thus obtain a nested sequence of ideals: 
$$\underbrace{\kg^{(1)}}_{\kg}\supseteq \underbrace{\kg^{(2)}}_{[\kg, \kg]}\supseteq \underbrace{\kg^{(3)}}_{ [\kg, x]+ [[\kg,\kg,] \kg]}\supseteq \dots \supseteq \kg^{(2s)}=0.$$
We call the filtration $(\kg^{(i)})_{i \leq 2s-1}$ the \emph{weight filtration} of $\kg$ associated to $\XXab$. Indeed , note that it does not depend on the choice of representative $x\in \kg$ of $\XXab \in \kg_{ab}$. Furthermore $\kg^{[i]} \subseteq \kg^{(i)}$ for each $i$ with equality if $\XXab = 0$. We then choose supplementary subspaces $\km^{(i)}$ as before so that $\kg^{(i)}= \km^{(i)} \oplus \kg^{(i+1)}$ and denote by  $\XX$   the representative in $\km^{(1)}$:
\begin{equation} \XXab=\XX + [\kg,\kg], \,\,\, \XX \in \km^{(1)}.
\end{equation}  
We keep the notation $y^{(i)}=\pi^{(i)}(y)$ to denote the projection of $y \in \kg$ onto $\km^{(i)}$. We also define the dilations $D_r$ as in \eqref{dilation} above, and the $*'$ Lie product as in \eqref{*'-product}. We warn the reader that, unlike in the centered case, $\km^{(1)}$ may no longer be bracket-generating in $(\kg,*')$ as happens already in the Heisenberg group case. Finally we define the nilpotent operator $a_{\XX}:\kg \to \kg$ by setting \begin{equation}\label{axdef}a_{\XX}(y)=\pi^{(i+2)}([\XX,y]) \textnormal{ when }y \in \km^{(i)}.\end{equation}

Given a probability measure $\mu$ on $G$, we say that it has a finite moment of order $m\in (0,+\infty)$ for the weight filtration if $y \mapsto \|y^{(i)}\|^{m/i}$ is in $L^1(\mu)$ for each $i$ (where we take some norm, say Euclidean, on each $\km^{(i)}$). Since $\kg^{[i]} \subseteq \kg^{(i)}$ and the inclusion can be strict when $\XXab \neq 0$, this is a slightly weaker constraint on $\mu$ than the ordinary notion of moment on nilpotent Lie groups.

\begin{theorem}[Non-centered CLT]\label{CLT} Let $G$ be a simply connected nilpotent Lie group with Lie algebra $\kg$. Let $\mu$ be a probability measure on $G$ with finite second moment for the weight filtration induced by $\XXab:=\E(\mu_{ab})\in \kg_{ab}$. Assume that the support of $\mu_{ab}$ is not contained in a proper affine subspace of $\kg_{ab}$. Then \begin{equation}\label{tcl-conv}
\DilsN(\mu^{*N}*\delta_{-N\XX})  \underset{N\to +\infty}{\longrightarrow} \nu\end{equation}
where $\nu$ is a smooth probability measure on $\kg$. More precisely,  writing $\nu_t:=D_{\sqrt{t}} (\nu)$, we have $\nu_t=u(t,x)dx$ where the density $u(t,x)$ is the (smooth) fundamental solution to the  following hypoelliptic time-dependent PDE:
\begin{equation}\label{heatPDE}\partial_t v = \frac{1}{2} \sum_{i \leq \dim \kg_{ab}} (\exp(t a_{\XX})E_i)^2 v - (\exp(t a_{\XX})B_{\mu})v.\end{equation}
Here the $E_i$'s form a basis of $\mathfrak{m}^{(1)}$ in which $\pi^{(1)}(\mu)$ has identity covariance matrix. And $B_{\mu}:=\E_\mu(x^{(2)})$ is the ``commutator mean'' of $\mu$.
\end{theorem}

In \eqref{heatPDE}, a vector  $X\in \kg$ is seen as a left invariant vector field on $(\kg,*')$, with the corresponding Lie derivative
$$Xf = \lim_{\eps \to 0} \frac{f(x*'\eps X)-f(x)}{\eps}.$$

\bigskip

\noindent \emph{Remark}. It follows from the theorem that the limit measure $\nu$ depends only on $\XXab=\E(\mu_{ab})$, $\Cov(\mu_{ab})$ the covariance matrix of $\mu_{ab}$,  the mean of $\mu$ in projection to $\kg/\kg^{(3)}$, and the choice of weight decomposition $\kg=\oplus_{i}\km^{(i)}$.  

\bigskip

\noindent \emph{Remark}. The result remains valid verbatim if we replace the translation  $\delta_{-N\XX}$ in \eqref{tcl-conv} by an arbitrary one $\delta_{-N\XX+g_N}$, provided the sequence $g_N\in \kg$ is such that $\|g_N^{(i)}\|=o(N^{i/2})$ for each $i\ge 1$. We may also replace the multiplicative recentering by an additive recentering, up to adjusting the limiting smooth probability measure, see \Cref{CLT-additive}.

\bigskip



We shall also establish upper estimates on the density $\nu=u(1,x)dx$, namely the existence of  constants $C,c>0$ such that
\begin{equation}\label{up-dens-bd}
   u(1,x) \leq C  \exp(-d'(0,x)^c)
\end{equation}
where $d'$ is a left-invariant geodesic (e.g. Riemannian)  metric on $(\kg,*')$, and similar estimates hold for all partial derivatives of $u(1,.)$ (see \Cref{density-reg}), showing that $u(1,.)$ is a Schwartz function on $\kg$.

The question of a lower bound is more subtle. In the centered case,  Gaussian lower estimates have been known for a while \cite{varopoulos-JFA-88, varopoulos-saloff-coulhon92}. However, when one authorizes the walk to be non-centered, the picture changes drastically as the support of the limiting distribution $\nu$ may not even coincide  with  $\kg$.  In this case, the support of $\nu$ is a proper closed subset with non-empty (dense) interior.  Describing the support is a question related to \emph{control theory} on Lie groups (\cite{sachkov07}). We compute the support in an explicit example in the case of the filiform Lie algebra of step $3$ (\Cref{thread-support}). This leads to the following:

\begin{theorem}[Never full support] \label{never-full-thm}
Assume  $G$ is a \emph{free} nilpotent Lie group of step at least $3$. Then for any \emph{non-centered} probability measure $\mu$ on $G$ satisfying the assumptions of \Cref{CLT}, the support of the limiting distribution $\nu$ does not contain the identity in its interior.  
\end{theorem}

At the extreme opposite, we identify a class of nilpotent Lie groups for which every limiting distribution must have full support (\Cref{Sec-upper-triang}), and give a satisfactory lower bound (\Cref{lower-cor}). This class contains the upper triangular unipotent groups
$$Upp_d(\R):=\{a_{ij} \in \SL_d(\R); a_{ii}=1, i>j \Rightarrow a_{ij}=0\},$$ or any nilpotent Lie group of step at most  $2$. It is defined by the \emph{double cancellation condition} ({\bf DC}): there exists a bracket-generating family  $(v_{i})$ of elements in $(\kg, [.,.])$, such that any bracket of the form $[v_{i},[v , [\dots [ v, v_{i}]\dots ]]$ vanishes.

\begin{theorem}[Full support and lower bound] \label{full-lower-bound-thm}
Assume $(\kg, [.,.])$ satisfies \emph{({\bf DC})}. Then for any probability measure $\mu$ on $G$ as in \Cref{CLT}, the support of the limiting measure $\nu$ is all of $\kg$ and its density $u(1,x)$ satisfies:  $\exists \alpha, A>0$, $\forall x\in \kg$, 
$$u(1, x)\geq  \alpha \exp(-A \,d'(0, x)^2).$$
\end{theorem}

In \Cref{Sec-lower-bound} we discuss a general lower bound valid in general even when the  limiting distribution is not of full support. Using  the recent work  \cite{kogoj-polidoro16}, we provide an answer in terms \emph{Harnack chains} whose increments are not too close to the boundary of the support, see \Cref{K-lower-bound}. This analysis allows us to establish that $u(1,x)$ does not vanish on the interior of the support.  As as by-product, our lower bound allows us to characterize  the case where $\nu$ is an ordinary Gaussian distribution (in the ordinary Euclidean sense). The implication i)$\implies$iii) below is due to Crepel-Raugi \cite{crepel-raugi78}. 
 
\begin{theorem}[Characterization of the Gaussian case]\label{charact-gaussian} Keep the assumptions of \Cref{CLT}. The following are equivalent. 
\begin{itemize}
\item[i)] For all  $a\in \{1, \dots, s\}$, we have $$[\XX,\kg^{[a]}]=\kg^{[a+1]}.$$
\item[ii)] The eigenvalues of $\DilsN$ are exactly $\{N^{-(2k-1)/2},\, 1\leq k\leq s  \}$. 
\item[iii)] The limit measure $\nu$ is Gaussian (in the classical Euclidean sense) on $\kg$. 
\end{itemize}
\end{theorem}

\bigskip

In \cite{tutubalin64} (see also \cite[\S 5.4]{tutubalin-sazonov66}) Tutubalin asks for necessary and sufficient conditions for a pair of probability measures $\mu_1$ and $\mu_2$ on a nilpotent Lie group to be \emph{asymptotically close} in the sense that
$$\sup_{A \in \mathcal{C}}|\mu_1^{*n}(A)-\mu_2^{*n}(A)| \underset{n \to +\infty}{\longrightarrow} 0,$$ where $\mathcal{C}$ is a certain class of Borel subsets of $\kg$, which for definiteness we will take to be the class of all convex subsets. Solutions were provided in the case of the   upper triangular  unipotent $n \times n$ matrix group, first by Tutubalin \cite[\S 6]{tutubalin64} for $n=3$, then by Virtser \cite[Theorem 5]{virtser74} for arbitrary $n$ and only certain values of $\XXab$. Thanks to \Cref{CLT} we can answer Tutubalin's question in any simply connected nilpotent Lie group.

\begin{corollary}[Characterization of asymptotically close measures]\label{asymp-close} Let $G$ be a simply connected nilpotent Lie group with Lie algebra $\kg$. Consider two probability measures $\mu_1,\mu_2$ on $G$ with a finite second moment for their respective weight filtrations and non singular covariance matrix in the abelianization. 
The measures $\mu_1$ and $\mu_2$ are asymptotically close if and only if their projections to $\kg/[\kg, \kg]$ have the same mean and covariance matrix, and their projections to $\kg/\kg^{(3)}$ have the same mean.
\end{corollary}

We recall that $\kg^{(3)}=[\kg,[\kg,\kg]] + [\XXab,\kg]$ and $\XXab=\E(\mu_{ab})$. In particular the equivalence classes of aymptotically close measures are determined by a finite number of parameters. This is in sharp contrast with what happens in semisimple Lie groups, see \cite{tutubalin-sazonov66}, \cite{bougerol}.

While previous approaches for proving the CLT were based either on probabilistic methods (martingale convergence \cite{stroock-varadhan73, virtser74}) or functional analytic methods (perturbation of operator semigroups \cite{wehn62, crepel-raugi78, raugi78}), we use characteristic functions. This approach is motivated by \cite{diaconis-hough21, hough19}. This has a pay-off as it allows us to obtain Berry-Esseen bounds, yielding information on the rate of convergence. 
It is formulated in terms of  the  moments of $\mu$. Given any $r>0$ and a choice of norm on $\kg$, the $r$-th moment of $\mu$ (for the weight filtration) is defined by: 
$$m_{r}(\mu):=\sum_{1\leq i\leq 2s-1}\int_{\kg}\|x^{(i)}\|^{r/i} d\mu(x).$$

\begin{theorem}[Berry-Esseen estimate] \label{BE}
Keep the assumptions of \Cref{CLT}. Then for every smooth function of compact support $f$
 on $\kg$ we have for all $N\ge 1$
$$|\DilsN(\mu^{*N} * \delta_{-N \XX}) (f)-\nu(f)| \leq \eps_{N} \max_{|\alpha| \leq \dim \kg +4} \|\partial^\alpha f\|_{L^1}$$
where $\eps_{N}=o_{\mu}(1)$. 
Furthermore, if $\mu$ has a finite third moment  for the weight filtration, then we can take 
$$\eps_{N} \leq C (1+m_{2}(\mu)^{c_s})(1+m_{3}(\mu))N^{-1/2}$$ 
where $s\geq 1$ is the nilpotency class of $\kg$, and $c_s=6s^2$ while $C>0$ is a constant depending only on $(\kg, [.,.])$, a choice of norm and Haar measure, the line $\R\XXab\subseteq \kg_{ab}$, and the weight decomposition $(\km^{(i)})_{i}$.
\end{theorem}
Here, $\partial^\alpha $ refers to partial derivatives of $f$ (in the classical sense on a vector space), and $|\alpha|$ is the order of the derivation.  

\bigskip

Finally, we shall prove that the previous CLT can be upgraded to the convergence of the full interpolated process in the spirit of Donsker's theorem   \cite[Theorem XIII.1.9]{revuz-yor99}. We consider:
$$W^{(N)}(t) := \DilsN \big(X_{1}*\dots *X_{\lfloor tN \rfloor} *(tN-\lfloor tN \rfloor)X_{\lfloor tN \rfloor+1}* - tN \XX\big)$$
where the $X_{i}$'s are i.i.d. with law $\mu$. Note that $W^{(N)}$ defines a random variable with values in the space of locally Lipschitz paths on $\kg$. 

In view of \eqref{heatPDE}, the natural candidate to be the limit of $W^{(N)}$ is the continuous left-invariant diffusion process  on $(\kg, *')$ whose infinitesimal generator at time $t\geq 0$ (see \Cref{Sec-diff-general}) is
\begin{align} 
\mathscr{L}_{t}=\frac{1}{2}\sum_{i=1}^{q}  \left(\exp(ta_{\XX}) E_{i}\right)^2+ \exp(ta_{\XX})B_\mu.
 \end{align}
Denoting by $\sigma=\{\sigma_t\}_{t\in \R^+}$ the law of the diffusion on the space of paths,  \Cref{CLT} establishes the convergence in law of $W^{(N)}(1)$ towards $\sigma_1$. 

The  process associated to $(\mathscr{L}_{t})_t$ will have $\alpha$-H\"older sample paths for any $\alpha\in [0, 1/2)$. We endow the space  $C^{0, \alpha}(\R^+, \kg)$ of $\kg$-valued $\alpha$-H\"older paths with the topology of $\alpha$-H\"older convergence on compact intervals. It is induced by the (separable complete) metric $d_\alpha=\sum_{n \ge 0} 2^{-n}d_{\alpha,n}$, where:
$$d_{\alpha,n}((x(t))_t,(y(t))_t):= \|x(0)-y(0)\|+\sup_{t \neq s \in [0,n]} \frac{\|x(s)^{-1}*x(t)-y(s)^{-1}*y(t)\|}{|t-s|^\alpha}.$$
The $*$ operation in this formula could be replaced with $*'$ or $+$ without changing the induced topology on $C^{0, \alpha}(\R^+, \kg)$. Note that convergence in the $\alpha$-H\"older topology is more and more restrictive as $\alpha$ grows, and implies  uniform convergence on compact intervals (case $\alpha=0$). 

\begin{theorem}[Central limit theorem for processes] \label{Donsker-nilpotent} Keep the assumptions of \Cref{CLT} and assume further that $\mu$ has finite moments of all orders. Then for each $\alpha \in [0,\frac{1}{2})$,  as $N$ tends to $+\infty$, the distribution of $W^{(N)}$ converges to $\sigma$  in the space $C^{0, \alpha}(\R^+, \kg)$ endowed with the topology of $\alpha$-H\"older convergence on compact intervals.
\end{theorem}

Given a fixed $\alpha \in [0,\frac{1}{2})$, the proofs require only a finite moment of order $2m b_{\max}$ for the weight filtration, where $m$ is any integer $m > 1/(1-2\alpha)$, and  $b_{\max}\leq 2s-1$ is the length of the weight filtration $\kg^{(b)}$ defined by $\kg^{(b_{\max})}\supsetneq \kg^{{(b_{\max}+1)}}=\{0\}$.

Such invariance principles have been studied previously in the special case of centered walks. See in particular \cite{caramellinoetal}, which treats the $\alpha=0$ case, and \cite{breuillard-friz-huesmann}, where \Cref{Donsker-nilpotent} is proven for centered walks on the free nilpotent Lie group. We remark in passing that in the centered case the limiting diffusion is not time-dependent and even admits an explicit representation in terms iterated integrals of Brownian motion, see \cite[Theorem 6.1]{kunita} and \cite[Th\'eor\`eme 17]{benarous89}. We shall also briefly discuss convergence in the sense of rough paths.

\bigskip

\noindent \emph{Remark}. In a sequel to this paper \cite{benard-breuillardLLT}, we pursue our study further and establish the local limit theorem for biased walks on arbitrary simply connected nilpotent Lie groups, thus extending the recent work of Diaconis and Hough \cite{diaconis-hough21, hough19} to the biased case.

\subsection{Organization of the paper} In \Cref{Sec-cadre}, we  set up the notations to be used in the rest of the paper. In \Cref{Sec-densite}, we  study the class of diffusion processes on simply connected nilpotent Lie groups that arise as limits of rescaled right random walks with i.i.d. increments.  \Cref{Sec-3reductions} is devoted to various techniques of truncation, graded replacement and  Lindeberg replacement, that help in estimating the Fourier transform of the random walk. In \Cref{Sec-global-thm}, we prove the central limit theorem, the Berry-Esseen estimate, and the analogue of Donsker's invariance principle for nilpotent random walks. We also characterize  pairs of asymptotically close measures.





\section{Set-up and notation} \label{Sec2} \label{Sec-cadre}

In this section we set up the terminology and notations used throughout the paper.

\subsection{ Weight filtration on $\kg$} \label{weight-filtration}
If $\kg$ is a nilpotent Lie algebra and $\Xab \in \kg/[\kg,\kg]$, we define recursively the following nested sequence of ideals: $\kg^{(0)}=\kg^{(1)}=\kg$ and for $i\ge 1$
\begin{equation*}
    \kg^{(i+1)}=[\kg,\kg^{(i)}]+[\Xab,\kg^{(i-1)}].
\end{equation*}
 We note that it  is well-defined (independently of the choice of representative of $\Xab$ in $\kg$), and satisfies  for $i,j \ge 1$, \begin{equation}\label{nestedg}
    [\kg^{(i)},\kg^{(j)}] \subseteq \kg^{(i+j)}.
\end{equation}
We call $(\kg^{(i)})_{i \ge 1}$ the \emph{weight filtration} of $\kg$ associated to $\Xab$. When $\Xab = 0  \mod [\kg,\kg]$, this is the familiar descending central  series, which we denote by $(\kg^{[i]})_{i\ge 1}$. In general we have $\kg^{[i]} \subseteq \kg^{(i)}$, and $\kg^{(2i)}\subseteq \kg^{[i+1]}$. In particular, if $\kg$ has nilpotency class at most $s$, then $\kg^{[s+1]}=\kg^{(2s)}=0$. 
The weight filtration $(\kg^{(i)})_{i \ge 1}$ on $\kg$ allows to define a \emph{graded Lie algebra}:
\begin{equation}
\label{graded-def}
gr_{\Xab}(\kg):= \bigoplus_{i \ge 1} \kg^{(i)}/\kg^{(i+1)}
\end{equation}
whose Lie bracket is inherited from that of $\kg$ using \eqref{nestedg}.

\begin{example}In the case of  the Heisenberg group  $$\kg=\langle e_1,e_2,e_3 \,|\, [e_1,e_2]=e_3\rangle \,\,\,\,\,\, \,\,\,\,\, \text{with} \,\,\,\,\,\, \,\,\,\,\,\Xab=e_{2} \mod [\kg, \kg]$$  we have 
$\kg^{(0)}=\kg^{(1)}=\kg$, $\kg^{(2)}= \{0\}$, $\kg^{(3)}= \R e_{3}$, $(\kg^{(i)})_{i\geq 4}=\{0\}$, and $gr_{\Xab}(\kg)$ is abelian.
\end{example}

\subsection{ The bias extension $\tkg$}


We define a new Lie algebra $\tkg$ by adding to $\kg$ an independent copy $\chi$ of the bias $\Xab \in \kg/[\kg,\kg]$. This extension is a  convenient technical device to dissociate the contributions of the drift incarnated by $\Xab$, which ought to have weight $2$, from those of the fluctuations of the random walk in the direction of $\Xab$, which ought to have weight $1$. Many intermediate results will therefore be formulated in $\tkg$, before being translated back to $\kg$.

If $\Xab=0$, we just set $\tkg=\kg$, $\chi=0$ . Otherwise, we define $\tkg$ as the Lie algebra direct sum:
$$\tkg:=\kg \oplus \R$$
and given a lift  $\Xm$ of $\Xab$ to $\kg$, we set $\chi:=(\Xm,1)$. In this way,    $\tkg= \kg +\R \chi$ where $\chi$ is   linearly independent from $\kg$ and satisfies  $[\chi,x]=[\Xm,x]$ for all $x \in \kg$. We call $\tkg$  the \emph{bias extension} of $\kg$. 

We also extend the weight filtration from $\kg$ to $\tkg$  by setting $\tkg^{(i)}=\kg^{(i)}$ if $i \neq 2$, and $\tkg^{(2)}=\kg^{(2)}+ \R\chi$. It is straightforward to check that this is indeed a filtration, i.e. $[\tkg^{(i)},\tkg^{(j)}] \subseteq \tkg^{(i+j)}$, and is independent of the choice of lift.  The filtration $(\tkg^{(i)})_{i \ge 1}$ determines as in \eqref{graded-def} a graded Lie algebra $gr_{\Xab}(\tkg)$, which we call the \emph{graded bias extension}.

Observe that  $gr_{\Xab}(\tkg)=gr_{\Xab}(\kg)$ if $\Xab=0$. If $\Xab\neq 0$, $gr_{\Xab}(\tkg)$ is isomorphic to the Lie algebra \emph{semi-direct} sum  $$gr_{\Xab}(\kg)\oplus \R,$$ where the $\R$ factor acts on the ideal $gr_{\Xab}(\kg)$ via the nilpotent endomorphisms $a_{\Xm}(t)$ defined on $\kg^{(i)}/\kg^{(i+1)}$ by the formula $$a_{\Xm}(t)(y)=[t\Xm,y] \mod \kg^{(i+3)}.$$ Different choices of representative of $\Xm$ in $\kg$ lead to isomorphic extensions of $gr_{\Xab}(\kg)$.

\bigskip

\noindent \emph{Remark}. The limit measure $\nu_t$ appearing in the CLT (\Cref{CLT}) will arise as the projection of the value at $t$ of a one-parameter semigroup of probability measures $\lambda_t$ defined on $gr_{\Xab}(\tkg)$ and left-invariant for that Lie product. See \Cref{Sec-global-thm}.


\subsection{Grading}

We choose for all $i\leq 2s-1$ a vector subspace $\km^{(i)}$ in $\kg$ such that 
 $$\kg^{(i)} = \km^{(i)}\oplus \kg^{(i+1)}. $$
It follows that $\kg$ is the (vector space) direct sum 
 $$\kg=\km^{(1)} \oplus \dots \oplus \km^{(2s-1)}$$
 and we may identify $\km^{(1)}$ with $\kg/[\kg,\kg]$. From now on, the notation $X$ will indicate the unique element of $\km^{(1)}$ lifting  $\XX$.  We call  $(\km^{(i)})_{i\leq 2s-1}$ a \emph{weight decomposition} of $\kg$. It naturally extends to a weight decomposition  of $\tkg$:
 $$\tkg=\tkm^{(1)} \oplus \dots \oplus \tkm^{(2s-1)}$$
 by setting $\tkm^{(i)}=\km^{(i)}$ if $i \neq 2$ and $\tkm^{(2)}=\km^{(2)}\oplus \R\chi$.
These splittings allow to define  new compatible Lie brackets on $\kg$, $\tkg$ by setting
\begin{equation} \label{bracket-gr}
[x,y]'= \pi^{(i+j)}([x,y])
\end{equation}
if $x \in \tkm^{(i)}, y \in \tkm^{(j)}$ and $\pi^{(k)}: \tkg \to \tkm^{(k)}$ denotes the linear projection modulo the other $\tkm^{(k')}$, $k' \neq k$. Notice that the vector space $\kg$ is an ideal of $(\tkg, [.,.]')$. The   Lie product associated to $[.,.]'$ will be denoted by  $*'$. With this new Lie bracket,  the vector space $\kg$ (resp. $\tkg$) becomes a Lie algebra naturally isomorphic to the graded Lie algebra $gr_{\Xab}(\kg)$  (resp. $gr_{\Xab}(\tkg)$) and the $\km^{(i)}$ (resp. $\tkm^{(i)}$) form a \emph{grading}.

\emph{We  fix a choice of weight grading $(\km^{(i)})_{i\leq 2s-1}$  and write $X\in \km^{(1)}$ the lift of $\Xab$}.

\smallskip

\noindent \emph{Remark}. Comparing \ref{nestedg} and  \ref{bracket-gr}, we see that $\Xm$ lies in the center of $(\tkg,*')$.

\begin{example} In the case of  the biased Heisenberg group  mentioned in \ref{weight-filtration}, we can choose $\km^{(1)}=\R e_{1}\oplus \R e_{2}$, $\km^{(2)}=\{0\}$, $\km^{(3)}= \R e_{3}$.  The bracket   $[.,.]'$ is such that  $[e_{i}, e_{j}]'=0$ and $[\chi, e_{j}]'=[e_{2}, e_{j}]$ for  $i,j\in \{1,2, 3\}$. 
\end{example}

\subsection{ Dilations}\label{dil} The weight decompositions on $\kg$ and $\tkg$ induce a one-parameter subgroup of dilations $D_r$, which are the linear maps that act on $\tkm^{(i)}$ by multiplication by $r^{i}$. Note that with the $*'$ product, $D_r$ becomes an automorphism of $\tkg$ (and $\kg$): $\forall x,y\in \tkg$, 
$$D_{r}(x*'y)=D_{r}x*'D_{r}y.$$
The determinant of the dilations $D_r$ on $\kg$ is $\det D_{r}=r^{\dd }$, where

\begin{equation} \label{hom-dim}\dd = \sum_{i \ge 1} \dim \kg^{(i)}. 
\end{equation}

We call $\dd$ the \emph{homogeneous dimension} of $\kg$ with respect to the weight filtration  induced by $\Xab$.  When $\Xab=0$, $\dd$ coincides with the homogeneous dimension of $\g$, namely $d_0=\sum_{i \ge 1} \dim \kg^{[i]}$, which controls the volume growth of balls in $G$ by the Bass-Guivarc'h formula \cite{guivarch73, breuillard14}. Note that $\dd \ge d_0$ in general.


\subsection{ Bi-grading on the free Lie algebra}  \label{Sec-bi-grading}
 For basics on the free Lie algebra we refer the reader to \cite[chapter 2]{serre06} and \cite{reutenauer93}. 
 We denote by  $L^{[r]}$ the \emph{$r$-bracket}, which is defined recursively as $L^{[1]}(x)=x$ and 
$$L^{[r]}(x_{1}, \dots, x_{r})  =[L_{r-1}(x_{1}, \dots, x_{r-1}), x_{r}],$$
an element of the free Lie algebra on $r$ letters.

Throughout we will consider formal expressions involving Lie brackets of $N$ indeterminates $x_1,\ldots,x_N$ corresponding to $\tkg$-valued random variables. In particular the product of $\Pi(\xx):=x_{1}*x_{2}*\dots*x_{N}$ can be expanded via the Campbell-Hausdorff formula (see \cite{dynkin47}) into an element of the $s$-step free Lie algebra on $N$ variables of the form:  

$$\Pi(\xx) = \sum_{r=1}^s \sum_{ \underset{ \,1\leq n_{1}< \dots < n_{t}<\infty}{r_{1}+\dots+r_{t}= r}, \,r_{i}\neq 0}  L_{\rr}(x_{n_{1}}^{\otimes r_{1}}, \dots, x_{n_{t}}^{\otimes r_{t}})$$
where $\underline{r}=(r_{1}, \dots, r_{t})$ and  $L_{\rr}$ is a linear combination with rational coefficients of terms, each of which is an $r$-bracket $L^{[r]}$ of some permutation of the $r$ variables involved.

\begin{example} The Campbell-Hausdorff formula allows to compute   $L_{1}(x_{1})=x_{1}$, $L_{1,1}= \frac{1}{2}L^{[2]}$, $L_{2,1}=L_{1,2}= \frac{1}{12}L^{[3]}$, 
$$L_{1,1,1}(x_{1}, x_{2}, x_{3})=\frac{1}{6}L^{[3]}(x_{1}, x_{2}, x_{3}) +\frac{1}{6}L^{[3]}(x_{3}, x_{2}, x_{1}).$$ 
\end{example}
\bigskip

Note that the free Lie algebra comes with a natural grading given by bracket length. To account for the weight structure discussed above, we will need to consider a further grading on the $s$-step free Lie algebra. To each indeterminate $x_j$ we associate $2s-1$ extra indeterminates $x_j^{(i)}$ for $i=1,\ldots,2s-1$ so that $x_j=x_j^{(1)}+\ldots+x_j^{(2s-1)}$ and expand the product $\Pi(x_1,\ldots,x_N)$ as an element of the $s$-step free Lie algebra  $\kF_s$ on the set of $(2s-1)N$ indeterminates $\{x_j^{(i)}\}_{1\le j \le N, 1 \le i \le 2s-1}$.   We define a grading on this free Lie algebra by assigning weight $i$ to each $x_j^{(i)}$ and defining the weight of an $r$-bracket of $x_j^{(i)}$'s to be the sum of the weights of its $r$ variables. On the other hand, we have the ordinary grading of the free Lie algebra, where the degree of an $r$-bracket equals its length $r$. This yields a bi-grading 
\begin{equation}
    \kF_s = \bigoplus_{1\leq a \leq s, \, 1\leq b \leq 2s-1} \kF_s^{[a,b)}
\end{equation}
where the homogeneous part $\kF_s^{[a,b)}$ is spanned by all $a$-brackets of weight $b$ in the variables $x_j^{(i)}$. We thus get a decomposition:
$$\Pi(\xx) = \oplus_{a,b} \Pi^{[a,b)}(\xx).$$

\noindent We also set $\Pi^{[a]}(\xx):=\oplus_{b} \Pi^{[a,b)}(\xx)$ 
and $\Pi^{(b)}(\xx):=\oplus_{a} \Pi^{[a,b)}(\xx)$. 
 
Coming back to our Lie algebra $\tkg$ of nilpotency class at most $s$ (we use the words step and nilpotency class interchangeably)  endowed with the grading  $\tkg= \oplus_{ 1\leq b \leq 2s-1} \tkm^{(b)}$, the element  $\Pi^{[a,b)}(\xx) \in \kF_s$ naturally yields an evaluation mapping on $\tkg^N$ with values in $\tkg^{(b)}\cap \tkg^{[a]}$, which, by abuse of notation, we continue to denote by $\Pi^{[a,b)}(\xx)$:
\begin{equation}\label{piab}
    \Pi^{[a,b)}: \tkg^N \to \tkg^{(b)}\cap \tkg^{[a]}.
\end{equation}
In fact, when the evaluation is taken in $\tkg$ with the $*'$ product, we will denote the evaluation mapping by $\Pi'^{[a,b)}$. We note that $\Pi'^{[a,b)}$ is homogeneous of degree $b$ with respect to the dilations $D_r$, namely for any $\xx \in \tkg^N$:
$$\Pi'^{[a,b)}(D_r \xx)=r^{b}\Pi'^{[a,b)}(\xx).$$

\subsection{Norm, Lebesgue measure, and dual}\label{basis} Throughout we shall fix a basis $(e_j^{(i)})_{1\leq j \leq q_i}$ of each $\km^{(i)}$, adding $e_0^{(2)}=\chi$ to $\tkm^{(2)}$. We shall consider the induced Euclidean norm $\|\cdot\|$ on $\kg$ and $\tkg$, thus making the subspaces $\tkm^{(i)}$  mutually orthogonal. The associated Lebesgue measure will be denoted by $dx$. Notice that $dx$ is also a left and right invariant Haar measure on the associated Lie groups (for both $*$ and $*'$). 

We also fix a choice of left invariant Riemannian metrics on the various group structures $(\kg, *)$, $(\kg, *')$, $(\tkg, *')$ and denote them respectively by $\norm{.}$, $\normp{.}$, $\normtp{.}$.

We will often consider   the dual $\dkg$ of linear forms on $\kg$. It is endowed with operator norm and the Lebesgue measure $d\xi$ associated to the dual basis of the $(e_j^{(i)})_{i,j}$. Linear forms $\xi\in \dkg$ will be identified with their extension to $\tkg$ by setting $\xi(\chi)=0$.

\subsection{Polynomial functions and degree}\label{poldef} 
With the above basis, we obtain a well-defined notion of degree of a polynomial function by assigning degree $i$ to each member ${e_j^{(i)}}^*$ of the dual basis of $\tkm^{(i)}$. For example, composing $\Pi^{[a,b)}(\xx)$ with any linear form on $\tkm^{(b)}$ yields a polynomial function of degree $b$. This notion of degree  depends on the filtration $(\tkg^{(b)})_{b \ge 1}$, but not on the choice of weight decomposition and nor on the choice of basis.

A function on $\kg$ or $\tkg$  is Schwartz if it is smooth with all derivatives decaying faster than any polynomial.  The class of Schwartz functions is solely determined by the vector space structure of $\kg$ or $\tkg$.

\subsection{The driving measure $\mu$} \label{Sec-cadre-mu}

Let $\mu$ be a probability measure on $\kg$. For $r>0$, we say that $\mu$ has \emph{finite $r$-th moment} for the weight filtration if one has $$m_{r}(\mu):=\sum_{1\leq b\leq 2s-1}\int_{\kg} \|x^{(b)}\|^{r/b}d\mu(x)<\infty .$$
In particular, for every polynomial $P:\kg\rightarrow \R$ of degree at most $m$ for the weight filtration
$$\int_{\kg} |P(x)|d\mu(x)<\infty .$$
Two measures $\mu$, $\eta$ with finite $m$-th moment are said to \emph{coincide up to order $m$} if $\mu(P)=\eta(P)$ for all $P$ of degree at most $m$.

We will estimate the $\mu$-average (or $\mu^{*N}$-average) of an integrable test function $f\in L^1(\kg, dx)$ using Fourier analysis. Given $\xi \in \dkg$, we set
$$e_{\xi} : \kg \rightarrow \C, x\mapsto e^{-2i\pi \xi(x)} $$
and we define the \emph{Fourier transform} of  $\mu$ or $f$ by 
$$\hmu(\xi)= \int_{\kg} e_{\xi}(x)d\mu(x) \,\,\,\,\,\,\,\,\,\,\,\,\,\,\,\,\hf(\xi)= \int_{\kg} e_{-\xi}(x) f(x)dx.$$
If $f$ is continuous and $\hf$ is integrable, the Fourier inversion formula states that
\begin{align*}
\mu (f)&=\int_{\dkg} \hf(\xi) \,\hmu(\xi) d\xi.
\end{align*}

\subsection{The bias extension $\tmu$}  \label{notation-bias}

We will consider a measure $\mu$ with finite second moment for the weight filtration and whose projection $\mu_{ab}$ to $\kg/[\kg, \kg]$ has expectation $\Xab \mod [\kg, \kg]$.  We define the \emph{bias extension} $\tmu$ of $\mu$ as the image of $\mu$ by $\kg\rightarrow \tkg, x \mapsto (x,1)$.  Given a weight decomposition $(\km^{(b)})$ (thus lifting $\Xab$ into $\Xm\in \km^{(1)}$ and defining $\chi=(\Xm, 1)\in \tkg$), one may write for every $N\geq 1$,
$$\tmu^{*N} = \mu^{*N}-N\Xm+N\chi =\mu^{*N}*-N\Xm*N\chi  .$$
Here, the notation $ \mu^{*N}-N\Xm+N\chi $ represents the image of $\mu^{*N}$ by the map $\tkg\rightarrow \tkg, x\mapsto x-N\Xm+N\chi$, and similarly for $\mu^{*N}*-N\Xm*N\chi$. The measure $\tmu$ reports the drift $\Xm \in \km^{(1)}$ into $\chi\in \tkm^{(2)}$ while keeping  the fluctuations around the drift in $\R\Xm \subseteq \km^{(1)}$.

\subsection{Moderate deviations} \label{Sec-cadre-deviations}

Given $N\geq 1$, $g,h\in \tkg$, we denote by $g*\tmu^{*N}*h$, or $\tmu^{*N}_{g,h}$ for short, the image of $\tmu^{*N}$ by $\tkg\rightarrow \tkg, x\mapsto g*x*h$. A similar notation holds for deviations of $\mu$ on $\kg$. 


We shall prove  estimates for $\tmu^{*N}_{g,h}$ uniformly for $g, h$ in a certain range of \emph{moderate deviations}: given  $\delta_{0} \in [0, 1)$, we let
$$\DN(\delta_{0})=\{x \in \tkg  \,:\, \forall i\leq 2s-1,\,\, \|x^{(i)}\| \leq  N^{i/2+s^{-1}\delta_{0}}\},$$ where $x^{(i)}=\pi^{(i)}(x)$ is the coordinate projection with respect to the grading $(\tkm^{(b)})_{b}$ of $\tkg$. It is worth pointing out that $\DN$ behaves well with respect to the product $*$ :  There exists $C>0$ such that for all $\delta_{0}>0$, for all $N \geq 1$ large enough,  
 $$x,y\in \DN(\delta_{0}) \implies x*y\in \DN(C \delta_{0}).$$

\subsection{Asymptotic notations} \label{Sec-cadre-asympt}
In the rest of the text, we use the standard Vinogradov notation  $\ll$ and Landau notations $O(.)$, $o(.)$. Unless otherwise stated, the implied constants will depend only on the initial data $\kg, \Xab, \mu, (\km^{(b)})_{b}, (e^{(i)}_{j})_{i,j}, \norm{.}, \normp{.}, \normtp{.}$.  In some statements, we will also use the notation $r_{1} \lll r_{2}$ to mean that the conclusions that follow are true up to choosing $r_{1}\leq c r_{2}$  where $c>0$ is a (small) constant depending only the initial data. Subscripts are added to indicate any additional dependency (for instance, on the parameter $\delta_{0}$ controlling the range of deviations). Furthermore $\N$ denotes the non-negative integers and $\R^+$ the non-negative reals.


\addtocontents{toc}{\protect\setcounter{tocdepth}{2}}

\section{Limiting diffusion on a nilpotent Lie group} \label{Sec-densite}

In this section, we study  the diffusion processes on a nilpotent Lie group that  arise as limits of rescaled i.i.d. right random walks. We start by recalling  background material on the general theory of diffusion processes on Lie groups. Then we focus on the case of   limiting diffusions  on nilpotent Lie groups, and establish hypoellipticity as well as a general Gaussian upper bound. We continue with a description of the support of the time-$1$ distribution, and give several examples. Finally, we give a general lower bound, and several corollaries, notably  characterizing the case where the limiting process is  Gaussian in the ordinary Euclidean sense.

\subsection{Diffusions on a connected Lie group} \label{Sec-diff-general}

We start by recalling some standard facts about diffusion  processes on a connected real Lie group. 
We refer the reader to the books \cite{liao-book} and \cite{hazod-siebert} as well as Hunt's original article \cite{hunt56} and the papers by Stroock and Varadhan \cite{stroock-varadhan72,stroock-varadhan73} for more background material on diffusion processes. The support of invariant diffusions was also studied by Siebert, see \cite[Theorem 2]{siebert82}. Connections with sub-Riemannian geometry are discussed in \cite{varopoulos-saloff-coulhon92, kupka} and references therein.

Let  $H$ be a  connected real Lie group and denote by $1_{H}$ the identity element of $H$. 
A \emph{continuous left-invariant diffusion process} on $H$ is a stochastic process $(\W(t))_{t\in  \R^+}$ satisfying the following:
\begin{itemize}
\item $\W(0)=1_{H}$ with probability $1$.
\item The random map $\W: \R^+\rightarrow H, t\mapsto \W(t)$ is continuous with probability $1$.
\item For any $k\in \N$ and $0= t_{0}<\dots<t_{k}<+\infty$, the increments $\left(\W(t_{i})^{-1}\W(t_{i+1})\right)_{i=0, \dots, k-1}$ are mutually independent.
\end{itemize}
 In addition, if for any $0\leq s< t$, the distribution of the increment $\W(s)^{-1}\W(t)$ only depends on $t-s$, then the process is called \emph{homogeneous}. 

Denote by $\Omega(H)$ the set of continuous paths $c: \R^+\rightarrow H$, set $\mathcal{F}_{t}$ the $\sigma$-algebra generated by the projections $s\mapsto c(s)$ where $s\leq t$, and $\mathcal{F}_{\infty}=\cup_{t\in \R^+} \mathcal{F}_{t}$. The law of a diffusion process $\W$ on $H$ is a certain distribution $\sigma$ on $(\Omega(H), \mathcal{F}_{\infty})$. We denote by $\sigma_{s}$ the image of $\sigma$ by $c\mapsto c(s)$, or in other terms the law of $\W(s)$. We denote by $\mathcal{Q}$ the collection of probability measures on $(\Omega(H), \mathcal{F}_{\infty})$ which arise from a diffusion, and call them diffusion laws on $H$. 
\bigskip

A left-invariant second order differential operator on  $H$ defines a diffusion law on $H$ as follows.  Denote by $\kh$ the Lie algebra  of $H$, identify any vector $X\in \kh$ with its associated left-invariant vector field on $H$ and with the corresponding Lie derivative. In particular, for any $f\in C^1(H)$, 
$$X.f :x\mapsto \lim_{\eps \underset{\neq}{\to}  0}\frac{f(x  e^{\eps X})-f(x)}{\eps }.$$
Consider a time-dependent left-invariant  second order differential operator $\LL=(\LL_{t})_{t\in \R^+}$ of the form
\begin{align*}
\mathscr{L}_{t}=\frac{1}{2}\sum_{i=1}^q  {E_{i}(t)}^2 + B(t)
\end{align*}
where  $E_i:\R^+\rightarrow \kh$ for $i=1,\ldots,q$  and $B  : \R^+\rightarrow \kh$ are smooth maps.

\begin{fact}[Infinitesimal generator \cite{liao-book}] \label{fact-diffusion}
There exists a unique diffusion law $\sigma \in \mathcal{Q}$ such that for every $s\geq 0$ and every $\mathcal{F}_{s}$-measurable variable $f\in C^\infty_{c}(H)$, the process 
\begin{align*}
f(c(t)) -f(c(s))- \int_{s}^t(\mathscr{L}_{r}f)(c(r)) dr   \tag{$t\in \R^+$}
\end{align*}
is a martingale on $(\Omega(H), (\mathcal{F}_{t})_{t\geq s}, \sigma)$.  
\end{fact}

Note that $\sigma$ characterizes $\LL$ because for any fixed $f\in C^\infty_{c}(H)$, 
$$ \frac{1}{\eps} \E_{\sigma}\left(f  \left[c(s)^{-1}c(s+\eps)\right]  -f(1_{H}) \,|\, \mathcal{F}_{s}\right) \underset{\eps\to 0^+}{\longrightarrow} \mathscr{L}_{s}f(1_{H})$$ 
where the arrow stands for the convergence in probability.

A  diffusion process  on $H$ with infinitesimal generator $\LL$ is by definition a stochastic process $\W$ on $H$ with law $\sigma$. Observe then that for any $s>0$, the process $(\W(s)^{-1}\W(s+t))_{t\geq 0}$ is a diffusion with infinitesimal generator $(\mathscr{L}_{s+t})_{t\in \R^+}$. It follows that $(\W(t))_{t\geq 0}$ is homogeneous if and only if $\mathscr{L}$ is  independent of the time parameter, and in this case the distributions $\sigma_{t}$ form a semigroup: $\forall s,t\in \R^+$, 
$$\sigma_{s}*\sigma_{t}=\sigma_{s+t}. $$

\bigskip
We now  describe the support of the time-$t$ distribution $\sigma_{t}$ associated to $\LL$. Given an open interval $I$, a $C^1$ path $\gamma : I\rightarrow H$, and $s\in I$, we define the multiplicative derivative $\partial\gamma(s)\in \kh$ as the derivative at $t=0$ (in the classical sense) of the path $t\mapsto \gamma(s)^{-1}\gamma(s+t)$. We say that a continuous, piecewise $C^1$ path $\gamma: [0, t]\rightarrow H$ is $\mathscr{L}$-\emph{horizontal } if  its multiplicative derivative satisfies
$$\partial\gamma(s) \in \sum_{i=1}^q \R E_{i}(s)+ B(s) $$
for every $s$ in the complement of a finite subset of $[0, t]$. 

\begin{fact}[Support \cite{stroock-varadhan72, stroock-varadhan73}] \label{fact-support}
The support of $\sigma_{t}$ is  the closure of the set of arrival points $\gamma(t)$, where $\gamma: [0, t]\rightarrow H$ is a continuous piecewise $C^1$ $\mathscr{L}$-horizontal  path. 
\end{fact}
 This fact is a consequence of Stroock-Varadhan's support theorem \cite{stroock-varadhan72} for diffusion processes on $\R^n$, applied to $\W$ locally in charts of $H$ as in \cite[page 284]{stroock-varadhan73}, where it is shown that invariant diffusions on Lie groups never explode in finite time.

\bigskip

Next we give a smoothness criterion for the distribution $\sigma_{t}$. Denote by $\mathcal{H}$ the Lie algebra of vector fields on   the  product manifold $\R_{>0}\times H$ generated by $(E_{i})_{i\leq q}$ and $\partial_{t}+ B$, where $\partial_{t}$ stands for the constant  vector field $(1,0)$. We see $E_{i}, B$ as functions $\R_{>0}\times H\rightarrow TH$, where $TH$ is the tangent bundle of $H$.  
Recall  H\"ormander's criterion for hypoellipticity:

\begin{fact}[Hypoellipticity \cite{hormander67}]    \label{fact-hypo} Assume the H\"ormander condition: for  every  point $(t,x)\in \R_{>0}\times H$, every vector $v\in TH_{x}$ tangent to $x$, there exists $Y\in \mathcal{H}$ such that $Y(t,x)=v$. Then  the differential equation on   $\R_{>0}\times H$ given by
\begin{align} \label{eq-hypo}
\partial_{t} - \frac{1}{2}\sum_{i=1}^q E_{i}^2+ B=0
\end{align}
is hypoelliptic, i.e. every solution (in the sense of distributions) is smooth.
\end{fact}

A direct computation shows that the distribution $u=\int_{\R_{>0}}\delta_{t}\otimes \sigma_{t}\,dt$ is solution of  \eqref{eq-hypo}. Assuming \eqref{eq-hypo} is hypoelliptic, one may then write for $t>0$, 
$$\sigma_{t}=u(t,x)dx$$
 where $u:\R_{>0}\times H\rightarrow \R^+$ is a smooth function, and $dx$ stands for a left-invariant Haar measure on $H$. 

Hypoellipticity also yields the following domination principle \cite[Corollary III.1.3]{varopoulos-saloff-coulhon92}: for any  open set  $\Omega\subseteq \R_{>0}\times H$, any  compact subset  $K\subseteq \Omega$,  $p\in \N$, there exists $C_{1}>0$ such that for every solution $v(t,x)$ of  \eqref{eq-hypo}, one has 
 \begin{align} \label{bound-der-int}
 \| v\|_{C^p(K)} \leq C_{1} \int_{\Omega} |v| \, dtdx 
 \end{align}
where $\| v\|_{C^p(K)}$ is the supremum on $K$ of the absolute value of the derivatives of $v$ up to order $p$. Later, we will use  this inequality to convert a moment estimate on $\sigma$ into a bound on the derivatives of the density $u$.
\bigskip

We also record a Harnack inequality, namely a statement to the effect that non-negative solutions of \eqref{eq-hypo} do not decrease too much while going along a small horizontal path. We rely here on the recent work of Kogoj and Polidoro \cite{kogoj-polidoro16} addressing the setting of a time-dependent hypoelliptic equation. Let $1_{H}\in U\subseteq H$ be an open neighborhood of the identity, whose closure is compact and included in a chart of $H$. Let $0<\eps_{0}<\eps_{1}$ be two parameters and define $\Pc$ as the set of couples $(\tau,x)\in (\eps_{0}, \eps_{1}) \times U$ such that there exists a $U$-valued horizontal path $\gamma : [\eps_{0}, \tau]\rightarrow U$ such that $\gamma(\eps_0)=1_H$, $\gamma(\tau)=x$.

\begin{fact}[Harnack inequality  \cite{kogoj-polidoro16}]\label{fact-harnack} Assume the H\"ormander condition from \Cref{fact-hypo}.  Let $U,\eps_{0}, \eps_{1}, \Pc$ as above. Let $L$ be a compact set in the interior of the closure of $\Pc$ in $[\eps_{0}, \eps_{1}]\times \overline{U}$. 
Then there exists a constant $C>1$ such that 
for any  non negative solution $v :(\frac{\eps_{0}}{2}, 2\eps_{1})\times U \rightarrow \R^+$ of  \eqref{eq-hypo}, one has 
$$v(\eps_{0}, 1_H)\leq C \inf_{(t,x)\in L} v(t, x) .$$
\end{fact}
This fact is a consequence
of \cite[Corollary 2]{kogoj-polidoro16}, applied in the bounded chart $U$ to the differential operator $\partial_{t} - \frac{1}{2}\sum_{i=1}^q( E'_{i})^2+ B'$ where $E'_{i}(t)=E_{i}(-t), B'(t)=-B(-t)$. The time reversal we use here allows to convert the supremum on the left-hand side in the Harnack inequality from \cite{kogoj-polidoro16}  into an infimum on the right-hand side, as displayed above.
\bigskip

Finally, we put aside any assumption of hypoellipticity on $\sigma$ and record a moment estimate due to  Kisynski \cite[Theorem 2]{kisynski79} (see also \cite[Proposition 3.1]{jorgensen75}) dealing with homogeneous diffusions. Denote by $d_{H}$ the distance induced by a  left-invariant Riemmanian metric on $H$. 

\begin{fact}[Quadratic exponential moment \cite{kisynski79}] \label{fact-moment} Assume $\LL$ is independent of the time parameter (but not necessarily hypoelliptic). Then there exists $\eps>0$ such that for  $x\in H$, all $T>0$
$$\sup_{t\in [0,T]}\int_{H} \exp\left(\eps d_{H}(x,y)^2 \right)\,d\sigma_{t}(y) <+\infty.$$
\end{fact}

\subsection{Invariance and regularity of the limiting diffusion} \label{Sec-inv-reg}

In this subsection we restrict our attention to a certain family of diffusion processes on nilpotent Lie groups. They are those that arise as limiting processes for right random walks with i.i.d. increments. Using \Cref{Sec-diff-general}, we  prove that every  such   diffusion is smooth, and give a Gaussian upper bound on its density and each of its derivatives (\Cref{density-reg}). 
In the biased case, the diffusion process we have to deal with can be seen in two ways, either as a time-dependent left-invariant diffusion process on $\kg$ with absolutely continuous increments (with generator \eqref{limit-diff-gen} below), or as the restriction to the affine hyperplane $\kg + t\chi$ of a bona fide left-invariant time-homogeneous diffusion process on the bias extension $\tkg$ (\Cref{homogeneisation}). Both points of view are good to bear in mind.

Let $\kg$ be a nilpotent Lie algebra, $\Xab \in \kg/[\kg,\kg]$ a bias, and $\kg=\oplus_{b} \km^{(b)}$ an adapted weight decomposition for the  weight filtration induced by $\Xab$.  We saw in  \Cref{Sec-cadre} that these data define a bias extension $\tkg$, a variable $\chi$ for which  $\tkg=\kg+\R\chi$, and graded Lie bracket or Lie product $[.,.]'$, $*'$ on $\tkg$. 

For the rest of the section, we fix vectors $E_{1}, \dots, E_{q}, B, \YY\in \tkg$ such that the linear span of $\{E_{1}, \dots, E_{q}\}$ is $\km^{(1)}$, $B\in \km^{(2)}$, and $\YY\in \chi +\km^{(2)}$. Our goal is to study the diffusion law $\sigma$ on $(\kg, *')$ whose infinitesimal generator is given by $\LL=(\LL_{t})_{t\in \R^+}$ where 
 \begin{align} \label{limit-diff-gen}
\mathscr{L}_{t}=\frac{1}{2}\sum_{i=1}^{q}  \left(\Ad(t\YY) E_{i}\right)^2+ \Ad(t\YY)B.
 \end{align}
 Here the notation $\Ad$ refers to the adjoint representation for $*'$, namely for every $X,Y$ in $\tkg$, $\Ad(Y)X=Y *' X*' (-Y)$. We will occasionally use the notation $\W$ to refer to a diffusion process with law $\sigma$.

If the random walk on $(\kg,*)$ is driven by a measure $\mu$ with finite second moment and mean $\Xab$ in the abelianization, then we shall prove  (\Cref{Sec-global-thm}) that the limiting process has  infinitesimal generator $\LL=\LL(\mu, (\km^{(b)}))$ such that $E_{1}, \dots, E_{q}$ is a basis of $\km^{(1)}$ in which the covariance matrix of $\mu^{(1)}$ is the identity, $B=\E_\mu(x^{(2)})$, and $\YY=\chi$. We allow here $\YY$ to be more generally in $\chi+ \km^{(2)}$ in order to consider a class of processes that is stable by change of weight decomposition. This flexibility will be useful later for explicit computations of the support of the time-$t$ distribution $\sigma_{t}$. 

\begin{lemme}[Changing the weight decomposition] \label{changing-dec}
Let $\kg=\oplus_{b} \mathring{\km}^{(b)}$ be some other choice of weight decomposition for the weight filtration $(\kg^{(b)})$ induced by $\Xab$.  Call $\mathring{\chi}$, $\mathring{*}'$ the associated objects on $\tkg$.

 Let  $\phi : \tkg\rightarrow \tkg$ be the unique linear map such that $\phi : \tkm^{(b)} \rightarrow \widetilde{\mathring{\km}}^{(b)}$ is the isomorphism  that factorizes via the identity on the vector space $\tkg^{(b)}/\tkg^{(b+1)}$.
 
 Then $\phi$ is a group isomorphism from $(\tkg, *')$ to $(\tkg, \mathring{*}')$. It satisfies $\phi(\km^{(2)}) = \mathring{\km}^{(2)}$ and  $\phi(\chi)\in \mathring{\chi} +\mathring{\km}^{(2)}$.  Moreover, $\phi \circ \sigma$ is the diffusion law on $(\tkg, \mathring{*}')$ determined by the parameters $\phi(E_{i}), \phi(B), \phi(\YY)$. 
\end{lemme}

\noindent\emph{Remark}. $\phi(\chi)$ may not be equal to $\mathring{\chi}$. It follows from the proof that the equality case occurs exactly when the respective lifts of $\Xab$ to $\km^{(1)}$ and $\mathring{\km}^{(1)}$ have the same projection in the quotient vector space $\kg/\kg^{(3)}$.

\begin{proof}
For the first claim, recall the graded Lie algebra $gr_{\Xab}(\tkg)$ from \Cref{Sec-cadre}.  A choice of weight decomposition for $\kg$ yields a vector space isomorphism by $\tkg$ to $gr_{\Xab}(\tkg)$, by sending via the identity the subspace of weight $b$ to $\tkg^{(b)}/\tkg^{(b+1)}$. For the choices $\kg= \oplus \km^{(b)}$, $\kg= \oplus \mathring{\km}^{(b)}$, call those maps $pr, \mathring{pr}$. The graded bracket $[.,.]'$ is defined so that $pr : (\tkg, [.,.]')\rightarrow gr_{\Xab}(\tkg)$ is a morphism of Lie algebras, and the same holds for $\mathring{[.,.]}', \mathring{pr}$. As $\phi= \mathring{pr}^{-1}\circ pr$, we deduce that $\phi$ is an isomorphism of Lie algebras from $(\tkg, [.,.]')$ to $(\tkg, \mathring{[.,.]}')$, which implies the claim.

For the second claim, we observe that $pr(\km^{(2)})=\kg^{(2)}/\kg^{(3)}=\mathring{pr}(\mathring{\km}^{(2)})$, so in particular $\phi(\km^{(2)}) = \mathring{\km}^{(2)}$. Now, slightly abusing notations,  denote by  $\Xm\in \km^{(1)}$, $\mathring{\Xm}\in \mathring{\km}^{(1)}$ the respective lifts of $\Xab \mod [\kg, \kg]$. One can write $\Xm=\mathring{\Xm}+\mathring{Y}+\mathring{Z}$ where $\mathring{Y}\in \mathring{\km}^{(2)}$, $Z\in \kg^{(3)}$. This equality can be rewritten as $\chi=\mathring{\chi}+\mathring{Y}+Z$, and by definition of $\phi$, we get 
$$\phi(\chi)=\mathring{\chi}+\mathring{Y}.$$

The third claim is a straightforward consequence of the first claim. The second claim ensures that the type of diffusions  we consider is indeed preserved by $\phi$.
\end{proof}

We now stick with the given weight decomposition $\kg=\oplus_{b} \km^{(b)}$ and present  invariance properties for the diffusion  $\sigma$. Recall that $D_r$ denotes the dilation $D_{r} =\oplus_{b}\,r^{b}\text{Id}_{\km^{(b)}}$.

\begin{proposition}[Invariance] \label{density-inv}
The diffusion law $\sigma$ associated to $\mathscr{L}$ (defined in \cref{limit-diff-gen}) satisfies  for all $r, s,t>0$, 
 \begin{align*}
\sigma_{s}*'\Ad(s\YY)\sigma_{t}= \sigma_{s+t}\,\,\,\,\,\,\,\,\,\,\,\,\,\,\,\,\,\,\,\,\,\,\,\,\,\,\,\,\,\,\,\,\,\,\,\,\,\,\,\,D_{\sqrt{r}}\sigma_{t}=\sigma_{rt} .
 \end{align*}
\end{proposition}

\begin{proof}
For the first identity on the left hand side, it is sufficient to check that the process $(\Ad(s\YY)\W(t))_{t\in \R^+}$ is the diffusion with infinitesimal generator $(\mathscr{L}_{s+t})_{t\in \R^+}$. This follows from definitions and the straightforward computation 
 \begin{align*}
 \mathscr{L}_{t}( f\circ \Ad(s\YY)) = \mathscr{L}_{s+t}( f)\circ \Ad(s\YY)  
 \end{align*}
 
 For the identity on the right hand side, it is sufficient to check that the diffusion processes $(\W(rt))_{t\in \R^+}$ and $(D_{\sqrt{r}} \W(t))_{t\in \R^+}$ have the same infinitesimal generator $(r\mathscr{L}_{rt})_{t\in \R^+}$. This follows on the one hand by a simple change of variables, and on the other hand from the straightforward computation
   \begin{align*}
 \mathscr{L}_{t}( f\circ D_{\sqrt{r}} ) = r\mathscr{L}_{rt}( f)\circ D_{\sqrt{r}}  .
 \end{align*}

\end{proof}

We now study the regularity of $\sigma_t$. We denote by $\normtp{x}$ the  distance between  $0$ and $x$ for a fixed left-invariant Riemannian metric on  $(\tkg, *')$. For $\alpha=(\alpha^i_{j})_{i,j}\in \N^d$, we write $\partial^\alpha$  the operator differentiating (for the addition) $\alpha^i_{j}$ times in each direction $e^{(i)}_{j}$ (see \Cref{basis} for this notation). We also set  $\|\alpha\|=\sum_{i,j} i\alpha^i_{j}$ and recall that $\dd$ is the \emph{homogeneous dimension} of the weight filtration induced by $\Xab$ (\Cref{hom-dim}).

\begin{proposition}[Regularity] \label{density-reg}
The diffusion law $\sigma$ associated to $\mathscr{L}$  
can be written 
$$\sigma_{t}= u(t,x)dx$$ 
where $u : \R_{>0}\times \kg\rightarrow \R^+$ is a smooth function, and for each $\alpha\in \N^d$, there exists $C>0$ such that for all $t>0$, $x\in \kg$, 
 \begin{align}
 \partial^\alpha u(t,x)\leq C t^{-\frac{1}{2}(\dd +\|\alpha\|)}\,e^{-\frac{1}{C}\normtp{D_{\frac{1}{\sqrt{t}}}x}^2}. \label{derivu}
 \end{align}
\end{proposition}

\noindent\emph{Remarks}. 
 
1) The distance $\normtp{.}$ depends on the choice of grading $\km^{(b)}$. However any two left invariant geodesic metrics on a connected Lie group are coarsely equivalent and in our setting (\emph{ball box principle}, see e.g. \cite{bellaiche96}) for every Euclidean norm $\|.\|$ on $\tkg$, there exists $\eps\in (0,1]$ such that $$\eps \|.\|^\eps \leq \normtp{.} \leq \eps^{-1}\|.\|$$
outside of a compact neighbourhood of the origin in $\kg$. 
In particular, \eqref{derivu} implies that $u(t,x)$ is a \emph{Schwartz function} in $x$ and \cref{up-dens-bd} holds.

2) We will see in the proof of \Cref{density-reg} that \eqref{derivu} continues to hold for iterated right derivatives with respect to the product $*'$ in place of additive derivatives. 

3) The exponential rate of decay in the upper bound is probably not sharp in general. We expect that a $*'$-left invariant Riemannian metric on $\kg$ can be taken in place of  the distance $\normtp{.}$ in \cref{derivu}. This is what happens in the Heisenberg group case or whenever one is in the situation of \Cref{gaussian-criterion}. Not that it would be futile to obtain a similar looking lower bound in general, because the support of $\sigma_t$ may not be full in general as we shall see.

\bigskip

\noindent{\bf Plan of the proof}. The argument to prove \Cref{density-reg} will be in two parts. First, we use \Cref{fact-hypo}  to show that the diffusion $\sigma$ is hypoelliptic, justifying the smoothness of $\sigma_{t}$ and yielding also a principle of domination derivative-integral, that bounds the value of any fixed derivative of $u$  at a point by the integral of $u$ on a small neighborhood. Second, we show that $\sigma_{t}$ has quadratic exponential moment for $\normtp{.}$, by observing that $\sigma$ is the projection to $\kg$ of some  homogeneous (though degenerate) diffusion law on $(\tkg, *')$ and applying \Cref{fact-moment}.

\bigskip
The hypoellipticity of $\sigma$  relies on the following geometrical statement. 
\begin{lemme}\label{prehypo}
We have the equality $[\tkg, \tkg]'= \kg^{(2)}$. In particular, $(\tkg, [.,.]')$ is generated as a Lie algebra by $\km^{(1)} \oplus \R \YY$.  
\end{lemme}

\begin{proof}
The inclusion $[\tkg, \tkg]'\subseteq \kg^{(2)}$ follows directly from the definition of $[.,.]'$. For the converse inclusion, we show by induction that $\kg^{(b)}\subseteq  [\tkg, \tkg]'$ for all $b\in \{2, \dots, 2s\}$. The initial step $\{0\}=\kg^{(2s)}\subseteq [\tkg, \tkg]'$ is clear. Now  assume $\kg^{(b+1)}\subseteq  [\tkg, \tkg]'$. We  consider $x\in \kg^{(b)}$ and show $x\in [\tkg, \tkg]'$. By definition of $\kg^{(b)}$, we may suppose that $x$ is a (non-graded) bracket $x= [v_{1},[v_{2},[ \dots, [v_{a-1}, v_{a}] \dots]]]$ ($a\geq 2$) where the $v_{i}$'s belong to $\km^{(1)}$, and if $j$ denotes the number of occurences of $\Xm$ among them, then $a+j\geq b$. As $[\chi,.]=[\Xm,.]$, we may replace each  $v_{i}$ by $w_{i}=v_{i}\1_{v_{i}\neq \Xm}+ \chi \1_{v_{i}= \Xm}$ without altering the value of the total bracket. By induction hypothesis, we can further assume $a+j=b$.  Now  $x=[w_{1},[w_{2},[ \dots, [w_{a-1}, w_{a}] \dots]]]$ coincides with $[w_{1},[w_{2},[ \dots, [w_{a-1}, w_{a}]' \dots]']']'$ modulo $\kg^{(b+1)}$. As $\kg^{(b+1)}\subseteq  [\tkg, \tkg]'$ by induction hypothesis, it follows that $x\in  [\tkg, \tkg]'$. This proves the induction step, and the claim $[\tkg, \tkg]'= \kg^{(2)}$.

The second claim follows because  $\km^1 \oplus \R \YY$ projects surjectively to $\tkg/\kg^{(2)}$, which we just established to be the abelianization of $(\tkg, [.,.]')$. 
\end{proof}

We  combine \Cref{prehypo} and  \Cref{fact-hypo} to obtain

\begin{lemme} \label{hypo} The vector fields $\Ad(t\YY)E_{1}, \dots, \Ad(t\YY)E_{{q}}, \partial_{t}+\Ad(t\YY)B$ satisfy the H\"ormander condition from \Cref{fact-hypo}. In particular, the differential operator 
$$\partial_{t}- \frac{1}{2}\sum_{i=1}^{q}\left(\Ad(t\YY) E_{i}\right)^2+ \Ad(t\YY)B$$
 is hypoelliptic on $\R_{>0}\times \kg$. 
 \end{lemme}

Remember that for $X\in \kg$, the notation $\Ad(t\YY)X$ refers to the vector field  $\R_{>0}\times \kg\rightarrow T\kg, (t,x)\mapsto TL_{x}(\Ad(t\YY)X)$ where $L_{x}:y\mapsto x*'y$. It is  usually not left-invariant on the product Lie group $\R_{>0}\times \kg$.

\begin{proof}
We start with a few useful observations concerning the Lie brackets of vector fields on the manifold $\R_{>0}\times \kg$. Such brackets are denoted by $[.,.]$. The notation $[.,.]'$ is reserved for the graded bracket of elements in $\tkg$, seen as vectors and not vector fields. First, note that
$$[\partial_{t}, \Ad(t\YY)X]= \sum_{k=1}^s [\partial_{t}, \frac{t^k}{k!} \ad(\YY)^kX]$$
where $\ad(\YY):=[\YY, .]'$. 
Combined with the general formula $[Y,\varphi Z]=\varphi[Y,Z]+Y.\varphi Z$, and the equality $[\partial_{t}, Z]=0$ for every vector field $Z$ on $\kg$ extended to $\R_{>0}\times \kg$ by $\R$-invariance, we get  
$$[\partial_{t}, \Ad(t\YY)X]= \sum_{k=1}^s  \frac{t^{k-1}}{(k-1)!} \ad(\YY)^kX= \Ad(t\YY)[\YY, X]'.$$
Write $\kg^{[i]'}$ the descending central filtration of $(\kg, [.,.]')$, and similarly for $\tkg$. It follows that if $X$ belongs to $\kg^{[i]'}$, noting that $B\in [\tkg, \tkg]'$ (\Cref{prehypo}),  we can rewrite the bracket vector field 
$$[\partial_{t}+\Ad(t\YY)B, \Ad(t\YY)X]= \Ad(t\YY)([\YY, X]' +Z)$$
where $Z:= [B, X]' \in [\tkg^{[2]'}, X]'$ is a vector in $\kg$. 

We now check H\"ormander's condition from \Cref{fact-hypo}. Call $\mathcal{H}$ the Lie algebra  generated by the vector fields $\Ad(t\YY)E_{1}, \dots, \Ad(t\YY)E_{{q}}, \partial_{t}+\Ad(t\YY)B$ on $\R_{>0}\times \kg$. Write $\kh\subseteq \kg$ the subset of vectors $X\in \kg$ such that the vector field corresponding to $\Ad(t\YY)X$ is in $\mathcal{H}$. It is enough to check that $\kh=\kg$. Note that $\kh$  contains $E_{1}, \dots, E_{q}$ by definition. $\kh$ is also a sub Lie algebra of $(\kg, [.,.]')$, and in view of the previous paragraph, if $X\in \kh$ then $[\YY, X]'$ is in  $\kh + [\tkg^{[2]'}, X]'$.  By \Cref{prehypo}, this forces $\kh=\kg$ and concludes the proof.
\end{proof}

\begin{lemme}[Homogenisation] \label{homogeneisation}
Let $\lambda$ be the  \emph{homogeneous}  diffusion law on $(\tkg, *')$ associated to the differential operator 
$$\widetilde{\LL}=\frac{1}{2}\sum_{i=1}^{q}E_{i}^2+ (B+\YY).$$ 
If $W$ is a diffusion process on $\kg$ with law $\sigma$, then  $\Z(t):=\W(t)*'t\YY$ is a diffusion process on $\tkg$ with law $\lambda$. 
\end{lemme}



\bigskip

\noindent\emph{Remarks}.  
1) The diffusion $\lambda$ is not hypoelliptic.

2)
 Writing $\LL_{+}:=\frac{1}{2}\sum_{i=1}^{q}\Ad(t(B+\YY))E_{i}^2$, we see that $\widetilde{\LL_+}=\widetilde{\LL}$. In particular \Cref{homogeneisation} implies that the process $\W(t)*'t\YY*'(-t(B+\YY))$ is a diffusion on $(\kg, *')$ with infinitesimal generator $\LL_{+}$.  This trick  allows to suppress the drift parameter in the generator. It will be useful later for explicit computations of the support of $\sigma$ (\Cref{Sec-support}).

\begin{proof}
We start by noticing that for every $s\geq0$, if $F\in C^{\infty}_{c}(\R^+\times \kg)$ is a $\mathcal{F}_{s}$-measurable variable, then the stochastic process $(I(t))_{t\geq s}$ given by
\begin{align} \label{hom1}
I(t):=F(t, c(t))-F(s, c(s))- \int_{s}^t (\partial_{e_{1}}+\LL_{r})F(r, c(r)) dr
\end{align}
is a martingale on $(\Omega(\kg), (\mathcal{F}_{t})_{t\geq s}, \sigma)$. Here $\partial_{e_1}$ denotes the partial derivative with respect to the $\R^+$ variable. Indeed, for some martingales $(M_{i}(t))_{i\leq 3}$, we have
\begin{align*}
&F(t, c(t))-F(t, c(s))- \int_{s}^t (\LL_{r})F(r, c(r)) dr\\
&=  \int_{s}^t \LL_{r}[F(t,.)-F(r,.)]  (c(r)) dr +M_{1}(t)\\
&=  \int_{s}^t \LL_{r} \left[\int_{r}^t \partial_{e_{1}}F(z,.) dz\right] \,(c(r)) dr +M_{1}(t)\\
&=  \int_{s}^t \int_{r}^t  \LL_{r}\left[ \partial_{e_{1}}F(z,.) \right] dz \,(c(r)) dr +M_{1}(t)\\
&=  \int_{s}^t \int_{s}^z  \LL_{r}\left[ \partial_{e_{1}}F(z,.) \right](c(r))  dr dz +M_{1}(t)\\
&=  \int_{s}^t \partial_{e_{1}}F(z,c(z)) -\partial_{e_{1}}F(z,c(s))+M_{2}(z) \,dz +M_{1}(t)\\
&=  \int_{s}^t \partial_{e_{1}}F(z,c(z))  \,dz -F(t, c(s))+F(s, c(s))+M_{3}(t)
\end{align*}
which rewrites as $I(t)=M_{3}(t)$. 

Now set $F(s, x)=f(x*'s\YY)$ where  $f\in C^{\infty}_{c}(\kg)$ is a $\mathcal{F}_{s}$-measurable variable. Observe that 
\begin{align}\label{hom2}
 \partial_{e_{1}}F(r, c(r)) =(\YY .f)(c(r)*'r\YY)
 \end{align}
and for any vector $X\in \kg$,
\begin{align} \label{hom3}
(X.F)(r, c(r)) = (\Ad(-r\YY)X .f)(c(r)*'r\YY).
\end{align}
The result now follows by combining \eqref{hom1}, \eqref{hom2}, \eqref{hom3}.
 \end{proof}

We now combine \Cref{Sec-diff-general}, Lemmas \ref{hypo}, \ref{homogeneisation} to prove \Cref{density-reg}.

\begin{proof}[Proof of \Cref{density-reg}]
 From \Cref{hypo} and the background material from \Cref{Sec-diff-general}, one may write $\sigma_{t}= u(t,x)dx$ 
where $u : \R_{>0}\times \kg\rightarrow \R^+$ is a smooth function. The  relation $\sigma_{t}=D_{\sqrt{t}}\sigma_{1}$ obtained in \Cref{density-inv} can be restated as 
 \begin{align}
 u(t,x)=t^{-\frac{\dd}{2}}u(1, D_{\frac{1}{\sqrt{t}}}x ). \label{equivu}
 \end{align}
As $D_{\frac{1}{\sqrt{t}}}$ acts linearly on $\kg$, and $D_{\frac{1}{\sqrt{t}}}e^{(i)}_{j}=t^{-i/2}e^{(i)}_{j}$, the upper bound  \eqref{derivu} reduces to the case where $t=1$. 

Notice also that, in order to prove the upper bound \eqref{derivu},
 we may replace the additive derivative $\partial^\alpha u(t,x)$ by an iterated multiplicative derivative $(\underline{X} u)(1,x)$ where $\underline{X}=X_{1}\dots X_{k}$ is a finite sequence of vectors in $\kg$. Indeed, every additive derivation with respect to some $e^{(i)}_{j}$ can be written as a combination  with polynomial coefficients (in the $x$ variable) of right derivations for $*'$, and this remains true under iteration. The polynomial coefficients appearing do not matter, as they are absorbed by the exponential term (here $t=1$).

 We  introduce the bounded neighborhood of $(1,0)\in \R\times \kg$ given by  $\Omega :=]\frac{1}{2}, \frac{3}{2}[ \times \{x\in \kg, \, \normtp{x}<1\}$. For  $x\in \kg$, $(t,y)\mapsto u(t,x*'y)$ is a solution to \ref{eq-hypo}. Therefore   \Cref{bound-der-int} guarantees the existence of a constant $C_{1}>0$ that is independent of $x$ and for which
\begin{equation}\label{eqderint}
    |\underline{X}u(1,x)| \leq C_{1} \int_{\Omega}  u(t,x*'y)\,dt dy .
\end{equation}
On the other hand, the combination of \Cref{fact-moment} applied to $H:=(\tkg,*')$ and \Cref{homogeneisation} implies that 
$$I:=\sup_{ t\in [0,3/2]} \int_{\kg} \exp(\eps \normtp{y*'t\chi}^2) u(t,y)\, dy<\infty$$
for some $\eps >0$. 
For $(t,y) \in \Omega$, we have $\normtp{x*'y*'t\chi} \ge \normtp{x}/2$ whenever $\normtp{x}\geq 2(1+\sup_{t \in [0,3/2]}\normtp{t\chi})$. Hence 
$$\int_{\Omega} u(t,x*'y)\,dt dy\leq I \exp(-\frac{\eps}{4} \normtp{x}^2).$$
Finally, from \cref{eqderint} we get:
 $$|\underline{X}u(1,x)| \leq C_{1} I \exp(-\frac{\eps}{4} \normtp{x}^2) $$   
which concludes the proof.

\end{proof}

\subsection{Support of the limiting diffusion} \label{Sec-support}

We keep the notations of \Cref{Sec-inv-reg} and describe the support of the distribution $\sigma_{1}$. After a general and explicit description, we turn to two remarkable examples: for the upper triangular nilpotent Lie algebra, the support of $\sigma_{1}$ is always the full Lie algebra, whereas for non-centered walks (i.e. $\Xab\neq 0$) on a free nilpotent Lie algebra of step at least $3$, the support is always a proper closed subset. The emergence of a limiting diffusion without full support is  unexpected, and specific to the non-centered case.

\bigskip
Let us begin with a general description of the support $\SS$ of $\sigma_{1}$. It is clearly a closed subset of $\kg$, with dense interior because $\sigma_{1}$ is smooth. We use the \emph{support theorem} (\Cref{fact-support}) to give an explicit description in our context.

Recall that if $I \subseteq \R$ is an open interval and $\gamma : I \rightarrow \kg $ is $C^1$, we may define its (multiplicative) derivative at $t\in I$ by 
$$\partial \gamma (t)= \lim_{\eps \underset{\neq}{\to} 0} \frac{1}{\eps} \gamma(t)^{-1} *' \gamma(t+\eps).$$
When $*'$ is abelian, it is just the usual derivative $\gamma'(t)$, and we have in general the relation $\gamma'(t)=T L_{\gamma(t)} \partial \gamma(t)$ where $T L_{\gamma(t)} $ denotes the tangent map of the left $*'$-multiplication by $\gamma(t)$.

A theorem of Strichartz \cite{strichartz87} ensures that the derivation $\partial$ admits an inverse, generalizing the  classical Riemann integral in the abelian case, expressed in terms of \emph{iterated integrals} (\cite{chen}) as follows. Given a bounded and measurable path $c :[a,b] \rightarrow \kg$,  set 

\begin{align} \label{strich-formula}
\int^{*'}_{[a,b]}c \,&:= \sum_{r=1}^s \sum_{\tau\in \mathfrak{S}_{r}} \frac{(-1)^{e(\tau)}}{r^2 \binom{n-1}{e(\tau)}} \int_{\{a<t_{1}<\dots <t_{r}<b\}} L^{[r]'}(c(t_{\tau(1)}), \dots, c(t_{\tau(r)})) \,dt_{1}\dots dt_{r} \nonumber\\ 
&\,\\
&= \int_{[a,b]}c(t)dt \,+\, \frac{1}{2}\int_{\{a<t_{1}<t_{2}<b\}} [c(t_{1}), c(t_{2})]' dt_{1}dt_{2} \nonumber \\
&\,\,\,+\frac{1}{6}\int_{\{a<t_{1}<t_{2}<t_{3}<b\}} [[c(t_{1}), c(t_{2})]', c(t_{3})]' \,+\, [[c(t_{3}), c(t_{2})]', c(t_{1})]'dt_{1}dt_{2}dt_{3}\,+\,\dots \nonumber
\end{align}
where $L^{[r]'}$ refers to $r$ iterations of the bracket  $[.,.]'$ (see \Cref{Sec-bi-grading}) and $e(\tau)$ represents the number of consecutive subscripts whose order is reversed under the permutation $\tau$, i.e. $e(\tau)=\sharp\{i \in \{1, \dots, r-1\},\, \tau(i)>\tau(i+1)\}$.

\begin{fact}[Strichartz \cite{strichartz87}] \label{fact-integration} The map that sends $c$ to $\gamma(t)=\int_{[a,t]}^{*'} c(s)ds$ is a bijection between the set of piecewise continuous paths $c$ on $[a,b]$ and the set of continuous piecewise $C^1$ paths $\gamma$ on $[a,b]$. The inverse is given by the multiplicative derivative $\gamma\mapsto \partial \gamma$.  
\end{fact}

Recall from \Cref{Sec-diff-general} that a piecewise $C^1$ path $\gamma : [a,b]\rightarrow \kg$ is said to be $\LL$-horizontal if 
$$\partial \gamma(t) \in \Ad (t\YY)(\km^{(1)}+B)$$ 
for any $t$ outside of the finite set of discontinuity points of  $\partial \gamma$. 
Combining \Cref{fact-support} and  \Cref{fact-integration}, we get

\begin{lemme} \label{support-integral}
The support  of $\sigma_{1}$ is the closure of the set 
$$\left\{ \int^{*'}_{[0,1]}\Ad(t \YY) (u(t) +B) \,dt  \,:\, u :[0,1]\rightarrow \km^{(1)} \text{ piecewise continuous } \right\}.$$
\end{lemme}

This  formula has the advantage of reducing the computation of the support to the problem of solving a finite explicit family of polynomial equations in the variable $v$. We will use this point of view to show that walks on the upper triangular unipotent subgroup always have a fully supported limit distribution. 

We also present a discrete characterization of the support  of $\sigma_{1}$ that will be useful below to describe an example with non full support. Given vectors $u_{1}, \dots, u_{k}\in \tkg$, we set $\prod^{*'}_{1\leq i\leq k}u_{i}=u_{1}*'\dots*'u_{k}$. 

\begin{lemme}\label{support-sum}
The support of $\sigma_{1}$ is the closure of the set
$$\left\{ \prod^{*'}_{1\leq i\leq k} (u_{i}*'t_{i}(B+\YY)): \,\,k\geq 1, \, u_i\in \km^{(1)},\, t_{i}\geq 0, \,\sum_{i=1}^k t_{i}=1 \right\}*'(-\YY).$$
 \end{lemme}


\begin{proof}

\underline{Case $B=0$}. Call $\SS$ the support of $\sigma_{1}$ and  $\SS'$ the closure of the set above. 

We first check the inclusion $\SS\subseteq \SS'$. 
It is clear from the formula defining the multiplicative integral $\int^{*'}$ that for every smooth $u :[0,1]\rightarrow \km^{(1)}$,
\begin{align*}
\int^{*'}_{[0,1]}\Ad(t \YY) u(t) \,dt &=\prod^{*'}_{0\leq i\leq k-1}  \int^{*'}_{[\frac{i}{k},\frac{i+1}{k}]}\Ad(t \YY) u(t) \,dt \\
&= \prod^{*'}_{0\leq i\leq k-1} \left(\frac{1}{k} \Ad(\frac{i}{k} \YY)u(\frac{i}{k}) + o_{u}(k^{-1})\right)\\
&= \prod^{*'}_{0\leq i\leq k-1} \left( \Ad(\frac{i}{k} \YY)(\frac{1}{k}u(\frac{i}{k}) )\right)+ o_{u}(1)\\
&= \prod^{*'}_{0\leq i\leq k-1} \left(\frac{1}{k}u(\frac{i}{k}) *' \frac{\YY}{k}\right)*'(-\YY)+ o_{u}(1).
\end{align*}
Letting $k$ go to infinity, we get that $\int^{*'}_{[0,1]}\Ad(t \YY) u(t) \in \SS'$. By continuous approximation and the formula defining  $\int^{*'}$,  we still have $\int^{*'}_{[0,1]}\Ad(t \YY) u(t) \in \SS'$ if  $u$ is only piecewise-continuous. We then deduce from \Cref{support-integral} that $\SS\subseteq \SS'$.

For the reverse inclusion, one may assume that $t_{i}>0$ for all $i$. Let $0<\eps <\min_{i} t_{i}$. Set $\ot_{i}=t_{1}+\dots +t_{i-1}$ and  define a piecewise continuous function $u_{\eps}$ by 
$$u_{\eps}(t)   = \left\{
    \begin{array}{ll}
        \eps^{-1}u_{i} & \mbox{if } t\in [\ot_{i}, \ot_{i}+\eps) \\
        0 & \mbox{otherwise}
    \end{array}
\right.$$
Then $\int^{*'}_{[0,1]}\Ad(t \YY) u_{\eps}(t) \,dt \rightarrow \prod^{*'}_{1\leq i\leq k} (u_{i}*'t_{i}\YY)*'(-\YY)$ as $\eps$ goes to zero, yielding $\SS'\subseteq \SS$. 

\underline{General case}.  
Call $\W_{+}(t)$ a diffusion process with infinitesimal generator $\LL_{+}=\frac{1}{2}\sum_{i}\Ad(t(B+\YY))E_{i}^2$. It follows from  \Cref{homogeneisation} that the stochastic processes $\W_{+}(t)*'t(B+\YY)$ and $\W(t)*'t\YY$ have same distribution. As $\LL_{+}$ has no drift, we may apply the previous discussion to characterize the support of $\W_{+}(1)*'(B+\YY)$, in other terms $\SS*'\YY$, and the result follows. 
\end{proof}

\bigskip

\subsubsection{\large Example 1: the case of upper triangular matrices} \label{Sec-upper-triang}

We assume here that $(\kg, [.,.])$ satisfies the following double cancellation assumption 
\bigskip

\noindent ({\bf DC}) \emph{There exist vectors $v_{1}, \dots, v_{q}\in \kg$ spanning $\kg/[\kg, \kg]$  such that for all $1\leq i\leq q$,   $k\geq 0$, $w\in \kg$, one has $$[v_{i}, \ad(w)^kv_{i}]=0$$}and we show that the support of $\sigma_{1}$ must be $\kg$. Here the notation $\ad$ refers to the non-graded bracket $[.,.]$, i.e. $\ad(w)x=[w,x]$. Condition ({\bf DC})  clearly holds in every nilpotent Lie algebra with step at most $2$. As we show now it holds for the group of upper triangular unipotent matrices. 

\begin{lemme}
Assume $\kg$ is the Lie algebra of a maximal unipotent subgroup of a special linear group $SL_{n}(\R)$. Then $\kg$ satisfies  condition \emph{({\bf DC})}.
\end{lemme}

\begin{proof}
In this context, $\kg$ admits a vector space basis $(E_{i,j})_{1\leq i<j\leq q_{1}+1}$ with the bracket relations 
$$[E_{i,j}, E_{k,l}]= \1_{j=k}E_{i,l}  - \1_{i=l}E_{k,j}  $$
Setting $v_{i}=E_{i,i+1}$, the claim follows by direct computation. 


\end{proof}

The main result of the section is \Cref{horizontal-path-triang} below, which implies full support under condition ({\bf DC}), and will even yield a lower bound on $\sigma_{1}$ later in \Cref{Sec-lower-bound}. Given a  $\LL$-horizontal path $\gamma : I \rightarrow \kg$, we introduce its \emph{driving vector} $v_{\gamma}:I\rightarrow \km^{(1)}$ defined (outside a finite subset) by the relation $\partial \gamma(t)=\Ad(t\YY )(v_{\gamma}(t)+B)$. We call \emph{horizontal energy} of $\gamma$ the quantity 
$$E_{hor}(\gamma)=\int_{I} \|v_{\gamma}(t)\|^2\,dt.$$
 It is related to the classical notion of energy, but better suited for computations and more meaningful in our context. Indeed it expresses how hard it is for the $\LL$-diffusion to follow $\gamma$.  We fix a left-invariant Riemannian metric on $(\kg, *')$ and denote by $\normp{.}$ the distance to the origin.

\begin{proposition}\label{horizontal-path-triang} Assume \emph{({\bf DC})}.
Then for every point $x\in \kg$, there exists a continuous, piecewise $C^1$,  $\LL$-horizontal path $\gamma : [0,1]\rightarrow \kg$ such that $\gamma(0)=0$, $\gamma(1)=x$. Moreover, one may choose $\gamma$ such  that $E_{hor}(\gamma)\leq C \normp{x}^2 +C$ where $C>0$ is a constant independent from $x$.
 \end{proposition}

Combining  \Cref{horizontal-path-triang} and \Cref{fact-support} we get 

\begin{corollary}[Full support] \label{DC-full-support} Under  condition \emph{({\bf DC})}, the support of $\sigma_{1}$  is $\kg$.
\end{corollary}

\bigskip

We now engage in the proof of \Cref{horizontal-path-triang}. In accordance with condition ({\bf DC}), we fix $v_{1}, \dots, v_{q_{1}} \in \kg$ projecting to a basis of $\kg/[\kg, \kg]$ and satisfying $[v_{i}, \ad(w)^k v_{i}]=0$ for all $w\in \kg$, $k\geq 0$. 

\begin{lemme} In order to prove \Cref{horizontal-path-triang}, one may assume  
$$v_{1}, \dots, v_{q_{1}} \in \km^{(1)} \,\,\,\,\,\,\,\,\,\,\,\,\,\,\,\,\,\,\,\,\,\,\,\,\,\,\,\,B=0.$$ 
\end{lemme}

\begin{proof}
Let $\kg=\oplus_{b}\mathring{\km}^{(b)}$ be some other weight decomposition of $\kg$ such that $\mathring{\km}^{(1)}=\oplus \R v_{i}$, and $\phi : (\tkg, [.,.]')\rightarrow (\tkg, \mathring{[.,.]}')$ the isomorphism of graded structures from \Cref{changing-dec}. The map $\phi$ sends a left invariant metric on $(\kg,*')$ to a left-invariant metric on $(\kg,\mathring{*}')$, and any $\LL$-horizontal path $\gamma$ on $(\kg,*')$ to the $\phi(\LL)$-horizontal path $\phi \circ \gamma$ on $(\kg,\mathring{*}')$, which also satisfies $E_{hor}(\gamma)\ll_{\phi} E_{hor}(\phi \circ \gamma)\ll_{\phi} E_{hor}(\gamma)$. Hence, it is enough to prove \Cref{horizontal-path-triang} for the process generated by $\phi(\LL)$, in other words one may assume $v_{1}, \dots, v_{q_{1}} \in \km^{(1)}$. 

To remove the drift, observe that $\gamma$ is $\LL$-horizontal if and only if the path $c : t\mapsto \gamma(t)*' t\YY *' -t(B+\YY)$ is $\LL_{+}$-horizontal where $\LL_{+}=\frac{1}{2}\sum_{i}(\Ad(t(\YY+B))E_{i})^2$. A direct computation shows that $\gamma$ and $c$ have the same driving vector (for their respective generators $\LL$, $\LL_{+}$), in particular they have the same horizontal energy. We infer that it  is sufficient to prove  \Cref{horizontal-path-triang} for the process generated by $\LL_{+}$,  in other words one may assume $B=0$. 
\end{proof}

Say that a vector $w\in \kg$ satisfies the \emph{property} (P) if for every non-trivial interval $[a,b]\subseteq \R$, there exists a smooth function $u: [a,b]\rightarrow \km^{(1)}$ such that for all $r\in \R$,
\begin{align}\label{eq-piece}
\int^{*'}_{[a,b]} \Ad(t\YY)(r u(t)) \,dt =r w .
\end{align}

\begin{lemme} \label{piece-m1}
The elements $v_{1}, \dots, v_{q_{1}}$ satisfy the property \emph{(P)}.
\end{lemme}

We recall that throughout the paper, $\Ad$ refers implicitly to the graded structure $*'$ at study, i.e. $\Ad(y)x=y*'x*'-y= \sum_{n=0}^s (n!)^{-1}L^{[n+1]'}(y^{\otimes n}, x)$ for all $x,y\in \tkg$.

\begin{proof}
Let $v\in \{v_{1}, \dots, v_{q_{1}}\}$. 
The crucial\footnote{Indeed, we only use ({\bf DC}) to get \eqref{DC-commute} for every $v$ in some basis of $\kg/[\kg, \kg]$ lifted to $\kg$.} input of  the double cancellation property is that it implies that for any $t_{0}, t_{1}\in \R$,  
\begin{align} \label{DC-commute}
[\Ad(t_{0}\YY)v, \Ad(t_{1}\YY)v]'=0
\end{align}
 Indeed, one may assume $t_{0}=0$, then expanding $\Ad(t_{1}\YY)$ it is enough to check that $L^{[n+2]'}(v, \YY^{\otimes n}, v)=0$ for any $n\geq 0$. As $v\in \km^{(1)}$, $\YY\in \km^{(2)}$, we have by definition that $L^{[n+2]'}(v, \YY^{\otimes n}, v)$ is the projection to $\km^{(2+2n)}$ of the non-graded bracket $L^{[n+2]}(v, \YY^{\otimes n}, v)$ which is equal to zero by assumption ({\bf DC}). Hence \eqref{DC-commute}.

Now fix  a non-trivial interval $[a,b]\subseteq \R$. We look for $u$ in the form $u(t)=\alpha(t)v$ where $\alpha: [a,b]\rightarrow \R$ is a smooth function. For any such $u$, it follows from Strichartz' formula \eqref{strich-formula}  defining $\int^{*'}$  and \eqref{DC-commute} above that \begin{align*}
\int^{*'}_{[a,b]}\Ad(t\YY)(u(t))\,dt &=\int_{[a,b]}\Ad(t\YY)(u(t))\,dt \\
&= \sum_{n=0}^s  \left(\int_{[a,b]} \alpha(t) t^n \,dt\right)  \,\,(n!)^{-1} L^{[n+1]'}(\YY^{\otimes n}, v)
\end{align*}
where the first line converts the multiplicative integral into an additive integral, in the classical sense. 
To conclude that $v$ satisfies (P), we just need to choose $\alpha$ such that 
$$ \int_{[a,b]} \alpha(t) t^n \,dt = \left\{
    \begin{array}{ll}
        1 & \mbox{if } n=0 \\
        0 & \mbox{if } n \in \{1, \dots, s\}
    \end{array}
\right.
$$
The existence of such  $\alpha$ is clear by  independence of the linear forms $\alpha \mapsto  \int_{[a,b]} \alpha(t) t^n dt$ on $C^\infty([a,b], \R)$. 
\end{proof}

We  use \Cref{piece-m1} to generate more elements that satisfy (P). 

\begin{lemme}[horizontal pieces] \label{piece-ab}
There exist elements $w_{1}, \dots, w_{p}\in \kg$ such that $(w_{i})$ project to a basis of $\kg/[\kg, \kg]'$ and each $w_{i}$ satisfies the property \emph{(P)}.
\end{lemme}

\begin{proof} This follows from \Cref{piece-m1} and the two next observations:  
\begin{itemize}
\item If $w$ satisfies (P), then $\Ad(t_{0}\YY)w$ also satisfies (P)  for any $t_{0}\in \R$.

\item  The sub Lie algebra of $(\kg, [.,.]')$ generated by $\{\Ad(t \YY)\km^{(1)}, t\in \R\}$ contains $\km^{(1)}$ and is stable by $[\YY, .]'$, so it must be $\kg$ by \Cref{prehypo}. It follows that the abelianization $\kg/[\kg, \kg]'$ is spanned by the projection of $\{\Ad(t \YY)v_{i}, t\in \R, i\leq q_{1}\}$. 
\end{itemize}
\end{proof}

The next lemma claims that we can combine the horizontal paths generating the $\R w_{i}$'s via the property (P) in order to reach any point of $\kg$, with a total path of length  comparable to the distance to the origin for a $*'$-left invariant metric. 
\begin{lemme}[combining pieces]  \label{combining-pieces}
Given elements $w_{1}, \dots, w_{p} \in \kg$ such that $(w_{i})$ projects to a basis of $\kg/[\kg, \kg]'$, there exists $R_{0}, n_{0}\geq 1$ and a sequence of subscripts $i_{1}, \dots, i_{n_{0}} \in \{1, \dots, p\}$ such that every $x\in \kg$ can be written 
$$x= r_{1} w_{i_{1}}*'\dots *' r_{n_{0}} w_{i_{n_{0}}} $$
for some parameters $r_{i}$ satisfying $\max |r_{i}|\leq R_{0} \normp{x}+R_0.$
\end{lemme}

\begin{proof} 
Straightforward induction on the step of the Lie algebra $(\kg, [.,.]')$.
\end{proof}

\begin{proof}[End of proof of \Cref{horizontal-path-triang}] Combine  \Cref{piece-ab} and \Cref{combining-pieces}. 

\end{proof}

\subsubsection{\large Example 2: the case of the free nilpotent Lie algebra} \label{Sec-free}

We describe situations for which the limiting distribution $\sigma_{1}$ does not have full support in $\kg$. We start with an example of Lie algebra $\kg$ for which \emph{any biased limiting distribution} (i.e. $\Xab\neq 0 \mod [\kg, \kg] $) is properly supported. Note that the restriction to biased processes is necessary because for $\Xab = 0 \mod [\kg, \kg]$, full support is guaranteed by \cite{varopoulos-saloff-coulhon92}.  Combined with the central limit theorem established in \Cref{Sec-global-thm}, the following statement implies \Cref{never-full-thm} from the Introduction.



\begin{proposition}[Never full support]\label{free-non-full}
Assume that $\kg$ is a free nilpotent Lie algebra of step at least $3$, and $\Xab \in \kg/[\kg, \kg]$ is non zero.   Then the support of  
$\sigma_{1}$
does not contain $0$ in its interior. 
\end{proposition}

\noindent{\bf Remark}. For Lie algebras of step at most $2$, full support is assured by \Cref{DC-full-support}.

\bigskip
The strategy to prove \Cref{free-non-full} is to find a concrete example of step-$3$ nilpotent Lie algebra with a limiting distribution that is not fully supported, then use the assumption that $\kg$ is free to project $\kg$ to this example in an equivariant way. To carry out this plan, we introduce  $(\kl, [.,.])$ the so-called \emph{filiform nilpotent Lie algebra} of step $3$, defined by  $\kl=\oplus_{i=1}^4\R e_{i}$ where 
\begin{align*}
&[e_{1},e_{2}]=e_{3}, \,\,[e_{1}, e_{3}]=e_{4}, \,\,[e_{1}, e_{4}]=0 \\
&[e_{2},e_{3}]= [e_{2}, e_{4}]= [e_{3}, e_{4}]=0 .
\end{align*}

We shall denote  $A=e_{1}, T=e_{2}$ (as one can realise  $\kl$ in the affine group of $\R^3$, and $A$ corresponds to a linear transformation, while  $T$ is a translation). We let $\XXkl =T \mod [\kl, \kl]$ and choose the adapted weight decomposition
$$\widetilde{\kl}=\underbrace{\R A\oplus \R T}_{\km^{(1)}_{\kl}}\oplus \underbrace{\R \chi_\kl}_{\tkm^{(2)}_{\kl}} \oplus \underbrace{\R [A,T]}_{\km^{(3)}_{\kl}}\oplus \underbrace{\R [A,[A,T]]}_{\km^{(4)}_{\kl}}$$
where $[\chi_\kl, .]=[T,.]$ on $\kl$. Observe that the associated graded bracket $[.,.]'$ coincides with $[.,.]$ except that  $T$ becomes central.

\begin{lemme}[Explicit support]\label{thread-support}
Let $\sigma_\kl$ be the diffusion law on $(\kl, *')$ associated to an operator $\LL_{\kl}$ as in \eqref{limit-diff-gen}. Then the support of the time-$1$ distribution is 
$$\supp \,\sigma_{\kl\, 1}=\left\{s_{1}A + s_{2}T + s_{3}e_{3} + s_{4}e_{4} \,:\, 2s_{4}\geq(s_{1}+s_{3})s_{3} \right\}. $$
\end{lemme}

\noindent\emph{Remark}.
Of course, we may replace the condition $\XXkl=T$ by $\XXkl\in \R^*T$, and multiply $T$ appropriately in the support formula. However, if $\XXkl \notin \R^* T \mod [\kl, \kl]$, then the support is always full. We  prove below that the limiting distribution is Gaussian (in the classical Euclidean sense) if and only if $\XXkl \notin \R T \mod [\kl, \kl]$, see \Cref{gaussian-criterion}.

\begin{proof}

By \Cref{support-sum} and the facts that $T$ is  central in $(\widetilde{\kl}, *')$ and $\km^{(2)}_{\kl}=0$, the support of $\sigma_{\kl\, 1}$ is the closure of the subset of $\kl$ given by  
$$\R T *' \left \{\prod_{1\leq i\leq N}^{*'}(\alpha_{i}A*'t_{i}\chi_{\kl}),\,\,t_{i}\geq 0, \sum t_{i}=1, \alpha_{i}\in\R \right\}*'(-\chi_{\kl}) .$$
Setting $\overline{\alpha_{i}}=\sum_{j=1}^i\alpha_{j}$, one has 
\begin{align*}
\prod_{1\leq i\leq N}^{*'} (\alpha_{i}A*'t_{i}\chi_{\kl})  &=\left(\prod_{1\leq i\leq N}^{*'} \Ad(\oalpha_{i}A)(t_{i}\chi_{\kl})\right)*'(\oalpha_{N}A).
\end{align*}
The terms in the product commute because $\widetilde{\kl}^{(2)}$ is abelian. Hence, 
\begin{align*}
\prod_{1\leq i\leq N}^{*'}(\alpha_{i}A*'t_{i}\chi_{\kl}) &=\left(\sum_{i=1}^N \Ad(\oalpha_{i}A)(t_{i}\chi_{\kl})\right)*'(\oalpha_{N}A)\\
&=\left(\sum_{i=1}^Nt_{i}\chi_{\kl}  +\sum_{i=1}^N \oalpha_{i}t_{i} [A,\chi_{\kl}]' +\frac{1}{2}\sum_{i=1}^N \oalpha_{i}^2t_{i} [A[,A,\chi_{\kl}]]'  \right)*'(\oalpha_{N}A)\\
&=(\frac{1}{2}\sum_{i=1}^N \oalpha_{i}^2t_{i}) [A,[A,\chi_{\kl}]]' \,*'\,(\sum_{i=1}^N \oalpha_{i}t_{i}) [A,\chi_{\kl}]'  \,*'(\,\sum_{i=1}^N t_{i})\chi_{\kl} \,*'\,(\oalpha_{N}A).
\end{align*}
The above computation justifies that the support of $\sigma_{\kl\, 1}$ is exactly 
$$\left\{ r_{4}e_{4}*' r_{3}e_{3}*' \chi_{\kl} *'r_{0}A*'(-\chi_{\kl})\,:\, 2r_{4}\geq r_{3}^2 \right\} \oplus \R T .$$
We can formulate this subset of $\kl$ with linear combinations of the $(e_{i})$'s:
\begin{align*}
r_{4}e_{4}*' r_{3}e_{3}*' \chi_{\kl} *'r_{0}A*'(-\chi_{\kl}) &=r_{4}e_{4}*' r_{3}e_{3}*' (r_{0}A-r_{0}e_{3})\\
&=\underbrace{(r_{4} -\frac{r_{0}r_{3}}{2})}_{s_{4}}e_{4}+ \underbrace{(r_{3}-r_{0})}_{s_{3}}e_{3}+ \underbrace{r_{0}}_{s_{1}}A.
\end{align*}
The condition $2r_{4}\geq r_3^2 $ amounts to
$2s_{4} \geq (s_{1}+s_{3})s_{3} $
so the support of $\sigma_{\kl\, 1}$ is 
$$\left\{s_{1}A + s_{2}T + s_{3}e_{3} + s_{4}e_{4} \,:\, 2s_{4} \geq (s_{1}+s_{3})s_{3} \right\}  .$$
\end{proof}

\noindent\emph{Remark}. Setting $t_{0}=0$, $2t_{4}=t_{3}^2$, we see that  \emph{$\supp \,\sigma_{\kl\, 1}$ is not stable by the dilation $D_{r}$ for $r>1$}. However, it is stable by $D_{r}$ for $r\leq 1$. This second observation is true in general for any diffusion  as in  \Cref{Sec-inv-reg} with $B=0$. It comes from the equality
$$\sigma_{1} = D_{\sqrt{r}} \sigma_{1}  *'\Ad(r\YY) D_{\sqrt{1-r}}\sigma_{1}  $$
justified by \Cref{density-inv}, and the fact that the support of $\sigma_{1}$ contains $0$ (in its boundary at least) by \Cref{support-integral}. 

\bigskip

\noindent\emph{Remark}. In \cite{benoist21}, Benoist shows that for random walks on the discrete step-$3$ filiform nilpotent group, there may exist positive harmonic functions that are neither characters nor translates of harmonic functions  induced from characters. This is in contrast with the Heisenberg case, and the argument also relies on the analysis of the central component. 


\bigskip

We may now conclude the proof of  \Cref{free-non-full}.
\begin{proof}[Proof of \Cref{free-non-full}]
 As $\kg$ is free of step at least $3$ and $\Xab$ is non zero, there must exist a surjective morphism of Lie algebras $\pi: (\kg, [.,.])\rightarrow (\kl, [.,.])$ such that $\pi(\Xab)= T \mod [\kl, \kl]$. Observe that $\pi$ sends the weight filtration of $(\kg, \Xab)$ to that of $(\kl, T)$, meaning $\pi(\kg^{(b)})=\kl^{(b)}$ for every $b$. However, if $\kl=\oplus\km^{(b)}_{\kl}$ is the weight decomposition from \Cref{thread-support},  it is not necessarily true that $\pi(\km^{(b)})=\km^{(b)}_{\kl}$ or  even that $(\pi(\km^{(b)}))$  is a direct sum decomposition of $\kl$. 
 
  To reduce to this situation, we introduce some other weight decomposition $\kg=\oplus \mathring{\km}^{(b)}$ such that $\pi(\mathring{\km}^{(b)})=\km^{(b)}_{\kl}$ and set  $\phi : \tkg\rightarrow \tkg$ the unique linear map such that $\phi : \tkm^{(b)} \rightarrow \widetilde{\mathring{\km}}^{(b)}$ is the isomorphism  that factorizes via the identity on the vector space $\kg^{(b)}/\kg^{(b+1)}$. Call $\mathring{*}'$ the induced graded product on $\kg$, set $\mathring{E}_{i}=\phi(E_{i})$, $\mathring{B}=\phi(B)$,  $\mathring{\YY}=\phi(\YY)$, and $\mathring{\sigma}$ the induced diffusion law on $(\kg, \mathring{*}')$ as in \eqref{limit-diff-gen}.  We know from \Cref{changing-dec}  that $$\sigma =\phi^{-1} \circ \mathring{\sigma}.$$
 In particular, $\supp\,\sigma_{1}= \phi^{-1}(\supp  \,\mathring{\sigma}_{1})$, so it is sufficient to check the claim for $\mathring{\sigma}$.

Observe now that the condition $\pi(\mathring{\km}^{(b)})=\km^{(b)}_{\kl}$ guarantees that $p$ extends to a morphism from $(\tkg,\mathring{*}')$ to $(\widetilde{\kl}, *')$ by setting $p(\mathring{\chi})=\chi_{\kl}$. As $\km^{(2)}_{\kl}=\{0\}$, we have $p(\mathring{\YY})=\chi_{\kl}$, and it follows that 
$p\circ \mathring{\sigma}$ is the diffusion law generated by 
$$\frac{1}{2}\sum_{i=1}^q  \left(\Ad(t\chi_{\kl})p(\mathring{E}_{i})\right)^2+ \Ad(t\chi_{\kl})p(\mathring{B}).$$
By \Cref{thread-support}, the support of $p \circ \mathring{\sigma}_{1}$  does not contain $0$ in its interior, whence that of $\mathring{\sigma}_{1}$ does not either, and this concludes the proof.

\end{proof}

\subsection{Discussion on a lower bound} \label{Sec-lower-bound}

We keep the notations of \Cref{Sec-inv-reg} and address the question of a lower bound for the distribution $\sigma_{1}$, that is  the distribution at time $1$ of the diffusion law $\sigma$ associated to  $\LL$ from \cref{limit-diff-gen}. 
The fact that $\sigma_{1}$ may not have full support on $\kg$ yields some difficulties. They lead us to formulate a lower bound in terms of $K$-regular chain, where $K$ is a compact in the interior of the support of $\sigma_{1}$, see \Cref{K-lower-bound}.  When $\sigma_{1}$ has full support, we interpret \Cref{K-lower-bound} in terms of the distance determined by horizontal paths, or even a left $*'$-invariant metric on $\kg$ if $\kg$ satisfies the condition ({\bf DC}) from \Cref{Sec-upper-triang}. Finally, we use \Cref{K-lower-bound} to characterize the situations where $\sigma_{1}$ is a Gaussian distribution, in the Euclidean sense.


\bigskip
To obtain a general lower bound, we follow the classical strategy of iterating a suitable Harnack inequality. Establishing such an inequality requires more information on the set of arrival points of (continuous piecewise $C^1$) horizontal paths. We start with a lemma  telling us that this set has non empty interior, even if we only consider paths of bounded speed.


\begin{lemme}\label{Chow} Let $0<t_{0}<t_{1}$.  The set of arrival points $\gamma(t_{1})$,  for $\LL$-horizontal paths  $\gamma : [t_{0},t_{1}] \rightarrow \kg$ starting at $\gamma(t_{0})=0$ and whose derivative satisfies  $\|\partial \gamma \|_{\infty} \leq 1$, contains an open set of $\kg$. 
\end{lemme}

\begin{proof} Recall from the remark following \Cref{homogeneisation} that a path $\gamma$ is $\LL$-horizontal if and only if $t\mapsto \gamma(t)*'t\YY*'(-t(B+\YY))$ is $\LL_{+}$-horizontal, where $\LL_{+}:=\frac{1}{2}\sum_{i=1}^{q}\Ad(t(B+\YY))E_{i}^2$. We may thus assume that $B=0$.  Denote by $\tX$ the vector field on $\R\times \kg$ given by $( \frac{d}{dt}, \Ad(t\YY)X)$. We saw in \Cref{hypo} that the family $\{\tX,\, X\in \km^{(1)}\}$  is bracket-generating on $\R\times \kg$. Moreover, the associated flows project into horizontal paths. The result now follows from Chow's theorem \cite{bellaiche96}.
\end{proof}

Denote by $\SS$ the support of $\sigma_{1}$, and $\oSS$ the interior of $\SS$. We now use the previous lemma to show that any point in $\oSS$ can be reached by a horizontal path defined on the time interval $[0, 1]$ and starting at $0$. Moreover, the length of these paths is uniformly bounded if the set of prescribed arrival points is compact in $\oSS$.  

\begin{lemme}[Accessibility] \label{accessibility}
Let $K\subseteq \oSS$ be a compact subset. For  $R>0$ large enough, for  any $x\in K$, there exists $\gamma :[0,1]\rightarrow \kg$  $\LL$-horizontal, such that  $\|\partial \gamma \|_{\infty} \leq R$ and $\gamma(0)=0$, $\gamma(1)=x$. 
\end{lemme}

The idea is to use the characterization of $\SS$ as the closure of arrival points of horizontal paths  on $[0, 1]$ starting at $0$  (\Cref{fact-support}) to generate a path from $0$ to some $y$ very close to $x$ on the time interval $[0, 1-r]$ with $r$ small.  Then we complete the path from $y$ to $x$, or any $x'$ in a neighborhood of $x$, on the remaining time $[1-r, 1]$ using \Cref{Chow}. 
\begin{proof} 

For $r\in [0, 1]$, set $\Pc_{1-r}= \{\gamma(1-r) \,| \, \gamma:  [0, 1-r]\rightarrow \kg \text{ $\LL$-horizontal}, \gamma(0)=0\}$. Recall from \Cref{fact-support} and \Cref{density-inv} that $\Pc_{1-r}$ is dense in $D_{\sqrt{1-r}}\SS$. Fix $U\subseteq \kg$ a compact symmetric neighborhood of $0$ such that $K*'U\subseteq \oSS$, then there exists $r_{0}<1$ such that $D_{\sqrt{1-r}}\oSS\supseteq K*'U$ for $0\leq r\leq  r_{0}$.  It follows that for any open set $V\subseteq U$, one has $\Pc_{1-r} *' V \supseteq K$, and by compactness, $K$ is covered by a finite number of translates $\gamma_{i}(1-r) *' V$ where $\gamma_{i}(1-r)\in \Pc_{r}$. 

In order to conclude, we only need to show that for $r<1$ large enough, any point in some open set $V\subseteq U$ can be reached by a $\LL$-horizontal path on the interval $[1-r, 1]$ starting from $0$ and of bounded speed.  This is achieved by \Cref{Chow}.
\end{proof}

We now combine the result of Kogoj-Polidoro from \Cref{fact-hypo} and the previous analysis to deduce the following Harnack Inequality. 
\begin{lemme}[Harnack Inequality] \label{Harnack-1}
Let $K\subseteq \oSS$ be a compact subset. For all $\eps_{0}>0$ small enough, there exists $C>0$ such that for every  non-negative function $v: (\frac{\eps_{0}}{2}, 2) \times \kg \rightarrow \R^+$ solution of $\partial_{t}- \frac{1}{2}\sum_{i=1}^{q}  \left(\Ad(t\YY) E_{i}\right)^2 + \Ad(t\YY)B
=0$, we have 
$$v(\eps_{0}, 0) \leq C \inf_{x\in K} v(1, x).  $$
\end{lemme}

\begin{proof} Recall from \Cref{hypo} that the vector fields $\partial_{t}, \left(\Ad(t\YY) E_{i}\right)^2,\Ad(t\YY)B$ satisfy the H\"ormander condition from \Cref{fact-hypo}. Let $\eps_{0}>0$, let $0\in U\subseteq \kg$ an open ball, denote by $\Pc$ the set of couples $(\tau,x)\in (\eps_{0}, 2) \times U$ such that $x$ is accessible by a $U$-valued $\LL$-horizontal path on $[\eps_{0}, \tau]$ with starting point $0$. 
By \Cref{accessibility} and using dilations, one may choose $\eps_{0}>0$ small and $U$ large so that $\{1\}\times K$ is included in the interior of the closure of $\Pc$. 
Now the result  follows by application of the Harnack Inequality mentioned in \Cref{fact-harnack}. 
\end{proof}

We combine the previous Harnack Inequality with invariance properties of $\LL$ to  obtain the following.
\begin{lemme}[Harnack Inequality on a space-time cone] \label{Harnack-spacetime}
Let $K\subseteq \oSS$ be a compact subset, let $\eps>0$. There exists $R>1$ such that for every  non-negative solution $v: \R_{>0} \times \kg \rightarrow \R^+$  of $\partial_{t}- \frac{1}{2}\sum_{i=1}^{q}  \left(\Ad(t\YY) E_{i}\right)^2 + \Ad(t\YY)B
=0$, for all $t\in [\eps, +\infty)$, $r\in [0, 1]$, $x\in \kg$, we have 
$$v(t, x) \leq R \inf_{y\in \Ad(t\YY)D_{\sqrt{r}} K} v(t+r, x*'y) . $$
\end{lemme}

\begin{proof} Observe first that if $v$ defined $I\times \kg$ is a solution $\partial_{t}- \frac{1}{2}\sum_{i=1}^{q}  \left(\Ad(t\YY) E_{i}\right)^2 + \Ad(t\YY)B
=0$, then for any $x_{0}\in \kg$, $\tau_{0}\in \R$, $r_{0}>0$, it is also the case of the functions
\begin{itemize}
\item  $v(t, x_{0} *'x)$ defined on $(t,x)\in I \times \kg$
\item $v(\tau_{0}+t, \Ad(\tau_{0} \YY)x)$  defined on $(t,x)\in (I-\tau)\times \kg$
\item   $v(r_{0}t, D_{\sqrt{r_{0}}}x)$  defined on  $(t,x)\in r_{0}^{-1}I\times \kg$
\end{itemize}
Combining these invariance properties, we get the solution 
$$v(\tau_{0}+r_{0}t, \, x_{0}*'\Ad(\tau_{0}\YY)D_{\sqrt{r_{0}}} x) \text{  defined on  } (t,x)\in (r^{-1}_{0}I-\tau_{0})\times \kg.$$
Let $K'\subseteq \oSS$  be a compact neighborhood of $K$. Choose $\eps_{0} \in (0, \frac{\eps}{1+\eps})$ as in \Cref{Harnack-1} for the compact $K'$, and denote by $C>1$ the associated constant. It follows that for any  non-negative solution $v: \R_{>0} \times \kg \rightarrow \R^+$  of $\partial_{t}- \frac{1}{2}\sum_{i=1}^{q}  \left(\Ad(t\YY) E_{i}\right)^2 + \Ad(t\YY)B =0$, any $x_{0}\in \kg$, $\tau_{0},r_{0}>0$, one has 
\begin{align} \label{Harnack-spacetime-eq}
v(\tau_{0}+r_{0}\eps_{0}, \, x_{0})\leq C \inf_{y\in \Ad(\tau_{0}\YY)D_{\sqrt{r_{0}}} K'}v(\tau_{0}+r_{0}, \, x_{0}*'y) .
\end{align} 
Let $t\in [\eps, +\infty)$, $r\in [0, 1]$. We now choose $(\tau_{0}, r_{0})$ so that $\tau_{0}+r_{0}\eps_{0}=t$ and $\tau_{0}+r_{0}=t+r$, i.e. we set  $\tau_{0}=t-r\frac{\eps_{0}}{1-\eps_{0}}$, $r_{0}=\frac{r}{1-\eps_{0}}$. The domain restrictions on $(t,r)$ guarantee that $\tau_{0}>0$.  Plugging into \eqref{Harnack-spacetime-eq} yields 
$$v(t, \, x_{0})\leq C \inf_{y\in \Ad(\tau_{0}\YY)D_{\sqrt{r_{0}}} K'}v(t+r, \, x_{0}*'y) .$$
The result now follows by choosing $\eps_{0}$ small enough from the start so that 
 $\Ad(\tau_{0}\YY)D_{\sqrt{r_{0}}} K'$ contains $\Ad(t\YY)D_{\sqrt{r}}  K'$. Indeed, this condition can be satisfied independently of $(t,r)$ because it  amounts to the inclusion
 $ K\subseteq D_{(1-\eps_{0})^{-1/2}} \Ad(-\eps_{0}\YY) K'$.


\end{proof}

We  now deduce a general lower bound for the density $\sigma_{1}$. Given a compact subset $K\subseteq \oSS$, say that a finite sequence $(t_{i}, x_{i})_{i=0, \dots, n}$ in $\R_{>0}\times \kg$ is a \emph{$K$-regular chain} if for all $i\leq n-1$, one has $t_{i}<t_{i+1}$ and
$$x^{-1}_{i}*'x_{i+1} \in \Ad(t_{i}\YY) D_{\sqrt{t_{i+1}-t_{i}}}K  .$$
Given $a<b\in \R_{>0}$ and $x, y\in \kg$, set $E_{K}(a,x, b,y)$ the smallest integer $n\geq 1$ such that there exists a $K$-regular chain $(t_{i}, x_{i})_{i=0, \dots, n}$ from $(a,x)$ to $(b,y)$.

An iterated application of \Cref{Harnack-spacetime} yields the following lower bound on $\sigma_{1}$. Recall  for $t>0$ the notation $\sigma_{t}=u(t,x)dx$. 
\begin{theorem}[General lower bound]  \label{K-lower-bound} Let $K\subseteq \oSS$ be a compact subset, let $1>\eps>0$. There exists $R>1$ such that for all $x,y\in \kg$, 
$$u(\eps, x) \leq R^{E_{K}(\eps, x, 1, y)} u(1,y) .$$
\end{theorem}

\Cref{K-lower-bound} has  several concrete consequences, the first of which is the following: 

\begin{corollary}\label{positive-interior} The density $\sigma_{1}$ is positive on the interior of its support: $$\forall x\in \oSS, \,\,\,\,u(1,x)>0.  $$
\end{corollary}

\begin{proof} Choose the compact set $K\subseteq \oSS$ such that $x\in \mathring{K}$. There is a fixed small neighborhood $U$ of $x$ that is included in $\Ad(\eps \YY)D_{\sqrt{1-\eps}}\mathring{K}$ for any $\eps>0$ small enough. By \Cref{density-inv} taken with $r=\eps$, we see that for small $\eps$ the support of $u(\eps,\cdot)$ will intersect any given neighbourhood of $0$. In particular, one may choose such  $\eps$ so that for some $x_{0}\in \{u(\eps, .)>0\}$, one has $x^{-1}_{0}*'x\in U$. Now $(\eps, x_{0})$, $(1, x)$ is a $K$-regular chain, and the result follows from \Cref{K-lower-bound}.
\end{proof}

In case $\sigma_{1}$ has full support, one may compare $E_{K}$ with the energy functional determined by $\LL$-horizontal paths. More precisely, given $a<b\in \R_{>0}$ and $x, y\in \kg$, define the \emph{horizontal cost} $E_{hor}(a,x, b,y)$ to be the infimum of the energy $E_{hor}(\gamma)=\int_{[a,b]}\|v_{\gamma}\|^2 $ for $\gamma :[a,b]\rightarrow \kg$ continuous, piecewise $C^1$, $\LL$-horizontal,  such that $\gamma(a)=x$, $\gamma(b)=y$. 

\begin{corollary}[Lower bound when full support]  \label{full-lower-bound} Assume $\supp \,\sigma_{1}=\kg$. For all $\eps>0$, there exists $\alpha, A>0$ such that for all $x\in \kg$, 
$$u(1,x)  \geq \alpha \exp\left(-A \,E_{hor}(\eps, 0, 1, x) \right). $$
\end{corollary}

 \Cref{full-lower-bound}  follows from the combination of \Cref{K-lower-bound}, \Cref{positive-interior}  and the next lemma.
\begin{lemme}\label{dhor-EK} Let $K$ be a large enough compact subset of $\kg$. For $a<b\in \R$,  $x,y\in \kg$, 
$$ E_{K}(a,x, b,y) \leq  E_{hor}(a,x,b,y) +1 .$$
\end{lemme}

\begin{proof} 
It is sufficient to check that for some $K$ compact, every horizontal path $\gamma$ on an interval $[a,b]$ whose energy is bounded by $1$ is $K$-regular, in the sense that
$$\gamma(b)\in \gamma(a)*'\Ad(a\YY )D_{\sqrt{b-a}}K .$$

\noindent\underline{Case $B=0$}.  
We may replace $\gamma$ by $\Ad(-a\chi )( \gamma(a)^{-1}*' \gamma)$ which is horizontal on $[0, b-a]$ and has same energy as  $\gamma$. This allows to assume $a=0$. Then we need check that $D_{\frac{1}{\sqrt{b}}}\gamma(b)$ is  bounded uniformly in  $\gamma$. We set $c(t)=D_{\frac{1}{\sqrt{b}}}\gamma(bt)$. It is a continuous, piecewise-$C^1$ path defined on $[0,1]$. Moreover, using that $D_{\frac{1}{\sqrt{b}}}$ is an automorphism for $*'$, that $B=0$ and  $v_{\gamma}\in \km^{(1)}$, we get 
\begin{align*}
\partial c(t) &= D_{\frac{1}{\sqrt{b}}} b \lim_{\eps\to 0}\frac{\gamma(bt)^{-1}*' \gamma(bt+b\eps)}{b\eps}\\
&=D_{\frac{1}{\sqrt{b}}} b \Ad(bt \YY)v_{\gamma}(bt)\\
&=  \Ad(t \YY) \sqrt{b}\, v_{\gamma}(bt).
\end{align*}
Hence, the path $c$ is horizontal on $[0,1]$ with driving vector $v_{c}(t)=\sqrt{b}\, v_{\gamma}(bt)$. A change of variable then yields that $c$ and $\gamma$ have same energy, in particular $E_{hor}(c)\leq 1$. It follows that $c(1)$ is bounded, independently of $\gamma$, thus concluding the proof in the case $B=0$.

\noindent\underline{General case}. Define the path $\gamma_{0}(t)=\gamma(t)*'t\YY*'(-tB -t\YY)$ and observe $\gamma_{0}$ is $\LL_{0}$-horizontal for the (driftless) generator $\LL_{0}= \frac{1}{2}\sum_{i=1}^q (\Ad(t(\YY+B) ) E_{i})^2$, with same driving vector as $\gamma$.  By the first case examined above, we know that 
$$\gamma_{0}(b)\in \gamma_{0}(a)*'\Ad(a(\YY +B))D_{\sqrt{b-a}}K_{0}$$
for some compact $K_{0}$ independent of $\gamma$. Direct computation shows that this implies 
$$\gamma(b)\in \gamma (a)*'\Ad(a\YY)D_{\sqrt{b-a}}K$$
where $K=K_{0}*'(-Y+B)*'Y$, thus concluding the proof. 
\end{proof}

Applying \Cref{full-lower-bound} to Lie algebras satisfying the double cancelling assumption ({\bf DC}), we obtain

\begin{corollary}[Lower bound when ({\bf DC})]\label{lower-cor}
Assume $\kg$ satisfies the assumption \emph{({\bf DC})} from \Cref{Sec-support}. There exist constants $\alpha, A>0$ such that for every $x\in \kg$, 
$$u(1,x)\geq \alpha \exp\left(-A\normp{x}^2\right) . $$
\end{corollary}

\begin{proof}
It follows from \Cref{horizontal-path-triang} that for some constant $C>0$, for every $x\in \kg$, we have  
$$E_{hor}(1/2, 0, 1, x) \leq C \normp{x}^2+C .$$
Now the result follows from  \Cref{full-lower-bound} and \Cref{positive-interior}.  
\end{proof}

\begin{proof}[Proof of \Cref{full-lower-bound-thm}] This is a combination of  \Cref{DC-full-support}, \Cref{lower-cor} and our central limit theorem  established in \Cref{Sec-global-thm}. 
\end{proof}

To conclude, we characterize the case when the limit distribution $\sigma_{1}$ is an ordinary Euclidean Gaussian distribution. It is pleasant that the criterion below is expressed only in terms of the weight filtration of $\kg$, and ultimately only in terms of the adjoint action of the mean $\Xab$ on successive quotients of the descending central series. The forward direction ($4. \implies 6.$ below) is due to Cr\'epel and Raugi \cite[\S 4]{crepel-raugi78}. We are able to show that the converse holds thanks to our lower bound, \Cref{full-lower-bound}. The following statement encompasses \Cref{charact-gaussian} from the introduction. 

\begin{theorem}[Characterization of the Gaussian case] \label{gaussian-criterion}
Let $\sigma$ be the diffusion law associated to $\LL$ as in \cref{limit-diff-gen}. The following are equivalent.

\begin{enumerate}
\item The Lie algebra $(\kg, [.,.]')$ is abelian.
\item For all $a\in \{1, \dots, s\}$ 
$$\kg^{[a]}=\kg^{(2a-1)}.$$
\item For all  $i\in \{1, \dots, 2s-1\}$, 
$$\km^{(i)}\neq \{0\} \iff \text{$i$ is odd}.$$
\item For all  $a\in \{1, \dots, s\}$, we have
$$\kg^{[a]}=L^{[a]}(\Xm^{\otimes (a-1)} \otimes \kg).$$ 
\item For all  $a\in \{1, \dots, s\}$, we have $$[\Xm,\kg^{[a]}/\kg^{[a+1]}]=\kg^{[a+1]}/\kg^{[a+2]}.$$
\item The distribution $\sigma_{1}$ is Gaussian (in the classical Euclidean sense) on $\kg$. 
\end{enumerate}
\end{theorem}

\begin{proof}
$1. \implies 2.$  By definition of $[.,.]'$, the assumption $[.,.]'=0$ means that for all $1\leq i,j\leq 2s-1$, 
$$[\kg^{(i)},\kg^{(j)}]\subseteq \kg^{(i+j+1)} .$$
The result follows by induction on $a$. Indeed, the case $a=1$ is tautological. For the induction step, we observe that
$$\kg^{[a+1]}= [\kg^{[a]}, \kg] = [\kg^{(2a-1)}, \kg^{(1)}]\subseteq \kg^{(2a+1)}$$
and the reverse inclusion is clear from the definition of $\kg^{(2a+1)}$.

$2. \implies 3.$ An element in $\kg^{(2a)}$ must be in $\kg^{[a+1]}$, i.e. in $\kg^{(2a+1)}$ by assumption. Hence $\kg^{(2a)}=\kg^{(2a+1)}$. The claim on the  $\km^{(i)}$'s now follows from their definition.


$3. \implies 1.$ Since $[\km^{(i)},\km^{(j)}]'\subseteq \km^{(i+j)}$, by assumption the left hand side is zero unless both $i$ and $j$ are odd, in which case the right hand side is zero as $i+j$ is even.


 $2. \iff 4.$ Only $2. \implies 4.$ needs a proof. We argue by downwards induction on $a$. The result is clear for a $a=s$. For the induction step, it is enough to notice that
 $$\kg^{[a]}= \kg^{(2a-1)}\subseteq  L^{[a]}(\Xm^{\otimes (a-1)} \otimes \kg) + \kg^{[a+1]}$$
 with $\kg^{[a+1]}=L^{[a+1]}(\Xm^{\otimes a} \otimes \kg) \subseteq L^{[a]}(\Xm^{\otimes (a-1)} \otimes \kg) $ by the induction hypothesis. 
 
   $4. \iff 5.$ Similar as $2. \iff 4.$
 
  $1. \implies 6.$ One may replace $\YY$ by $\chi$ in the definition of the generator $\LL$ which because if $(\kg, [.,.]')$ is abelian, then $\Ad(t \YY)_{|\kg}=\Ad(t\chi)_{|\kg}$. One may also assume that $q=q_{1}$, i.e. the vectors $E_{i}$ in the second order term for $\LL$ define a basis of $\km^{(1)}$.  Now, $\nu=\sigma_{1}$ is the limit measure in the central limit theorem for $\mu$. 
  In this specific context where $4$. holds, the result is due to Crepel-Raugi \cite[Section 4]{crepel-raugi78}, and their theorem also states that $\nu=\sigma_{1}$ is Gaussian with expectation $B$.

 $6. \implies 1.$  If $(\kg, [.,.]')$ is not abelian, then there exists a horizontal path $\gamma : [0, 1]\rightarrow \kg$ such that the ratio $\|\gamma(1)\|/l(\gamma)$ is arbitrary large. It follows that the ratio $d_{hor}(1/2, 0, 1,x)/\|x\| $  converges to $0$ along a sequence of $x$ going to infinity. The lower bound from \Cref{full-lower-bound}  then contradicts the assumption that $\sigma_{1}$ is Gaussian (it does not decrease fast enough). 

 \end{proof}

\begin{proof}[Proof of \Cref{charact-gaussian}]  Consequence of \Cref{gaussian-criterion} and our central limit theorem  established in \Cref{Sec-global-thm}. 
 \end{proof}

\begin{example}[Full upper-triangular group] If $\kg =\oplus_{1\leq i<j\leq q_{1}+1}  \R E_{i,j}$ is the Lie algebra of the group of upper triangular matrices (see \Cref{Sec-upper-triang}),  then the distribution $\sigma_{1}$ is Gaussian if and only if $\Xab = \sum_{1\leq i\leq q_{1}} c_{i}E_{i,i+1} \mod \kg^{[2]}$ with at most one the $c_{i}$'s equal to zero. Indeed, if $c_{i}$ and $c_{i+k-1}$ both vanish, then the application  $[\Xab,.]$ from $\kg^{[k-1]}/\kg^{[k]}$ to $\kg^{[k]}/\kg^{[k+1]}$ is not  surjective: one checks easily that $E_{i,i+k}$ is not in the image. This prevents $\sigma_{1}$ to be Gaussian by \Cref{gaussian-criterion}. The converse implication is straightforward. 
\end{example}

\begin{example}[3-step filiform case] If $\kg = \R A \oplus \R T \oplus \R [A,T]\oplus \R [A,[A,T]]$ is the step-$3$ filiform Lie algebra  (see \Cref{Sec-free}),  then the distribution $\sigma_{1}$ is Gaussian if and only if $\Xab \in \R^*A \oplus \R T \mod \kg^{[2]}$. 
\end{example}


\section{The three reductions} \label{3reductions} \label{Sec-3reductions}

To prove the limit theorems announced earlier, we need to perform a series of reductions. This section is devoted to them.
We shall show that the random product can be altered slightly (allowing a small error) in three different ways. First, we shall truncate the random  variables in order to deal with the moment assumptions. Then (graded replacement) we shall replace the Lie product by the graded Lie product (for which dilations are automorphisms) and third, we shall show that we can replace the driving measure by any measure with the same moments of order at most 2 (Lindeberg replacement). While performing those reductions, we will need to allow  moderate deviations of the random product. This flexibility will be used (in a weak form) to take into account the recentering in the proof of the central limit theorem in \Cref{Sec-global-thm}. The full extent of moderate deviations is to be used in our follow-up paper  \cite{benard-breuillardLLT} where we establish the local central limit theorem.

The notations $(\kg, *, \tkg, \Xab, (\km^{(b)})_{b\leq 2s-1}, \Xm, \chi, *')$ as well as $\DN(\delta_0)$ and $\Pi^{[a,b)}$ are those of  \Cref{Sec-cadre}. We assume $\mu$ is a probability measure on $\kg$ with finite second moment for the weight filtration induced by $\Xab$, and such that $\mu_{ab}$ has expectation $\Xab$.  In particular $\Xm=\E(\mu^{(1)})$. We will mostly work with $\tmu=\mu-\Xm+\chi$ the associated probability measure on $\tkg$.

\subsection{Truncation and moment estimates} \label{Sec-3reductions-tronc}


We first establish moment estimates. To fix ideas, we start with the case where $\mu$ has a finite moment of every order.
Consider a monomial function $M:\tkg^{t}\rightarrow \R$ of degree $b\ge 1$, see \Cref{poldef}. We define its statistics $M^{\infty}$  on $\tkg^{\N^*}$  as the formal series
$$M^{\infty}(\xx)=\sum_{1\leq n_{1}<\dots <n_{t}<\infty}M(x_{n_{1}}, \dots, x_{n_{t}}).$$
For every $N\geq1$, we identify $\tkg^N$ as the subset of $\tkg^{\N^*}$ made of sequences that vanish after the $N$ first terms.  In particular, $M^{\infty}$ is a well-defined function on $\tkg^N$. 

\begin{lemme} \label{momentPi0}  Let $m\in \N$ and assume $\mu$ has a  finite moment of order $2mb$. Then for  $N\geq 1$,

$$\E_{\tmu^{\otimes N}}\left[M^{\infty}(\xx)^{2m} \right]\ll_{m} N^{bm},$$ and in particular for $\delta_{0}>0$, $g,h\in \DN(\delta_{0})$, $\yy=(g,\xx,h)$, $a, b\geq 1$, 

$$\E_{\xx\sim \tmu^{\otimes N}}\left[\|\Pi^{[a,b)}(\yy)\|^{2m} \right]\ll_{m} N^{(b+2\delta_{0})m }.$$
\end{lemme}

\begin{proof} Given $\alpha\in \N^{\dim\tkg}$,  denote by $x\mapsto x^\alpha$  the corresponding monomial on $\tkg$ (say for some ordering of the basis $e^{(i)}_{j}$ fixed in \Cref{Sec-cadre}) and write $d(\alpha)$ its degree for the weight filtration. Expanding the power $2m$, we can write  for $N\geq 1$, $\xx\in \tkg^N$, 
  \begin{align*}
M^{\infty}(\xx)^{2m} &= \sum_{l=1}^{2mb} \sum_{\aalpha \in \mathcal{A}^1_{l}}\sum_{\nn\in \llbracket 1, N\rrbracket^l}C_{\aalpha} x_{n_{1}}^{\alpha_{1}}\dots  x_{n_{l}}^{\alpha_{l}}
  \end{align*}
  where   $\sup_{\aalpha}|C_{\aalpha}| =O_{m}(1)$ and 
  $$\mathcal{A}^{1}_{l} =\left\{\aalpha \in (\N^{\dim\tkg})^l \,:\, \sum_{i=1}^l d(\alpha_{i})=2mb \,\text{ and }\,\inf_{i}d(\alpha_{i})\geq 1 \right\} . $$
    Then by linearity and independence,
    \begin{align*}
\E_{\tmu^{\otimes N}}\left[M^{\infty}(\xx)^{2m}\right] &= \sum_{l=1}^{2mb} \sum_{\aalpha \in \mathcal{A}^1_{l}}\sum_{\nn\in \llbracket 1, N\rrbracket^l}C_{\aalpha}\E_{\tmu}( x_{n_{1}}^{\alpha_{1}})\dots  \E_{\tmu}(x_{n_{l}}^{\alpha_{l}}).
  \end{align*}

The terms for which some $x_{i}^{\alpha_{i}}$ has degree $1$ satisfy $\E_{\tmu}(x_{i}^{\alpha_{i}})=0$, hence the $\aalpha$ that contribute belong to 
$$\mathcal{A}^{2}_{l}=\left\{\aalpha \in (\N^{\dim\tkg})^l \,:\, \sum_{i=1}^l d(\alpha_{i})=2mb \,\text{ and }\,\inf_{i}d(\alpha_{i})\geq 2 \right\}$$
thus forcing $l$ to belong to $\{1, \dots, mb\}$.  

It follows that 
\begin{align*}
\E_{\tmu^{\otimes N}}\left[M^{\infty}(\xx)^{2m}\right] &= \sum_{l=1}^{mb} \sum_{\aalpha \in \mathcal{A}^2_{l}}\sum_{\nn\in \llbracket 1, N\rrbracket^l}C_{\aalpha}\E_{\tmu}( x_{n_{1}}^{\alpha_{1}})\dots  \E_{\tmu}(x_{n_{l}}^{\alpha_{l}})
  \end{align*}
  and the bound for $\E_{\tmu^{\otimes N}}\left[M^{\infty}(\xx)^{2m}\right] $ follows because the parameters $l, \aalpha$ vary in a bounded set independent of $N$,  the parameter $\nn$ varies in a set of cardinality  $N^l\leq N^{mb}$, and the product of expectations is $O_{m}(1)$ by the moment assumption on $\mu$.

To prove the second assertion, we note that $\Pi^{[a,b)}(g,\xx,h)$ is a finite formal linear combination (depending on $a,b$ but not $N$) of statistics in $\xx$ associated to  monomials of degree at most $b$, and such that each statistics of degree $b' \leq b$ has its coefficients bounded above by    $$O(1)\sum_{i_{1}+\dots+j_{l}=b-b' } \|g^{(i_{1})}\|\dots \|g^{(i_{k})}\| \,\|h^{(j_{1})}\|\dots \|h^{(j_{l})}\|$$
if $b'<b$, and $O(1)$ otherwise.
 Given the condition on $g,h\in \DN(\delta_{0})$, this bound is in particular $O(N^{\frac{b-b}{2}'+ \delta_{0}})$.  The result follows by H\"older Inequality and the previous bound established for $\E_{\tmu^{\otimes N}}\left[M^{\infty}(\xx)^{2m}\right]$. 
\end{proof}

In order to obtain similar moment estimates but with weaker moment assumptions on $\mu$, we introduce  the \emph{truncation} $T_{N}\tmu$  as the image of $\tmu$ by 
$T_{N}: \tkg\rightarrow \tkg, x=\oplus x^{(b)}\mapsto \oplus T_{N}x^{(b)}$ where 
 $$
T_{N}x^{(b)} = \left\{
    \begin{array}{ll}
       x^{(b)}  & \mbox{if } \|x^{(b)} \|\leq N^{b/2} \\
        0 & \mbox{if } b\geq 2 \text{ and } \|x^{(b)} \|> N^{b/2}\\ 
        c_{N} & \mbox{if }  b= 1 \text{ and } \|x^{(1)} \|> N^{1/2}
    \end{array}
\right.
$$
and $c_{N}= - \tmu({\|x^{(1)} \|> N^{1/2}})^{-1}\E_{\tmu}(x^{(1)} \1_{{ \|x^{(1)} \|\leq N^{1/2}}})$ is chosen so that $T_{N}\tmu$ is centered in the $\km^{(1)}$-coordinate. 

It is good to notice first that one may replace $\tmu$ by $T_{N}\tmu$ at a very low cost. 

\begin{lemme}[Cost of truncation] \label{truncation-cost}
Let $q \in [2, +\infty)$. If $\mu$ has finite $q$-th moment then for all $N\geq 1$,   
$$\| \tmu^{\otimes N}- (T_{N}\tmu)^{\otimes N}\|   \leq \epsilon_{N}N^{-\frac{q}{2} +1} $$
where $\|.\|$ is the norm in total variation, and the sequence $\epsilon_{N} \in \R^{\N^*}_{>0}$ goes to zero while being bounded by $\epsilon_{N}\leq C (1+m_{q}(\mu))$ where  $C>0$ is a constant that may only depend on $q$. 
\end{lemme}

\begin{proof} Note that $$\tmu \{T_{N}x\neq x\} \leq \tmu \{\max_{b} \|x^{(b)}\|^{1/b}>N^{1/2}\} \leq \epsilon_{N}N^{-\frac{q}{2}}$$
with $\epsilon_{N}$ as in the statement. 
Hence $T_{N}^{\otimes N}$ is the identity on a subset of $\tkg^N$ of $\tmu^{\otimes N}$-measure at least $(1-\epsilon_{N} N^{-\frac{q}{2}+1})$, up to multiplying $\epsilon_{N}$ by a bounded constant.
\end{proof}

We now show that  $(T_{N}\tmu)^{\otimes N}$ satisfies the moment estimates from \Cref{momentPi0}. 

\begin{proposition}[Moment estimate] \label{momentPi} Assume that $\mu$ has a finite second moment. For  $N\geq 1$, $(N_{i})_{i\leq N}\in  [0, N]^N$,  $m\geq 1$ an integer,

$$\E_{\bigotimes_{i=1}^N T_{N_{i}}\tmu}\left[{M^{\infty}}(\xx)^{2m} \right]\ll_{m} N^{bm},$$ and in particular for $\delta_{0}\geq 0$, $g,h\in \DN(\delta_{0})$, $\yy=(g,\xx,h)$, $a,b\geq 1$, 
$$\E_{\xx\sim \bigotimes_{i=1}^N T_{N_{i}}\tmu }\left[\|\Pi^{[a,b)}(\yy)\|^{2m} \right]\ll_{m} N^{(b+2\delta_{0})m }.$$
Moreover, the implicit constant in $\ll_{m}$ can be taken of the form $C (1+m_{2}(\mu)^{mb})$ where $C=C(\kg,[.,.], \R \Xab, (\km^{(b)})_{b}, (e^{(i)}_{j})_{i,j}, m)>0$  is a constant and $m_{2}(\mu)$ is the second moment of $\mu$. 
\end{proposition}

It is quite remarkable that $\mu$ is only assumed to have finite second moment.

\begin{proof}
 Arguing as for \Cref{momentPi0} (with $\otimes_{i} T_{N_{i}}\tmu$ in place of $\tmu^{\otimes N}$) gives 
 \begin{align*}
\E_{\bigotimes_{i=1}^N T_{N_{i}}\tmu}\left[M^{\infty}(\xx)^{2m}\right] &= \sum_{l=1}^{mb} \sum_{\aalpha \in \mathcal{A}^2_{l}}\sum_{\nn\in \llbracket 1, N\rrbracket^l}C_{\aalpha}\E_{T_{N_{n_{1}}}\tmu}( x_{n_{1}}^{\alpha_{1}})\dots  \E_{T_{N_{n_{l}}}\tmu}(x_{n_{l}}^{\alpha_{l}})
  \end{align*}
 but it is no longer true that the product of expectation is $O_{m}(1)$. Nonetheless, note that for  $d(\alpha)\geq 2$, $N'\geq 1$,
$$\E_{T_{N'}\tmu}(|x^\alpha|)= \E_{\tmu}\left(|(T_{N'}x)^\frac{2\alpha}{d(\alpha)}|^{\frac{d(\alpha)}{2}-1} |(T_{N'}x)^\frac{2\alpha}{d(\alpha)}| \right)=O(N'^{\frac{d(\alpha)}{2}-1})$$
because $|T_{N'}x^\frac{2\alpha}{d(\alpha)}|=O(N')$ and  $ \E_{\tmu}(|(T_{N'}x)^\frac{2\alpha}{d(\alpha)}|)$ is bounded independently of $N'$ (by a constant of the form $C(1+ m_{2}(\mu))$ with $C$ as in the statement).   

This bound  implies    that
   \begin{align*}
\E_{\bigotimes_{i=1}^N T_{N_{i}}\tmu}\left[M^{\infty}(\xx)^{2m} \right]
&=   \sum_{l=1}^{mb}  \sum_{\aalpha \in \mathcal{A}^2_{l}}\sum_{\nn\in \llbracket 1, N\rrbracket^l}O_{m}(N^{mb-l})\\
&=O_{m}(N^{mb})
  \end{align*}
 as the parameters $l, \aalpha$ vary in a bounded set independent of $N$, while the parameter $\nn$ varies in a set of cardinality  $N^l$.
 
 The estimate for $\E_{\xx\sim \bigotimes_{i=1}^N T_{N_{i}}\tmu}\left[\|\Pi^{[a,b)}(\yy)\|^{2m} \right]\ll_{m} N^{(b+2\delta_{0})m }$ follows by the same argument as for  \Cref{momentPi0}.
 
 The comment on the constant in $\ll_{m}$ is just an inspection of the proof.
 \end{proof}

\subsection{Graded replacement}

In the remainder of \Cref{Sec-3reductions} we keep the assumption that $\mu$ has a finite second moment for the weight filtration induced by $\Xab$. We now show how to replace the product $*$ by the its graded counterpart $*'$, for which the dilations $(D_{r})_{r>0}$ act as automorphisms. 

\begin{proposition}[Graded replacement]  \label{graded-replacement}
For  $N\geq 1$, $(N_{i})_{i\leq N}\in [ 0, N]^N$, $g,h\in \DN(\delta_{0})$, $\yy=(g,\xx,h)$, $m\geq1$, 
 $$ \E_{\xx \sim \bigotimes_{i=1}^N T_{N_{i}}\tmu } \left( \|\DilsN \Pi(\yy) -\DilsN \Pi'(\yy)\|^{2m} \right) \ll_{m}N^{(-1+2\delta_{0})m} .$$
Moreover, the implicit constant in $\ll_{m}$ can be taken of the form $C (1+m_{2}(\mu)^{(2s-1)m})$ where $C=C(\kg,[.,.], \R \Xab, (\km^{(b)})_{b}, (e^{(i)}_{j})_{i,j}, m)>0$ and $m_{2}(\mu)$ is the second moment of $\mu$. 
\end{proposition}
\begin{proof}
Using the triangle inequality, it is sufficient to show that the bound for
$$\E_{\xx \sim \bigotimes_{i=1}^N T_{N_{i}}\tmu } \left( \|\DilsN \Pi^{(b)}(\yy) -\DilsN \Pi'^{(b)}(\yy)\|^{2m} \right) .$$
But $\Pi^{(b)}(\yy) - \Pi'^{(b)}(\yy)$ is the projection of $\Pi^{(b)}(\yy) $ to $\tkg^{(b+1)}$ parallel to $\tkm^{(b)}$. Hence, 
$$\|\DilsN(\Pi^{(b)}(\yy) - \Pi'^{(b)}(\yy)) \|\leq N^{-(b+1)/2} \,\|\Pi^{(b)}(\yy)\| $$
and the result follows because we proved in \Cref{momentPi} that
$$ \E_{\xx \sim \bigotimes_{i=1}^N T_{N_{i}}\tmu }\left[\|\Pi^{(b)}(\yy) \|^{2m}\right]\ll_{m} N^{(b+2\delta_{0})m}$$
where the implicit constant in $\ll_{m}$ has the desired form. 
\end{proof}

\noindent\emph{Remark}. The same holds for $\tmu$ in place of $T_{N}\tmu$ under the assumption of a finite moment of order $2mb_{\max}$, where $b_{\max}\leq 2s-1$ is the length of the weight filtration. Just use  \Cref{momentPi0} in place of  \Cref{momentPi}. 

\bigskip

We infer a graded approximation at the level of Fourier transforms. It will be used to prove the central limit theorem with Berry-Esseen estimate (and the local limit theorem in \cite{benard-breuillardLLT}).

\begin{corollary} \label{grading-Fourier}
 For $\xi$ a linear form on $\tkg$, $N\geq 1$, $(N_{i})_{i\leq N}\in [ 0, N]^N$, $g,h\in \DN(\delta_{0})$, $\yy=(g,\xx,h)$, 
$$\big|\E_{\xx \sim \bigotimes_{i=1}^N T_{N_{i}}\tmu } [e_{\xi} (\Pi(\yy) )-e_{\xi} (\Pi'(\yy) )] \big| \leq C(1+m_{2}(\mu)^s) \,\| \DilN \xi \| \,N^{-1/2+\delta_{0}}$$
where $C=C(\kg,[.,.], \R \Xab, (\km^{(b)})_{b}, (e^{(i)}_{j})_{i,j})>0$.
\end{corollary}

\begin{proof}
\begin{align*}
| \E_{\xx \sim \bigotimes_{i=1}^N T_{N_{i}}\tmu } [e_{\xi} (\Pi(\yy) )-e_{\xi} (\Pi'(\yy) )]|
&\ll  \E_{\xx \sim \bigotimes_{i=1}^N T_{N_{i}}\tmu } (|\xi \Pi(\yy) - \xi \Pi'(\yy) |)\\
&\leq \| \DilN \xi\| \, \E_{\xx \sim \bigotimes_{i=1}^N T_{N_{i}}\tmu } (\|\DilsN(\Pi(\yy) - \Pi'(\yy)) \|)\\
&\leq C(1+m_{2}(\mu)^s) \| \DilN \xi\| \,N^{-1/2+\delta_{0}}
\end{align*}
where the last line is given by the inequality $\|.\|_{L^1}\leq \|.\|_{L^2}$ followed by 
 \Cref{graded-replacement} (case $m=1$).
\end{proof}

We also deduce a result of graded approximation that is uniform on the time interval $\{0, \dots, N\}$. It will be useful to prove our theorem \`a la Donsker in \Cref{donsker-sec}. Here we do not use  truncation. 

\begin{corollary} \label{grading-uniform}
Let $\alpha \in [0, \frac{1}{2}) $. For $N\geq1$, set
$$F_{N}= \left\{ \xx \in \tkg^N :\sup_{k\leq N} \|\DilsN  \Pi\left((x_{i})_{i=1}^k, -k\chi \right) -\DilsN  \Pi'((x_{i})_{i=1}^k, -k\chi)\| >N^{-\alpha} \right\}. $$
If $\mu$ has finite moment of order $2mb_{\max}$ where $b_{\max}$ is the length of the weight filtration, then 
$$\tmu^{\otimes N} (F_{N})=O_{\alpha,m}(N^{(-1+2\alpha+\frac{1}{m})m}). $$
In particular $\tmu^{\otimes N} (F_{N}) $ goes to $0$ provided $m>1/(1-2\alpha)$.
\end{corollary}

\begin{proof} The result follows by replacing the supremum by a sum, then using the Markov inequality and (the remark after the proof of) \Cref{graded-replacement} with $\delta_{0}=0$.  
\end{proof}


\subsection{Lindeberg replacement}

We show how to replace the driving measure by any other measure which coincides up to order $2$ for the weight filtration. The substitution takes place at the level of Fourier transforms.

\begin{proposition}[Lindeberg replacement] \label{gaussian}
Let $\eta$ be a probability measure on $\kg$ with finite second moment  for the weight filtration  induced by $\Xab$ and which coincides with $\mu$ up to order $2$. For $\xi$ a linear form on $\tkg$, $N\geq 1$, $(N_{i})_{i\leq N}\in [ 0, N]^N$, $g,h\in \DN(\delta_{0})$, $\yy=(g,\xx,h)$, 
$$\big|\E_{\xx \sim \bigotimes_{i=1}^N T_{N_{i}}\tmu } [e_{\xi} (\Pi'(\yy) )]-\E_{\xx \sim \bigotimes_{i=1}^N T_{N_{i}}\teta } [e_{\xi} (\Pi'(\yy) )]\big| \leq \,(\| \DilN \xi\|+\| \DilN \xi\|^3)N^{3s \delta_{0}}\sum_{i=1}^N N^{-1}_{i} \,\eps_{N}$$
where $\eps_{N}=o_{\eta}(1)$. More generally, assuming $\mu$, $\eta$ both have finite $m$-th moment for some $m\in [2,3]$, we may take $\eps_{N}=o_{\eta}(N^{-\frac{m}{2}+1})$ if $m<3$ and $\eps_{N}=O_{\eta}(N^{-1/2})$ if $m=3$. 

Moreover, in the $m=3$ case, the implicit constant in $O_{\eta}$ can be taken of the form $C (1+m_{2}(\mu)+m_{2}(\eta))^{6s^2}(1+m_{3}(\mu)+m_{3}(\eta))$ where $C=C(\kg,[.,.], \R \Xab, (\km^{(b)})_{b}, (e^{(i)}_{j})_{i,j})>0$ and $m_{q}(\mu), m_{q}(\eta)$ denote the $q$-th moments of $\mu$ and $\eta$.
\end{proposition}

In this paper, we will apply this proposition with $(N_{i})$ constant equal to $N$. In \cite{benard-breuillardLLT}, we will need to consider gradually increasing $(N_{i})$. In both cases, $\sum_{i=1}^N N^{-1}_{i}=O(1)$, so the right-hand side is $O_{\eta}(1)(\| \DilN \xi\|+\| \DilN \xi\|^3)N^{3s \delta_{0}}\,\eps_{N}$.

\begin{proof} 
We replace $\tmu$-variables by $\teta$-variables successively. At the step $i$, we set 
$$\kappa_{i} = g \otimes T_{N_{1}}\tmu \otimes \dots \otimes T_{N_{i-1}}\tmu \,\,\,\, \,\,\,\,  \,\,\,\,  \,\,\,\,  \theta_{i}=T_{N_{i+1}}\teta \otimes \dots\otimes T_{N_{N}}\teta \otimes h$$
and we need show that 
$$\Delta_{i}:=|\E_{\kappa_{i}\otimes (T_{N_{i}}\tmu-T_{N_{i}}\teta) \otimes \theta_{i}}(e_{\xi}( \Pi'(\uu, v, \ww)))| \leq  (\| {\DilN }\xi\|+\| {\DilN }\xi\|^3)  \,N_{i}^{-1+3s\delta_{0}}\eps_{N} .$$
Setting $\xi_{N}=\DilN \xi$, $\uu_{N}=\DilsN \uu$, $v_{N}=\DilsN v$, $\ww_{N}=\DilsN \ww$, we can use that dilations commute with  $*'$ to rewrite  
\begin{align*}
\Delta_{i}=|\E_{\kappa_{i}\otimes  (T_{N_{i}}\tmu-T_{N_{i}}\teta) \otimes \theta_{i}}(e_{\xi_{N}}( \Pi'(\uu_{N}, v_{N}, \ww_{N})))| .
\end{align*}
We now isolate the part of $\Pi'(\uu_{N}, v_{N}, \ww_{N})$ depending on $v$. To do so, we  observe that for $g_{1}, g_{2}, g_{3}\in \tkg$,
$$g_{1}*'g_{2} *'g_{3}=g_{1}*'g_{3}+P_{g_{1},g_{3}}(g_{2})$$
where $P_{g_{1},g_{3}}$ is a polynomial of bounded degree without constant term and whose coefficients depend  polynomially on $g_{1},g_{3}$. Setting $x_{N}=\Pi' \uu_{N}$, $z_{N}=\Pi' \ww_{N}$, we then rewrite
\begin{align*}
\Delta_{i}&=|\E_{\kappa_{i}\otimes  (T_{N_{i}}\tmu-T_{N_{i}}\teta) \otimes \theta_{i}}\big(e_{\xi_{N}}( x_{N}*'z_{N}+P_{x_{N}, z_{N}}( v_{N}))  \big)  |\\
&=|\E_{\kappa_{i} \otimes \theta_{i}}\big[e_{\xi_{N}}( x_{N}*'z_{N}) \E_{ (T_{N_{i}}\tmu-T_{N_{i}}\teta)}(e_{\xi_{N}}(P_{x_{N}, z_{N}}( v_{N})))\big]  |\\
 &\leq \E_{\kappa_{i} \otimes \theta_{i}} |\E_{ (T_{N_{i}}\tmu-T_{N_{i}}\teta)}(e_{\xi_{N}}(P_{x_{N}, z_{N}}( v_{N})))  |.
\end{align*}
Using Taylor expansion up to order $3$, we have 
\begin{align*}
\E_{(T_{N_{i}}\tmu)}(e_{\xi_{N}}(P_{x_{N}, z_{N}}( v_{N}))) = 1+ \E_{(T_{N_{i}}\tmu)}(-2i\pi \xi_{N }P_{x_{N}, z_{N}}( v_{N})) &+\E_{(T_{N_{i}}\tmu)}(-4\pi^2 (\xi_{N }P_{x_{N}, z_{N}}( v_{N}))^2)\\
&+ \E_{(T_{N_{i}}\tmu)}(O(\xi_{N }P_{x_{N}, z_{N}}( v_{N}))^3)
\end{align*}
and the analogous formula for $T_{N_{i}}\teta$. 
Our purpose is to compare those two polynomial expansions, and we do this one monomial at a time. Recall that for $\alpha\in \N^{\dim \tkg}$, we write $x\mapsto x^\alpha$ the induced monomial on $\tkg$ (for some ordering of the basis $e^{(i)}_{j}$ fixed in \Cref{Sec-cadre}) and $d(\alpha)$ its degree.  
We compare the expectations of  $v^\alpha_{N}:=(v_{N})^\alpha$ when $v$ varies with respect to the measures $T_{N_{i}}\tmu$ and $T_{N_{i}}\teta$. 
\begin{itemize}
\item If $d(\alpha)=1$, then 
$$\E_{T_{N_{i}}\tmu}(v^{\alpha}_{N}) = \E_{T_{N_{i}}\teta}(v^{\alpha}_{N}) = 0$$
because our truncations are centered by definition.

\item If $d(\alpha)=2$, then we use that $\mu$ and $\eta$ coincide up to order 2 to write 
 \begin{align*}
|\E_{T_{N_{i}}\tmu-T_{N_{i}}\teta}(v^{\alpha}_{N})|&= N^{-1} |\E_{T_{N_{i}}\tmu-\tmu}(v^{\alpha})+ \E_{\teta-T_{N_{i}}\teta}(v^{\alpha})|\\
&\leq N^{-1} \E_{ \tmu+\teta}[(|v^{\alpha}| + |(T_{N_{i}}v)^{\alpha}|)1_{E_{N_{i}}}(v)]
\end{align*}
where $E_{N_{i}}=\{v: T_{N_{i}}v\neq v\}$ has measure $\tmu(E_{N_{i}}), \teta(E_{N_{i}})\leq o_{\eta}(N^{-m/2}_{i})$ if $\mu, \eta$ have finite $m$-th moment. Using the  H\"older inequality $\| f_{1}f_{2}\|_{L^1}\leq \|f_{1}\|_{L^{m/2}}\| f_{2}\|_{L^{\frac{m}{m-2}}}$,  we obtain
$$\E_{T_{N_{i}}\tmu-T_{N_{i}}\teta}(v^{\alpha}_{N})=o_{\eta}(N_{i}^{-m/2}).$$

\item If $d(\alpha)\geq 3$ and $\mu$, $\eta$ both have finite $m$-th moment, then we may we use \Cref{croissance-moment} below to get
 $$
\E_{T_{N_{i}}\tmu+T_{N_{i}}\teta}(|v^{\alpha}_{N}|) = \left\{
    \begin{array}{ll}
       o_{\eta}(N^{-m/2}_{i}) &  \mbox{if  $m\in [2,3)$}\\
        O_{\eta}(N^{-3/2}_{i}) & \mbox{if  $m\in [3,+\infty)$}
    \end{array}
\right.
$$

\end{itemize}

Combining the above Taylor expansions and the monomial estimates, we deduce that
\begin{align*} 
|\E_{(T_{N_{i}}\tmu-T_{N_{i}}\teta)}(P_{x_{N}, z_{N}}( v_{N})))  | \leq (\|\xi_{N}\|+ \|\xi_{N}\|^3) (1+\|x_{N}\|+\|z_{N}\|)^{3s} N^{-1}_{i}\eps_{N}
\end{align*}
and the domination of $\Delta_{i}$ follows by integrating in $\uu_{N}$, $\ww_{N}$ and using the moment estimate from \Cref{momentPi}. 

The final claim on  $\eps_{n}$ in the $m=3$ case follows from inspecting the proof. More precisely, the dependence in $(1+ m_{3}(\mu)+m_{3}(\eta))$ originates from the estimates of $\E_{(T_{N_{i}}\tmu-T_{N_{i}}\teta)}(v_{N}^\alpha)$ performed above, and the dependence in $(1+m_{2}(\mu)+m_{2}(\eta))^{6s^2}$ appears in the very last step, when we integrate with respect to $\uu_{N}$, $\ww_{N}$ and   use \Cref{momentPi}.
\end{proof}

To complete the proof of \Cref{gaussian}, we record the following.  
\begin{lemme} \label{croissance-moment}
Let $\varphi: (\Omega, \mathbb{P})\rightarrow \R^+$ be an integrable non-negative random variable. For $\beta>1$, 
$$\int_{\Omega} \varphi^\beta \1_{{\varphi\leq N}} \,d\mathbb{P}=o_{\varphi}(N^{\beta-1}).$$
\end{lemme}

\begin{proof}
Just observe that 
$$\int_{\Omega} \left(\frac{\varphi}{N}\right)^{(\beta-1)} \1_{{\varphi\leq N}} \,\varphi d\mathbb{P}\rightarrow0$$
by dominated convergence. 
\end{proof}


\section{Global limit theorems} \label{Sec-global-thm}

This section is dedicated to the proof of the global limit theorems announced in the introduction, namely Theorems \ref{CLT}, \ref{BE}, \ref{Donsker-nilpotent} and \Cref{asymp-close}. We will combine  \Cref{Sec-densite} and \Cref{3reductions}. The notations $(\kg, *, \tkg, \Xab, (\km^{(b)})_{b\leq 2s-1}, X, \chi, *')$ are those of  \Cref{Sec-cadre}. We assume that $\mu$ is a probability measure on $\kg$ with finite second moment for the weight filtration induced by $\Xab$, and such that $\mu_{ab}$ has expectation $\Xab$ and non-degenerate covariance matrix. In particular, we have $X=\E(\mu^{(1)})$.

Let $(E_{i})_{1\leq i\leq q_{1}}$ be a basis of $\km^{(1)}$ in which the covariance matrix of $\mu^{(1)}$ is the identity. Let $B\in \km^{(2)}$ be the expectation of $\mu^{(2)}$. Set $\LL=(\LL_{t})_{t\in \R^+}$ the time-dependent left-invariant differential operator on $(\kg, *')$ given by 
 \begin{align} \label{limit-diff-gen2}
\mathscr{L}_{t}=\frac{1}{2}\sum_{i=1}^{q_{1}}  \left(\Ad(t\chi) E_{i}\right)^2+ \Ad(t\chi)B .
 \end{align}
 We denote by $(W(t))_{t\in \R^+}$ a diffusion process on $(\kg,*)$ of infinitesimal generator $\LL$, and  set $Z(t)=W(t)*'t\chi$ as in  \Cref{homogeneisation}. The notations $\sigma$, $\lambda$ correspond to the respective distributions of $W$ and $Z$ in the space of continuous paths.  Finally we write $\nu=\sigma_{1}$ the time-$1$ distribution of $\sigma$, i.e. the law of $W(1)$.

\subsection{Central limit theorem and convergence rate} \label{Sec-CLT+BE}

We prove the central limit theorem and the Berry-Esseen estimate announced in the introduction. We start with a lemma that further connects $\mu$ to the diffusion $\sigma$, via their respective bias extensions.

\begin{lemme} \label{densite}
The time-$1$ distribution $\lambda_{1}$ is the bias extension $\teta$ to $\tkg$ of a measure $\eta$ on $\kg$ with finite $p$-th moment for all $p$ and which coincides with $\mu$ up to order $2$. In particular, one has for all $N\geq 1$,
$$\teta^{*'N}=\DilN\nu*'N\chi .$$
\end{lemme}

Later, we will bound all the moments of $\eta$  using the  second moment of $\mu$, see  \Cref{bound-eta-moments}.

\begin{proof}
The only possible candidate is $\eta=\lambda_{1}-\chi+\Xm$. It does have finite polynomial moments by \Cref{density-reg} and satisfies $\teta=\lambda_1$. By  \Cref{homogeneisation}, we have
$$\teta^{*'N}=\lambda^{*'N}_1= \lambda_N=\DilN\nu*'N\chi .$$ 

It remains to compute the covariance and expectations. First recall from \Cref{homogeneisation} that for any  $f\in C^\infty_{c}(\tkg)$, 
\begin{align} \label{tetat0}
\lambda_{1}(f) -f(0) =  \int_{0}^1\lambda_{s} (\widetilde{\LL}.f) ds .
\end{align}
Rewriting $\lambda_{s}(\widetilde{\LL}.f) = \lambda_{1}[(\widetilde{\LL}.f)\circ D_{\sqrt{s}}] $ thanks to \Cref{density-inv}, and using that $\lambda_{1}$ has finite polynomial moments, we see that  the relation  \eqref{tetat0} holds more generally for all $f\in C^\infty(\tkg)$ whose derivatives have polynomial growth (indeed, just approximate $f(x)$ by $f(x) \rho(\eps x)$ where $\rho$ is a bump function and $\eps\to 0^+$). Now, if $f$ is a linear form on $\km^{(1)}$, then $\widetilde{\LL}.f=0$, so the projection $\lambda^{\,(1)}_{1}$ is centered, i.e. $\E(\eta^{(1)})=\Xm$. If $(E'_{i} : \km^{\,(1)}\rightarrow \R)$ refers to the dual basis of $(E_{i})$, then $\widetilde{\LL}.f(E'_{i}E'_{j})=\delta_{ij}$, whence $\lambda^{(1)}_{1}$ and $\eta$ have covariance matrix the identity in the basis $(E_{i})$. Finally, if $f$ is a linear form on $\km^{(2)}$, then $\widetilde{\LL}.f=f(B)$, so $\eta(f)=\lambda_{1}(f)=f(B)$. This concludes the proof. 
\end{proof}

\bigskip
We now   prove Theorems \ref{CLT}, \ref{BE}.  With our current notations, they together  correspond to  \Cref{CLT+BE} below.  Recall that $X=\E(\mu^{(1)})$.

\begin{theorem}[CLT + Berry-Esseen] \label{CLT+BE} Keep the above notations.
Let $f \in C^{\infty}_{c}(\kg)$ be a smooth function with compact support. Then for $N\geq1$, 
$$|\DilsN(\mu^{*N}*-N\Xm)(f)-\nu(f)| \leq \eps_{N} \max_{|\alpha| \leq \dim \kg +4} \|\partial^\alpha f\|_{L^1}$$
where the notation $\eps_{N}$ stands for $\eps_{N}=o(1)$. 
Moreover,  if $\mu$ has finite third moment for the weight filtration, then we can take $\eps_{N} \leq C(1+m_{2}(\mu))^{6s^2}(1+m_{3}(\mu)) N^{-1/2}$ where $C=C(\kg,[.,.], \R \Xab, (\km^{(b)})_{b}, (e^{(i)}_{j})_{i,j})>0$. 
\end{theorem}

\begin{proof}[Proof of  \Cref{CLT+BE}]

First we prove the statement without additional moment assumption, in particular $\eps_{N}$ stands for $o(1)$. By Fourier Inversion formula, we have
\begin{align*} 
\DilsN(\mu^{*N}*-N\Xm)(f)&= \int_{\dkg} \hf(\xi) \,\htmuNchi({\DilsN}\xi)  \,d\xi 
\end{align*} 
where a linear form $\xi\in \dkg$ is seen as a linear form on $\tkg$ by setting $\xi(\chi)=0$, and where $\tmu^{*N}_{-N\chi}=\tmu^{*N}*-N\chi$.

Using the truncation estimate  from \Cref{truncation-cost}, followed by the graded replacement and the Lindeberg replacement from  \Cref{grading-Fourier}, \Cref{gaussian} applied with $\delta_{0}=0$ and to  the measure $\eta$ from \Cref{densite}, we deduce
\begin{align*} 
\DilsN(\mu^{*N}*-N\Xm)(f)
&= \int_{\dkg} \hf(\xi) \,\widehat{(T_{N}\tmu)^{*'N}_{-N\chi}}({\DilsN}\xi)  \,d\xi   \,+\, \eps_{N} \int_{\kg} |\hf(\xi)|d\xi \\
&= \int_{\dkg} \hf(\xi) \,\widehat{(T_{N}\teta)^{*'N}_{-N\chi}}({\DilsN}\xi)  \,d\xi \,+\, \eps_{N}\int_{\kg} |\hf(\xi)| (1+ \|\xi\| +\|\xi\|^3)   \,d\xi \\
&= \int_{\dkg} \hf(\xi) \,\widehat{\teta^{*'N}_{-N\chi}}({\DilsN}\xi)  \,d\xi \,+\,\eps_{N}\int_{\kg} |\hf(\xi)| (1+\|\xi\| +\|\xi\|^3)   \,d\xi \\
&= \nu(f)\,+\,\eps_{N}\int_{\kg} |\hf(\xi)| (1+\|\xi\| +\|\xi\|^3)   \,d\xi 
\end{align*} 
where the last line uses that $\teta^{*'N}=(\DilN\nu)*' N\chi$.

We can finally bound the dependence in $f$ in terms of the $L^1$-norm of derivatives of $f$. Indeed, using integration by parts, one sees that for all $k\geq 0$, $\xi \in \dkg \smallsetminus \{0\}$,
$$ |\hf(\xi)| \ll_{k} \frac{1}{\|\xi\|^k}  \sup_{|\alpha|=k}  \| \partial^\alpha  f\|_{L^1} $$
where $\partial^\alpha $ varies among the order $k$ additive derivatives of $f$ with respect to a fixed arbitrary basis. It follows that 
$$\int_{\dkg} \hf(\xi) (1+\|\xi\| +\|\xi\|^3)   \,d\xi = O(1)\sup_{|\alpha| \leq \dim \kg +4}  \| \partial^\alpha  f\|_{L^1},$$
which concludes the first part of the proof.

We now turn to the case where $\mu$ has finite third moment. Inspecting the above argument, and using the description of constants involved in \Cref{truncation-cost}, \Cref{grading-Fourier}, \Cref{gaussian}, it appears we may take
\begin{align} \label{eqepsN-}
\eps_{N} \leq C(1+m_{2}(\mu)^{6s^2}+m_{2}(\eta)^{6s^2})(1+m_{3}(\mu)+m_{3}(\eta)).
\end{align}
 It remains to bound the moments of $\eta$ using those of $\mu$. We will do it in \Cref{lem-a} and \Cref{bound-eta-moments} below. Then \Cref{CLT+BE} follows by combining \eqref{eqepsN-} and \Cref{bound-eta-moments}. 

\end{proof}

In order to study the moments of $\eta$, we  introduce $$a(\teta):=\int_{\tkg} \|x^{(1)}\|^2 +\|x^{(2)}\| \,d\teta(x)=m_{2}(\teta^{(1)})+ m_{1}(\teta^{(2)})$$
where we recall that $\teta^{(i)}=\pi^{(i)}(\teta)$ is the projection of $\teta$ to $\tkm^{(i)}$ parallel to $\oplus_{j\neq i}\tkm^{(j)}$. These notations also make sense for $\mu$ in place of $\teta$. We start with a preliminary bound, relating 
$a(\teta)$ to  $a(\mu)$.

\begin{lemme} \label{lem-a} We have
$$\|X\| \leq a(\mu)^{1/2} \,\,\,\,\,\,\,\,\,\text{and}\,\,\,\,\,\,\,\,\, a(\teta)\leq C (1+a(\mu))$$
where  $C=C(\kg,[.,.], \R \Xab, (\km^{(b)})_{b}, (e^{(i)}_{j})_{i,j})>0$.
\end{lemme}

\begin{proof}
The estimate $\|X\|  \leq m_{1}(\mu^{(1)})\leq a(\mu)^{1/2}$ follows from the Cauchy-Schwarz inequality. Let us check the second one. Note that $\teta^{(1)}$ is a centered probability measure on $\km^{(1)}$ with same covariance as $\mu^{(1)}$. Hence $m_{2}(\teta^{(1)})=m_{2}(\mu^{(1)}-X)\leq m_{2}(\mu^{(1)})$.  
To conclude, we check that  $m_{1}(\teta^{(2)})\leq  C(1+a(\mu))$. In the first part of the proof of \Cref{CLT+BE}, we established (without rate) the graded CLT $$ D_{\frac{1}{\sqrt{n}}} (\tmu^{*'n})\rightarrow \teta.$$
It yields
\begin{align}\label{eq1-a(eta)}
\int_{\kg} \|x^{(2)}\| \,d\teta(x) 
&\leq \liminf_{n} \frac{1}{n} \int_{\kg^n} \|(g_{1}*'\dots *'g_{n})^{(2)}\| \,d\tmu^{\otimes n}(\ug) \nonumber\\
& = \liminf_{n} \frac{1}{n} \int_{\kg^n} \|g_{1}^{(2)}+ \dots + g_{n}^{(2)} + \sum_{i<j} [g_{i}^{(1)}, g_{j}^{(1)}]'\| \,d\tmu^{\otimes n}(\ug). 
\end{align}
On the one hand, by the triangle inequality and recalling $\|X\| \leq m_{2}(\mu^{(1)})^{1/2}\leq 1+a(\mu)$, we have
$$ \int_{\kg^n} \|g_{1}^{(2)}+ \dots + g_{n}^{(2)} \| \,d\tmu^{\otimes n}(\ug) \leq  2n (1+a(\mu)).$$
On the other hand, by the Cauchy-Schwarz inequality, 
\begin{align}\label{eq2-a(eta)}
\int_{\kg^n} \| \sum_{i<j} [g_{i}^{(1)}, g_{j}^{(1)}]'\| \,d\tmu^{\otimes n}(\ug) \leq \left(\int_{\kg^n} \| \sum_{i<j} [g_{i}^{(1)}, g_{j}^{(1)}]'\|^2 \,d\tmu^{\otimes n}(\ug)\right)^{1/2}
\end{align}
then recalling that the norm $\|.\|$ is Euclidean, and using that $\tmu^{(1)}$ is centered, we get
\begin{align}\label{eq3-a(eta)}
\int_{\kg^n} \| \sum_{i<j} [g_{i}^{(1)}, g_{j}^{(1)}]'\|^2 \,d\tmu^{\otimes n}(\ug) 
&=\int_{\kg^n}\sum_{i<j}\sum_{k<l} \langle [g_{i}^{(1)}, g_{j}^{(1)}]'  \,,\,  [g_{k}^{(1)}, g_{l}^{(1)}]'\rangle \,d\tmu^{\otimes n}(\ug) \nonumber\\
&=\int_{\kg^n}  \sum_{i<j} \|[g_{i}^{(1)}, g_{j}^{(1)}]' \|^2 \,d\tmu^{\otimes n}(\ug)\nonumber\\
&\leq n^2 C (1+m_{2}(\mu^{(1)}))^2.
\end{align}
Combining \eqref{eq1-a(eta)}, \eqref{eq2-a(eta)}, \eqref{eq3-a(eta)} concludes the proof. 
\end{proof}

Using the previous lemma, we can bound all moments of $\eta$ using only the quantity $a(\mu)$ (and in particular the second moment of $\mu$).

\begin{proposition}[Bound for all moments]  \label{bound-eta-moments}
Given $r>0$, we have
$$m_{r}(\eta)\leq C (1+a(\mu)^{r/2}) \leq C (1+m_{r}(\mu))$$
where $C>0$ only depends on $\kg,[.,.], \R \Xab, (\km^{(b)})_{b}, (e^{(i)}_{j})_{i,j}$ and an upper bound on $r$.
\end{proposition}

The proof is inspired by \cite{bentkus-pap96}.
\begin{proof}  
We first establish that for any $\alpha \in \N^{\dim \kg} $, the $L^1(\teta)$-norm  the monomial $x^\alpha$ is related to that of monomial of lower orders. Assume the degree of $\alpha$ satisfies $d(\alpha)\geq 3$. Using the invariance relation 
$\teta=D_{\frac{1}{\sqrt{2}}}(\teta*' \teta)$, we have
\begin{align*}
\int_{\kg}|x^\alpha| \,d\teta(x) &= \frac{1}{2^{d(\alpha)/2}}\int_{\kg}|(x*' y)^\alpha| \,d\teta(x)d\teta(y)\\
&\leq \frac{1}{2^{d(\alpha)/2}}\left(2\int_{\kg}|x^\alpha|  \,d\teta(x)  + C'\sum_{\substack{d(\beta)+d(\gamma)= d(\alpha)\\ d(\beta), d(\gamma)<d(\alpha) }}  \int_{\kg}|x^\beta| d\teta(x) \int_{\kg} |x^\gamma| d\teta(x)\right).
\end{align*}
for some $C'=C'({\bf D}, \alpha)\geq 0$, where ${\bf D}$ refers to the data ${\bf D}=(\kg,[.,.], \R \Xab, (\km^{(b)})_{b}, (e^{(i)}_{j})_{i,j})$. Using  $d(\alpha)\geq 3$ to guarantee $1-\frac{1}{2^{d(\alpha)/2-1}} \geq 1/10$, we get
\begin{align*}
\int_{\kg}|x^\alpha| \,d\teta(x) 
\leq 10C' \sum_{\substack{d(\beta)+d(\gamma)= d(\alpha)\\ d(\beta), d(\gamma)<d(\alpha) }}  \int_{\kg}|x^\beta| d\teta(x) \int_{\kg} |x^\gamma| d\teta(x) 
\leq C'' a(\teta)^{d(\alpha)/2}
\leq C'' (1+a(\mu))^{d(\alpha)/2}.
\end{align*}
where the second inequality is obtained by iterating the first yielding $C''=C''({\bf D}, \alpha)>0$, while the third inequality follows from Lemma \ref{lem-a}. The same holds for $\eta$ instead of $\teta$ due to the relation $\eta=\teta-\chi+X$ and the estimate $\|\chi\|=\|X\|\leq a(\mu)^{1/2}$. 

This does not justify yet the  formula that was announced in the statement because we have only looked at monomials of integer degree so far. So let $r>0$. Write $2k_{i}$ the smallest even integer greater than $r/i$. Using the H\"older inequality, then the previous paragraph, we get
\begin{align*}
\int_{\kg}\|x^{(i)}\|^{r/i} \,d\eta(x) 
&\leq \left(\int_{\kg}\|x^{(i)}\|^{2k_{i}} \,d\eta(x)\right)^{r/(i2k_{i})} \\
&\leq  C (1+a(\mu)^{ik_{i} .r/(i2k_{i})})\\
&=  C (1+a(\mu)^{r/2})
\end{align*}
where $C=C({\bf D}, r)>0$. The second estimate in the statement follows by H\"older inequality. This concludes the proof of \Cref{bound-eta-moments} (whence that of \Cref{CLT+BE}). 
\end{proof}

We conclude this section with a discussion of several papers that preceded us in attempting to prove the central limit theorem for non-centered random walks on a general nilpotent Lie group. The case of the Heisenberg group was handled successfully by Tutubalin \cite{tutubalin64} but a substantial difficulty remained to formulate the correct renormalization for an arbitrary group. 

In \cite{virtser74}, Virtser proposes a renormalization for  walks on the full upper-triangular unipotent matrix group, i.e. $\kg=\sum_{1\leq i<j\leq n} \R E_{i,j}$ with $[E_{i,j},E_{k,l}]=\1_{j=k}E_{i,l}  - \1_{i=l}E_{k,j}$. While Virtser's normalization is appropriate (and coincides with ours) for certain values of $\Xab$, as in \cite[Theorem 2 \& 3]{virtser74}, for certain other values his renormalization is too strong and may lead to non absolutely continuous limiting measures. This occurs for instance if $n=4$, and $\Xab = s E_{1,2}+t E_{3,4} \mod [\kg, \kg]$ with $s, t\neq 0$.  Indeed, in this case direct computation shows that our subspace of weight  at least $3$ is $\kg^{(3)}=\R E_{1,4}+\R(E_{1,3} - t/s E_{2,4}) \subsetneq \kg^{[2]}$. However, Virtser renormalizes the whole subspace $\kg^{[2]}$ with coefficient $N^{-3/2}$, hence shrinking to a point the coordinate $\kg^{[2]}/\kg^{(3)}$ which lives at scale $N$.

In \cite{raugi78},  Raugi proposes a renormalization for walks on a general nilpotent Lie group. We explain below why his renormalization is too weak to prevent escape of mass in non-centered cases, despite the CLT formulated in his paper.

Raugi's renormalization is defined with respect to the sequence of ideals $\{\ki^{k,l}, 0\leq l<k\leq s\}$, where $\ki^{k,l}$ is the linear span of $\kg^{[k+1]}$ and brackets containing  $k$ elements of $\kg$, at least $l$ of which are in $\R \Xab$. One also authorizes the case $(k,l)=(1,1)$ and set $\ki^{(1,1)}=\kg^{[2]}+\R\Xab$. The ideals $\ki^{k,l}$ are then ordered by lexicographical order, i.e. 
$$\ki^{1,0}\supseteq \ki^{1,1} \supseteq \ki^{2,0}   \supseteq \ki^{2,1}   \supseteq \dots \supseteq \ki^{k,0} \supseteq \ki^{k,1}\supseteq\dots \supseteq \ki^{k,k-1}\supseteq \dots \supseteq \ki^{s+1,0}=\{0\} $$
and for each each $(k,l)$ above, of successor $(k,l)_{+}$, one may choose a supplementary subspace $\kr^{k,l}$ such that 
$$\ki^{k,l} = \kr^{k,l} \oplus \ki^{(k,l)_{+}}  $$
in particular $\kg=\oplus  \kr^{k,l}$. Raugi's system of dilation is the family $R_{t}$ of linear maps defined on $\kr^{k,l}$ by $R_{t}x= t^{k+l}x$ except on  $\kr^{1,1}$ where  $R_{t}x= t x$. The \emph{main problem} with Raugi's definition is that the weight of the dilation may be strictly subadditive when taking brackets: $[\kr^{k_{1},l_{1}}, \kr^{k_{2},l_{2}}]$ is not necessarily included in $\oplus\{ \kr^{k,l} \,:\, k+l\geq k_{1}+l_{1}+k_{2}+l_{2}\}$. This leaves room for escape of mass, contradicting \cite{raugi78}. We now formulate a precise counterexample  where this phenomenon occurs. 

\bigskip

\noindent{\bf Counterexample to \cite{raugi78}}.
Let $\kl$ be the step-3 filiform algebra defined by $\kl= \oplus_{i=1}^4 \R e_{i}$ where $[e_{1}, e_{2}]=e_{3}$, $[e_{1}, e_{3}]= e_{4}$, $[e_{2}, e_{3}]=0$, $e_{4}\in \mathfrak{z}(\kl)$. In particular, 
$$\kl = \R e_{1} \oplus  \R e_{2} \oplus  \R [e_{1}, e_{2}]   \oplus  \R [e_{1}, [e_{1}, e_{2}]]  .$$
Let $\kl'$ be a nilpotent Lie algebra of step $4$,  and $(e'_{j})_{1\leq j\leq d'}$ a basis adapted to its central descending series, in particular $e_{d'}\in \kl'^{[4]}$. Endow the vector space  $\kg:=\kl\oplus \kl'$ with the bracket $[.,.]$ defined by $[\kl, \kl']=0$, extending the bracket on $\kl'$, and whose restriction to $\kl$ corresponds to the original bracket on $\kl$ modified by imposing $[e_{2}, e_{3}]=e'_{d}$.  Direct computation shows that $(\kg, [.,.])$ is indeed a Lie algebra (of step $4$). Now let $\mu$ be a compactly supported non-lattice probability measure on $\kg$ with expectation $\E(\mu_{ab})=e_{2} \mod [\kg, \kg]$ in the abelianization. Notice $\ki^{3,2} =\ki^{4,0}=\kl'^{[4]}$  and $\ki^{4,1}=\{0\}$. Given an arbitrary choice of $(\kr^{k,l})_{k,l}$, Raugi  renormalises the projection of $\mu^{*n}*-N\Xm$ to $\kr^{4,0}=\kl'^{[4]}$ by  $N^{-4/2}$, but  \Cref{CLT+BE} (and even Crepel-Raugi \cite{crepel-raugi78} in this particular case) predicts that this projection is spread at scale $N^{5/2}$, whence the  escape of mass.

\subsection{Asymptotically close measures} 

We deduce from our central limit theorem a necessary and sufficient condition for two measures on $\kg$ to be asymptotically close, namely \Cref{asymp-close}. We keep the notations from the beginning of \Cref{Sec-global-thm} and start with a variant of the CLT whose recentering is not multiplicative, but additive. 

\begin{lemme}[CLT with additive recentering] \label{CLT-additive} As $N$ goes to $+\infty$, we have 
$$\DilsN (\mu^{*N}- N\Xm) \longrightarrow (\nu*'\chi)- \chi .$$
\end{lemme}

We emphasize that $\DilsN (\mu^{*N}- N\Xm)$ differs from the measure $\DilsN (\mu^{*N}* - N\Xm)$ that we studied until now. Indeed,  $\DilsN (\mu^{*N} - N\Xm)$ refers to the image of $\mu^{*N}$ by the map $\kg\rightarrow \kg, x\mapsto \DilsN(x - N\Xm)$ (see conventions in \Cref{notation-bias}), so here the recentering refers to the addition on $\kg$ seen as a vector space, and not the Lie product $*$. 
In a similar way,  the measure in the limit is the image of $\nu$ by the map $\tkg\rightarrow \tkg, x\mapsto (x*'\chi)- \chi$ which induces a diffeomorphism of $\kg$. In particular, it is also a smooth probability measure on $\kg$.

\begin{proof}
We observe that 
$\mu^{*N}- N\Xm = \tmu^{*N}- N\chi$, so we just need to check that $\DilsN   \tmu^{*N} \rightarrow \nu*'\chi$. 
By the graded replacement lemma (\Cref{grading-Fourier}),  we may replace $\DilsN   \tmu^{*N}$ by $\DilsN   \tmu^{*'N} $, so it boils down to showing the convergence $\DilsN   (\tmu^{*'N} *'-N\chi) \rightarrow \nu$.  Applying \Cref{grading-Fourier} again, we are left to check that
$\DilsN   (\tmu^{*N} *-N\chi) \rightarrow \nu$.  As $\tmu^{*N} *-N\chi= \mu^{*N} *-N\Xm$, this follows from the central limit theorem \ref{CLT+BE}. 

\end{proof}

\begin{proof}[Proof of \Cref{asymp-close}]

If two measures are asymptotically close, then so are their projections to the abelianization. Combining the  classical CLT and the fact that the class of convex subsets is invariant by dilations and translations, one  deduces that those projections must have same mean and covariance matrix. In particular,  the weight filtrations for both measures  coincide and we may thus choose a common family of dilations and recentering for both measures. Using again the invariance properties of $\mathcal{C}$,   the limiting distributions in  \Cref{CLT-additive} must coincide, hence the corresponding measures $\nu$ as well, or more generally the underlying diffusion laws $\sigma$. Examining infinitesimal generators, we conclude that the commutator means, i.e. the $B$ term in $\LL$, also coincide.

The converse statement  is a direct consequence of the fact that  $\mathcal{C}$ is closed under dilations and translations, \Cref{CLT-additive}, and the fact that weak convergence of measures toward an absolutely continuous distribution implies uniform convergence on the class $\mathcal{C}$ \cite[Theorem 4.2]{ranga-rao62}.
\end{proof}

\subsection{Convergence of processes}\label{donsker-sec}

In this subsection, we prove \Cref{Donsker-nilpotent}, namely the analog of Donsker's invariance principle \cite[Theorem XIII.1.9]{revuz-yor99} for the $\mu$-walk on $(\kg, *)$. For $N\geq 1$, $t \ge0$,  we recall the interpolated random process  $W^{(N)}: \R^+\rightarrow \kg$ defined by 
$$W^{(N)}(t)= \DilsN \big(X_{1}*\dots *X_{\lfloor tN \rfloor} *(tN-\lfloor tN \rfloor)X_{\lfloor tN \rfloor+1}* - tN \Xm\big)$$
where the $X_{i}$'s are i.i.d. with law $\mu$, and $X=\E(\mu^{(1)})$ is the lift to $\km^{(1)}$ of $\Xab=\E(\mu_{ab})$.  
We also recall  that we have associated to $\mu$ a $*'$-left-invariant diffusion process on $(\kg,*')$, namely the time-dependent diffusion $(W(t))_{t \ge 0}$  with infinitesimal generator $\LL_t(\mu,(\km^{(b)})_b)$ and law $\sigma$  defined in  \eqref{limit-diff-gen2}. 

Our goal is to show that $W^{(N)}$ converges in law to $W$ in the topology of $\alpha$-H\"older convergence on bounded intervals. 
We begin by rewriting the process as:
\begin{align} \label{rewriting-WN}
W^{(N)}(t) =\DilsN \big(x_{1}*\dots *x_{\lfloor tN \rfloor}*(tN-\lfloor tN \rfloor)x_{\lfloor tN \rfloor+1} *- tN \chi\big)
\end{align}
where $x_{i}=X_{i}-\Xm+\chi$ is now a family of i.i.d. variables with law $\tmu$. Then we set 
$$S^{(N)}(t) = \DilsN \big(x_{1}*'\dots *'x_{\lfloor tN \rfloor}*'(tN-\lfloor tN \rfloor)x_{\lfloor tN \rfloor+1}\big) .$$
From  \Cref{grading-uniform} and the remark following it,  \Cref{Donsker-nilpotent} will reduce to the following. We recall that $Z(t)=W(t)*'t\chi$. 
\begin{proposition} \label{Donsker-tilde} Assume that $\mu$ has finite $m$-th moment for any $m>0$.  Given any  $\alpha \in [0,\frac{1}{2})$,   the process $S^{(N)}$ converges in law towards the process $Z$ in $C^{0,\alpha}(\R_+, \tkg)$ endowed  with the topology of $\alpha$-H\"older convergence on bounded intervals. 
\end{proposition}

 We shall use a classical approach, via Kolmogorov's tightness criterion  \cite[Chapter VIII]{revuz-yor99}. This method reduces the proof of \Cref{Donsker-tilde} to the combination of the next two lemmas.

\begin{lemme}[Convergence of finite-dimensional distributions]\label{fd-conv}
For every $k\geq 1$, $0\leq t_{1}<\dots <t_{k}<\infty$, one has the convergence in law, as $N\to +\infty$:
$$(S^{(N)}(t_{i}))_{i\leq k}  \longrightarrow (Z(t_{i}))_{i\leq k}. $$
\end{lemme}

\begin{proof}
For $N\geq 1$, denote by $t^{(N)}_{i} =\inf \{k/N : k/N\geq t_{i}\}$. Then, for all $\eps>0$,  
$$\mathbb{P}(\|S^{(N)}(t_{i}) -S^{(N)}(t^{(N)}_{i})\|>\eps)=o(1)\,\,\,\,\,\,\,\,\,\,\,\,\,\,\, \mathbb{P}(\| \Z (t_{i}) -\Z (t^{(N)}_{i})\|>\eps)=o(1)$$
so it is sufficient to prove that for any $F : \tkg^k\rightarrow \R$ continuous and bounded, 
$$(S^{(N)}(t^{(N)}_{i}))_{i\leq k}(F)  - (\Z(t^{(N)}_{i}))_{i\leq k}(F)=o_{F}(1) .$$
But, setting $s^{(N)}_{i}= t^{(N)}_{i}-t^{(N)}_{i-1}$ (where $t^{(N)}_{0}=0$), one may rewrite
$$(S^{(N)}(t^{(N)}_{i}))_{i\leq k}= \Psi( \mathcal{S}^{(N)}_{1}, \dots,  \mathcal{S}^{(N)}_{k} ) \,\,\,\,\,\,\,\,\,\,\,\,\,\,\,  (\Z(t^{(N)}_{i}))_{i\leq k}= \Psi( \mathcal{Z}^{(N)}_{1}, \dots,  \mathcal{Z}^{(N)}_{k} )$$
where $\Psi : \tkg^k\rightarrow \tkg^k, (x_{1}, \dots, x_{k}) \mapsto(x_{1}*'\dots *'x_{i})_{i\leq k}$ is a continuous map,  
 the variables $\mathcal{S}^{(N)}_{1}, \mathcal{Z}^{(N)}_{1}, \dots, \mathcal{S}^{(N)}_{k}, \mathcal{Z}^{(N)}_{k}$ are mutually independent, with  each  $\mathcal{S}^{(N)}_{i}$ (resp. $\mathcal{Z}^{(N)}_{i}$)  distributed as $ S^{(N)}(s^{(N)}_{i})$ (resp. $\Z(s^{(N)}_{i})$). To conclude, it is now sufficient to check that $\mathcal{S}^{(N)}_{i} \rightarrow \mathcal{Z}^{(N)}_{i}$ for each $i$, but this follows from the central limit theorem \ref{CLT+BE}.
\end{proof}

\begin{lemme}[Tightness]\label{tightness} Let $d'$ be a left-invariant Riemannian metric on $(\kg, *')$.
For every segment $[0, R]$, every integer $m\ge1$ there is $C_m>0$ such that for all $N \ge 1$,
\begin{align}  \label{variation}
\sup_{0 \le s<t \le R}\E\left(d'(S^{(N)}_{s},S^{(N)}_{t})^{2m} \right)\leq C_m |s-t|^{m}.
\end{align}
\end{lemme}

\begin{proof} 
Suppose $s=\frac{i}{N}$, $t=\frac{i+j}{N}$ for integers $i,j\ge 0$. Then 
\begin{align*}
d'(S^{(N)}_{s},S^{(N)}_{t}) &= d'(0,-S^{(N)}_{s}*'S^{(N)}_{t}) = d'(0, \DilsN(x_{i+1}*'\dots *'x_{i+j}))\\
&\ll_{d'} \|\DilsN(x_{i+1}*'\dots *'x_{i+j})\| .
\end{align*}
For any $m\geq 1$, \Cref{momentPi} and the remark following it tells us that 
$$ \mathbb{E}\left(\|\DilsN(x_{i+1}*'\dots *'x_{i+j})\|^{2m} \right) \ll_m \left(\frac{j}{N}\right)^{m} = |t-s|^m,$$ 
as desired. In general with $i:=\lfloor Ns \rfloor$, $i+j=\lfloor Nt \rfloor$, we may write 
$$d'(S^{(N)}_s,S^{(N)}_t)^{2m} \ll_m d'(S^{(N)}_s,S^{(N)}_{i/N})^{2m} + d'(S^{(N)}_{i/N},S^{(N)}_{(i+j)/N})^{2m}+ d'(S^{(N)}_{(i+j)/N},S^{(N)}_{t})^{2m}.$$
Taking expectations, this becomes $O_{d',m}(1) |t-s|^m$, as we may see after distinguishing the cases $|t-s|<1/N$ and $|t-s|\ge1/N$.
\end{proof}

\begin{proof}[Proof of \Cref{Donsker-tilde}] According to Kolmogorov's tightness criterion and the proof of Theorem 2.1 in \cite{revuz-yor99} (using the Riemannian metric $d'$ in place of the Euclidean metric), \Cref{tightness} implies that the sequence of processes $S^{(N)}$ is relatively compact in the topology of $C^{0,\alpha}(\R_+,\tkg)$ for any $\alpha < \frac{m-1}{2m}$. Together with \Cref{fd-conv} this ensures the desired convergence in law to $\lambda$.
\end{proof}

\begin{proof}[Proof of \Cref{Donsker-nilpotent}] Let $W'^{(N)}$ be the process obtained from the formula \eqref{rewriting-WN} for $W^{(N)}$ by replacing the  product $*$ by the product   $*'$. From \Cref{grading-uniform} we see that the sequence of processes $W^{(N)}$ and $W'^{(N)}$ have the same limit in the $\alpha$-H\"older topology for each $\alpha \in [0,\frac{1}{2})$. 

On the other hand, $W'^{(N)}(t)=S^{(N)}(t)*'-t\chi$, and the operator on $C^{0, \alpha}(\R, \tkg)$ defined by $c(t)\mapsto c(t)*'t\chi$ is continuous. Hence it follows from \Cref{Donsker-tilde} that $W'^{(N)}$ converges in law to the random process $(Z(t)*'-t\chi)_{t\geq 0}$, which is just $(W(t))_{t \geq0}$ by definition.

\end{proof}

\bigskip

\noindent\emph{Remark}. When $\alpha=0$, \Cref{Donsker-tilde} can also be derived from the main result of \cite{stroock-varadhan73}.

\bigskip

\noindent\emph{Remark}. (Rough paths) We note in passing that, in fact, a stronger version of Proposition \ref{Donsker-tilde} holds. Namely, for a suitable choice of dilations $\DilsN$ (i.e. of subspaces $(\tkm^{(b)})_b$), the convergence also holds in the finer \emph{rough paths} topology, see \cite{lyons-qian, friz-victoir10, friz-hairer}, namely in the Polish space of $\alpha$-H\"older paths with respect to the $*'$-invariant Carnot-Carath\'eodory metric $d'_{cc}(x,y)$  defined on $(\tkg,*')$ by any Euclidean norm on $\tkm^{(1)}\oplus\R\chi$. This is defined exactly as the Riemannian metric, except that we only allow paths that are tangent to the distribution of planes formed by the left translates of  $\tkm^{(1)}\oplus\R\chi$. If $\tkm^{(b)}$ is chosen so that it lies in the span $V_b$ of all brackets of length at most $b$ in $\tkm^{(1)}\oplus\R\chi$ (which is possible thanks to \Cref{prehypo}), then for small $x$ (see ``ball-box theorem'', \cite[Thm 7.3.4]{bellaiche96}, \cite[Thm 0.5.A, p97]{gromov96}) $$d'_{cc}(0,x) \simeq \max_{b\ge 0} d_{eucl}(x,V_b)^{1/(b+1)} \ll \max_{b \ge 1} \|x^{(b)}\|^{1/b},$$
where $V_0:=0$, while for large $x$, $d'_{cc}(0,x)$ is controlled by $d_{eucl}(0,x)$, where $d_{eucl}$ is the Euclidean distance on $\tkg$. The same proof then shows \Cref{tightness} with $d'_{cc}$ in place of $d'$. As a consequence the sample paths of the diffusion $\lambda$ are $\alpha$-H\"older for the Carnot-Carath\'eodory metric $d'_{cc}$ for each $\alpha \in [0,\frac{1}{2})$. In the centered case on the free nilpotent Lie group, Donsker's theorem in the rough paths topology was first obtained in \cite{breuillard-friz-huesmann} by a similar method.



\addtocontents{toc}{\protect\setcounter{tocdepth}{0}}

\bigskip
\bigskip

\noindent \emph{Acknowledgments}. We warmly thank Yves Benoist, Peter Friz and Laurent Saloff-Coste for very helpful discussions, as well as the anonymous referee for their useful comments. The first author acknowledges support from the European Research Council (ERC) under the European Union’s Horizon 2020 research and
innovation programme (grant agreement No. 803711).

\bigskip

\bigskip

\noindent\textsc{LAGA – Institut Galilée, 99 avenue Jean Baptiste Clément, 93430 Villetaneuse, France}

\noindent\textit{Email address}: \texttt{benard@math.univ-paris13.fr },

\bigskip

\noindent\textsc{Mathematical Institute, University of Oxford, Woodstock Rd Oxford OX2 6GG, United Kingdom}

\noindent\textit{Email address}: \texttt{breuillard@maths.ox.ac.uk}

\end{document}